\documentclass[12pt,final]{amsart}

\usepackage{euscript,nicefrac}
\usepackage{xspace}
\usepackage{amscd}
\usepackage[notref,notcite]{showkeys}
\usepackage{latexsym,todonotes}
\usepackage{xr,xr-hyper}
\usepackage{epsfig,pxfonts}
\usepackage{expl3, xparse}
\usepackage{ifpdf,fp}

\ifpdf
  \usepackage{pdfsync,hyperref}
\fi
\usepackage[margin=3cm,footskip=25pt,headheight=21pt]{geometry}
\usepackage{fancyhdr}
\DeclareMathAlphabet{\mathbbm}{U}{bbm}{m}{n}
\DeclareMathAlphabet{\mathscr}{U}{mathrsfs}{m}{n}
\newtheorem{theorem}{Theorem}[section]
\newtheorem{lemma}[theorem]{Lemma}
\newtheorem{proposition}[theorem]{Proposition}
\newtheorem{corollary}[theorem]{Corollary}

\newtheorem{itheorem}{Theorem}

{
\theoremstyle{plain}
\newtheorem{definition}[theorem]{Definition}
\newtheorem{example}[theorem]{Example}

\newtheorem{remark}[theorem]{Remark}

}

\renewcommand{\theequation}{\arabic{section}.\arabic{equation}}
\newcounter{subeqn}
\renewcommand{\thesubeqn}{\theequation\alph{subeqn}}
\newcommand{\subeqn}{%
  \refstepcounter{subeqn}
  \tag{\thesubeqn}
}
\makeatletter
\@addtoreset{subeqn}{equation}
\newcommand{\newseq}{%
  \refstepcounter{equation}
}

\newcommand{\nc}{\newcommand}
\newcommand{\doubletilde}[1]{
  \tilde{{\tilde{#1}}}}
\newcommand{\ttalg}{{\it \doubletilde{T}}\xspace}
\nc{\PR}{D^i\!R}
\nc{\rola}{X}
\nc{\wela}{Y}

\nc{\cat}{\mathcal{V}}
\nc{\func}{\EuScript{T}}
\nc{\res}{\operatorname{res}}

\newcommand{\id}{\operatorname{id}}

\renewcommand{\dim}{\operatorname{dim}}

\nc{\bQ}{\mathbb{N}}
\nc{\op}{\operatorname{op}}

\newcommand*\circled[1]{\tikz[baseline=(char.base)]{
            \node[shape=circle,draw,inner sep=2pt] (char) {#1};}}

\newcommand{\Bi}{\mathbf{i}}
\newcommand{\Bj}{\mathbf{j}}

\newcommand{\By}{\mathbf{y}}

\nc{\dF}{\mathsf{F}}
\nc{\dE}{\mathsf{E}}
\nc{\lle}{\Gamma}

\nc{\slehat}{\mathfrak{\widehat{sl}}_e}
\nc{\sllhat}{\mathfrak{\widehat{sl}}_\ell}
\nc{\glehat}{\mathfrak{\widehat{gl}}_e}
\nc{\slnhat}{\mathfrak{\widehat{sl}}_n}
\nc{\glnhat}{\mathfrak{\widehat{gl}}_n}
\nc{\eE}{\EuScript{E}}
\newcommand{\arxiv}[1]{\href{http://arxiv.org/abs/#1}{\tt arXiv:\nolinkurl{#1}}}
\nc{\ep}{\epsilon}
\nc{\RHom}{\mathbb{R}\operatorname{Hom}}
\nc{\eF}{\EuScript{F}}
\nc{\fF}{\mathfrak{F}}
\nc{\fE}{\mathfrak{E}}
\nc{\fI}{\mathfrak{I}}
\nc{\ssy}{\mathsf{y}}

\nc{\fp}{\mathfrak{p}}
\nc{\fq}{\mathfrak{q}}

\newcommand{\K}{\mathbbm{k}}

\newcommand{\Z}{\mathbb{Z}}
\newcommand{\Q}{\mathbb{Q}}
\nc{\Qlb}{\mathbb{\bar \Q}_\ell}
\nc{\Fq}{\mathbb{F}_q}
\nc{\Fqb}{\mathbb{\bar F}_q}

\nc{\walg}{W}
\newcommand{\om}{\omega}

\newcommand{\Bq}{\mathbf{q}}
\newcommand{\N}{\mathbb{N}}
\newcommand{\R}{\mathbb{R}}

\newcommand{\C}{\mathbb{C}}

\nc{\KZ}{\mathsf{KZ}}

\newcommand{\la}{\leftarrow}

\newcommand{\sS}{\mathsf{S}} 
\newcommand{\sR}{\mathsf{R}}
 
\newcommand{\sT}{\mathsf{T}}

\nc{\Bv}{\mathbf{v}}
  \nc{\Bw}{\mathbf{w}}
  \nc{\Bp}{\mathbf{p}}
  \nc{\Bb}{\mathbf{b}}
\nc{\tU}{\mathcal{U}}
\nc{\Bu}{\mathbf{u}}
 \nc{\Fl}{\mathscr{F}\!\ell}
\nc{\Tr}{\operatorname{Tr}}
\nc{\cata}{\mathfrak{V}}
\nc{\tcat}{\tilde{\mathcal{V}}}
\nc{\tcata}{\tilde{\mathfrak{V}}}
\nc{\sheafK}{\EuScript{K}}
\nc{\bmu}{{\boldsymbol{\mu}}}
\nc{\bpi}{\boldsymbol{\pi}}
\nc{\dwalg}{\mathbb{W}}
\nc{\dalg}{\mathbb{T}}
\nc{\aalg}{\mathbb{A}}
\nc{\alm}{\mathscr{A}}
\nc{\bra}{\mathfrak{B}}
\nc{\bO}{\mathbb{O}}

\nc{\Kos}{\EuScript{K}}
\nc{\tilt}{\EuScript{T}}

\renewcommand{\la}{\lambda}

\newcommand{\al}{\alpha}
\newcommand{\be}{\beta}

\renewcommand{\d}{\delta}

\newcommand{\Hom}{\operatorname{Hom}}

\newcommand{\cO}{\mathcal{O}}

\nc{\WB}{\EuScript{LB}}
\nc{\tWB}{\widetilde{\WB}}

\newcommand{\Ext}{\operatorname{Ext}}
\newcommand{\Tor}{\operatorname{Tor}}

\newcommand{\excise}[1]{}

\newcommand{\End}{\operatorname{End}}

\newcommand{\Ind}{\operatorname{Ind}}

\newcommand{\fg}{\mathfrak{g}}
\newcommand{\fl}{\mathfrak{l}}
\newcommand{\mmod}{\operatorname{-mod}}

\newcommand{\alg}{T}

\newcommand{\bla}{{\underline{\boldsymbol{\la}}}}

\newcommand{\Lotimes}{\overset{L}{\otimes}}
\newcommand{\bentodo}{\todo[inline,color=blue!20]}

\newcommand{\thetitle}{Categorified skew Howe duality and comparison of knot homologies}
\nc{\iwedge}[1]{\bigwedge\nolimits^{\! #1}}
\nc{\wedgep}[1]{\iwedge{#1}\C^n}
\nc{\wedgepq}[1]{\iwedge{#1}_q\C^n_q}
\nc{\fsl}{\mathfrak{sl}}
\nc{\sln}{\mathfrak{sl}_n}
\nc{\gln}{\mathfrak{gl}_n}

\nc{\wF}{\mathsf{f}}
\nc{\wE}{\mathsf{e}}

\usetikzlibrary{external} 
\usetikzlibrary{decorations.pathreplacing,backgrounds,decorations.markings,calc,
shapes.geometric}
\tikzset{wei/.style={draw=red,double=red!40!white,double distance=1.5pt,thin}}
\tikzset{bdot/.style={fill,circle,color=blue,inner sep=3pt,outer sep=0}}
\tikzset{dir/.style={postaction={decorate,decoration={markings,
    mark=at position .8 with {\arrow[scale=1.3]{<}}}}}}
\tikzset{rdir/.style={postaction={decorate,decoration={markings,
    mark=at position .8 with {\arrow[scale=1.3]{>}}}}}}
\tikzset{edir/.style={postaction={decorate,decoration={markings,
    mark=at position .2 with {\arrow[scale=1.3]{<}}}}}}
\tikzset{external/export=false}

\begin{document}

\begin{center}
  \noindent {\Large \bf \thetitle}
  \bigskip\\

  \begin{tabular}{c@{\hspace{15mm}}c}
    {\sc\large Marco Mackaay}
    &{\sc\large Ben Webster}\\
    \it Center for
    Mathematical Analysis, &   \it Department of Mathematics,\\
    \it Geometry and Dynamical Systems,&  \it 
                                         University of Virginia\\ 
    \it Instituto
    Superior T\'{e}cnico \& \\
    \it Mathematics Department,\\ \it Universidade do Algarve\\
    \email{mmackaay@ualg.pt}&\email{bwebster@virginia.edu}
  \end{tabular}
  \vspace{3mm}
\end{center}

{\small
\begin{quote}
\noindent {\em Abstract.}
In this paper, we show an isomorphism of homological knot invariants
categorifying the Reshetikhin-Turaev invariants for
$\mathfrak{sl}_n$. Over the past decade, such invariants have been
constructed in a variety of different ways, using matrix
factorizations, category $\cO$, affine Grassmannians, and diagrammatic
categorifications of tensor products. 

While the definitions of these theories are quite different, there is
a key commonality between them which makes it possible to prove that
they are all isomorphic: they arise from a skew Howe dual action of
$\mathfrak{gl}_\ell$ for some $\ell$.  In this paper, we show that the
construction of knot homology based on categorifying tensor products
(from earlier work of the second author) fits into this framework, and
thus agrees with other such homologies, such as Khovanov-Rozansky
homology.  We accomplish this by categorifying the action of
$\mathfrak{gl}_\ell\times \mathfrak{gl}_n$ on
$\bigwedge\nolimits^{\!p}(\C^\ell\otimes \C^n)$ using diagrammatic
bimodules.  In this action, the functors corresponding to
$\mathfrak{gl}_\ell$ and $\mathfrak{gl}_n$ are quite different in
nature, but they will switch roles under Koszul duality.
\end{quote}
}
\bigskip

\section{Introduction}
\label{sec:introduction}

The field of knot homology has seen remarkable development over the
past 15 years, since Khovanov's groundbreaking categorification of the
Jones polynomial \cite{Kho00}. This paper, and the important work
that came after it, led naturally to the question of which
 other knot invariants had similar categorifications.  One natural
generalization is provided by the Reshetikhin-Turaev invariants of
type A quantum groups.  These seem to be in an accessible middle
ground, with a considerably more complex structure than Khovanov
homology, but using more familiar tools and technology than the
categorifications of general Reshetikhin-Turaev invariants defined by
the second author \cite{Webmerged}.  In fact, the program of
categorifying these invariants has been arguably {\it too} successful,
in that a great variety of approaches led to defining different
homology theories, but it has proved easier to give definitions than to
prove equivalence with the other definitions.
\begin{itemize}
\item One theory was defined using matrix factorizations by Khovanov and Rozansky \cite{KR04}
  for the vector representation, with an extension to arbitrary
  fundamentals by Wu \cite{Wu} and Yonezawa \cite{Yon}.
\item One theory was defined using foams by \cite{MSVsln} for the vector representation,
  building on work in the case of $\mathfrak{sl}_3$ by Khovanov
  \cite{Kh-sl3}; this was shown to match Khovanov-Rozansky.  Much later, a version of foams using only local
  relations and allowing the use of arbitrary fundamentals was given
  by Queffelec and Rose \cite{QR}.
\item One theory was defined using the categories of coherent sheaves on convolution
  varieties for the affine Grassmannian by Cautis and Kamnitzer
  \cite{CK,CKII}.
\item One theory was defined using category $\cO$ in type A by
  Mazorchuk-Stroppel \cite{MS09} for the vector representation and by
  Sussan for arbitrary fundamentals \cite{Sussan2007}.  
\item One theory $\mathcal{K}_n$ was defined using diagrammatic categorifications of
  tensor products by the second author \cite{Webmerged}
\end{itemize}
Luckily, there has been excellent progress on resolving this
quandary in recent years.  Cautis (building on his previous work with
Kamnitzer) gave a "universal" construction of type A knot homologies
based purely on the structure of the action of the categorified quantum group
$\mathfrak{sl}_\infty$ on a categorification of its Fock space
representation.  \excise{While this seems like it is making the problem
worse, it actually holds the key to resolving it:} 
All the
disparate theories mentioned above can be rephrased in terms of
such a categorical action of $\mathfrak{sl}_\infty$.   

Our aim in this paper is to explain how this is done for the last theory
mentioned.  This gives us the result:
\begin{itheorem}\label{thmA}
  The knot homology $\mathcal{K}_n$ for framed links labeled with
  representations of $\fsl_n$ coincides with that defined by
  Cautis in \cite[9.3]{Cauclasp}, and thus with all other invariants
  mentioned above.
\end{itheorem}
We show this theorem by understanding the relationship between
Cautis's construction (and thus other constructions based on skew Howe
duality, such as the recent work of Queffelec and Rose \cite{QR}) and
the approach of \cite{Webmerged}, based on the direct categorification
of tensor products and the maps between them.  In the latter paper,
the second author shows that given a sequence $p_i\in [1,n]$, we can
define a categorical representation of $\mathfrak{sl}_n$ sending the
weight $\mu$ to the finite dimensional representations of an algebra
$T^{p_1,\dots, p_\ell}_\mu$ with the Grothendieck group of the sum being $\wedgep{p_1}\otimes \cdots \otimes
\wedgep{p_\ell}$.  

The skew Howe duality approach requires thinking
about several of these tensor products at once.   Its most basic
observation is
that the vector space $\bigwedge\nolimits^{\!p}(\C^\ell\otimes \C^n)$
has a natural $\fsl_\ell\times \fsl_n$-action, and as an $\fsl_n$-module, we
have an isomorphism
\[\bigoplus_{p_1+\cdots+p_\ell=p}\wedgep{p_1}\otimes \cdots \otimes
\wedgep{p_\ell}\cong \bigwedge\nolimits^{\!p}(\C^\ell\otimes \C^n).\] 
Defining the $\mathfrak{sl}_\ell$-action is highly non-trivial, and
requires making one very serious sacrifice: we can no longer consider just
the abelian category of $T^{p_1,\dots, p_\ell}_\mu$-modules, but have to work 
with its bounded derived category.

\begin{itheorem}\label{dualaction}
  There is a categorical $\fsl_\ell\times \fsl_n$-action which associates $D^b(T^{p_1,\dots, p_\ell}_\mu\mmod)$ to the weight
  $(\Bp,\mu)$ if $p_1+\cdots+p_\ell=p$ and the zero
  category to all other weights.  This action categorifies the skew
  Howe dual action on $\bigwedge\nolimits^{\!p}(\C^\ell\otimes \C^n)$,
\end{itheorem}
This theorem has an obvious asymmetry: the $\fsl_\ell$ and $\fsl_n$
action are quite different.  For example, the latter action is by exact
functors and the former not. However, there is a non-obvious symmetry
exchanging these two actions: the algebras $T^{p_1,\dots,
  p_\ell}_{(m_1,\dots, m_n)}$ and $T^{m_1,\dots,
  m_n}_{(p_n,\dots, p_1)}$ 
are Koszul dual since it is shown in \cite[Prop. 9.11]{Webmerged} 
that these are equivalent to blocks of parabolic category $\cO$ for $\mathfrak{gl}_{n}$ whose duality is proven by a 
result of Backelin \cite{Back99}. The corresponding
equivalence of derived categories interchanges these two actions.  This
fact will not be surprising to readers familiar with the theory of
category $\cO$, since the equivalence to blocks of $\cO$ sends the $\fsl_\ell$-functors
to Zuckerman functors and the $\fsl_n$-functors to translation
functors.  The interchange of these under Koszul duality is proven
by Ryom-Hansen \cite{RH}.  We hope that the diagrammatic action of
$\fsl_\ell$ we have described will be of some interest to
representation theorists as a ``hands-on'' understanding of Zuckerman functors. 
We also intend it to point the way to understanding foam categories attached to other Lie algebras with well-understood spiders.  The strategy it suggests is to define bimodules over tensor product algebras  in other types attached to diagrams in the spider, and foams which correspond to morphisms between these bimodules (in the derived category), and relations between these foams that match the relations of these morphisms.

 In general, we have striven to write this
paper dealing solely with diagrammatics and not appealing too often to
facts about category $\cO$, but using a couple of facts (the
$t$-exactness of Zuckerman functors in the linear complex
$t$-structure, and the ``Struktursatz'' of Soergel) considerably reduces
the number of equalities we have to check, and we gave in to
temptation on these points.

Now let us summarize the content of the paper.  In Section
\ref{sec:background}, we discuss the necessary background for the
paper from higher representation theory.  Then we turn in Section
\ref{sec:ladder-bimodules} to constructing the bimodules needed for the
$\fsl_\ell$-action, and probing their basic properties, including the
connection to category $\cO$.  In Section \ref{sec:dual-categ-acti},
we construct the action of Theorem \ref{dualaction}, using the results
from category $\cO$ mentioned above to reduce to checking the relations in a
particular representation. In fact, this representation matches that
defined by Khovanov and Lauda on the cohomology of flag varieties.
Finally, in Section \ref{sec:knot-homology}, we apply these results to
knot homology to prove Theorem \ref{thmA}.

\section*{Acknowledgements}

We thank Sabin Cautis, Aaron Lauda, Hoel Queffelec and David Rose for explaining
their work on related topics to us; Mikhail Khovanov for
simultaneously welcoming us at Columbia where this work started; and
Christian Blanchet and Eric Vasserot for facilitating trips to
Paris, where it continued.  MM
is partially funded by FCT/Portugal through the projects
PEst-OE/EEI/LA0009/2013 and EXCL/MAT-GEO/0222/2012. BW is partially
funded by the NSF under Grant DMS-1151473 and by the Alfred P. Sloan
Foundation.

\excise{\begin{itheorem}
  The categories $T^{\mathbf{p}}_{\mu}\mmod$ and
  $T^{\mathbf{p}^!}_{\mu^!}\mmod$ categorifying tensor products over
  $\fsl_n$ and $\fsl_\ell$ are Koszul dual and this duality
  interchanges the standard categorical $\fsl_n$-action with the
  action of Theorem \ref{dualaction}, and similarly for the
  $\fsl_\ell$ actions.
\end{itheorem}}

\section{Background}
\label{sec:background}

\subsection{Skew Howe duality and webs}
\label{sec:skew-howe-duality}

One of the key ingredients in this paper will be a $q$-deformed
version of skew Howe duality, which we discuss in this section. Let
$n\in\Z_{>1}$ be arbitrary but fixed for the rest of this subsection
and $\dot{\mathbf{U}}_q(\mathfrak{gl}_n)$ denote the idempotented version
of the quantum algebra over $\Q(q)$ associated to
$\mathfrak{gl}_n$.  Let
$\C_q^n$ be the $q$-deformation of the vector representation of
$\mathfrak{gl}_n$ and $\wedgepq{a}$ the $a$-fold $q$-deformed wedge
product. Fix another integer $\ell\in\Z_{>1}$. For any $n$-bounded
$\mathfrak{gl}_\ell$ weight $\Bp=(p_1,\ldots,p_\ell)$, i.e.
$p_i\in [0,n]$ for all $i\in [1,\ell]$, we let
$\wedgepq{\Bp}:=\wedgepq{p_1}\otimes \cdots \otimes \wedgepq{p_\ell}.$
Denote the set of all such $\Bp$ by $\Gamma^n_\ell$ and let
$\vert \Bp \vert:=p_1+\ldots +p_\ell$.  By quantum skew Howe duality,
there is a surjective homomorphism
\begin{equation}
\label{eq:qskewhom}
\dot{\mathbf{U}}_q(\mathfrak{gl}_\ell)\to \bigoplus_{\Bp,\Bp'\in\Gamma^n_\ell, \vert\Bp\vert =\vert\Bp'\vert}\Hom_{\dot{\mathbf{U}}_q(\mathfrak{gl}_n)} \left(\wedgepq{\Bp},\wedgepq{\Bp'}\right).
\end{equation}
The kernel of this map is the ideal generated by all $\mathfrak{gl}_\ell$ weights which are not $n$-bounded. The 
quotient of $\dot{\mathbf{U}}_q(\mathfrak{gl}_\ell)$ by this kernel is denoted $\dot{\mathbf{U}}_q^n(\mathfrak{gl}_\ell)$. 
This is all proved in~\cite [Thm. 4.4.1]{CKM}. 
Recall that both algebras above can be seen as categories whose objects 
are the weights, and the homomorphism above as a linear functor. 

The elements of 
$\Hom_{\dot{\mathbf{U}}_q(\mathfrak{gl}_n)} \left(\wedgepq{\Bp},\wedgepq{\Bp'}\right)$
can be represented by $\mathfrak{gl}_{n}$ {\bf webs}, whose definition
we now briefly recall. We only need a restricted class of webs for our
purposes, called {\bf $n$-ladders}.  The basic building blocks for these 
are the webs denoted $Y$ and $Y^*$:
\[\tikz[baseline,very thick]{\draw[wei] (0,0) -- node [at end, below]{$a+b$} (0,-.5) ; \draw[wei] (0,0)-- node [at end, above]{$b$}  (.4,.5);\draw[wei]
  (0,0) --node [at end, above]{$a$} (-.4,.5); \node at (-1.2,0){$Y:=$};}\qquad \tikz[baseline,very thick]{\draw[wei] (0,0) -- node [at end, above]{$a+b$} (0,.5) ; \draw[wei] (0,0)-- node [at end, below]{$b$}  (.4,-.5);\draw[wei]
  (0,0) --node [at end, below]{$a$} (-.4,-.5);\node at (-1.2,0){$Y^\star:=$};}\]
with $a,b,a+b\in \{0,\ldots,n\}$. Let $\Bp\in\Gamma^n_\ell$. We use $Y$ and $Y^*$ to define two basic $n$-ladders with $\ell$ uprights: 
   \begin{center}
\tikz[xscale=.8, yscale=.6]
{
\node at (-2.2,2) {$\wF^{(c)}_i1_{\Bp}:=$};
\node at (-.7,2){$\dots$};
\node at (2.7,2){$\dots$};
\draw[wei] (0,0)--(0,4); \draw[wei] (2,0)--(2,4); \draw[wei] (0,1.5)--(2,2.5);
\node at (0,-0.5) {$p_i$};
\node at (2,-0.5) {$p_{i+1}$};
\node at (1,2.5) {$c$};
\node at (0,4.5) {$p_i-c$};
\node at (2,4.5) {$p_{i+1}+c$};
\node at (5.3,2) {$\wE^{(c)}_i1_{\Bp}:=$};
\draw[wei] (7.5,0)--(7.5,4); \draw[wei] (9.5,0)--(9.5,4); \draw[wei] (7.5,2.5)--(9.5,1.5);
\node at (7.5,-0.5) {$p_i$};
\node at (9.5,-0.5) {$p_{i+1}$};
\node at (8.5,2.5) {$c$};
\node at (7.5,4.5) {$p_i+c$};
\node at (9.5,4.5) {$p_{i+1}-c$};
\node at (6.8,2){$\dots$};
\node at (10.2,2){$\dots$};
}
\end{center}
In the figure we have suppressed $\ell-2$ uprights (vertical strands from bottom to top). 
The subscript $i$ indicates that the rung
is between the $i$th and the $i+1$st upright, with $i\in [1,\ell-1]$. Again, we require all labels to be in the range from $0$ to $n$. 
In practice, strands with label 0 act as a placeholder for when we
keep the number of strands fixed, and for most purposes they can be
ignored. In this way we can recover $Y$ and $Y^*$ as special
cases of $\wE^{(c)}_i1_{\Bp} $ and $\wF^{(c)}_j1_{\Bp'}$ with one of the labels being equal to $0$. 

An arbitrary $n$-ladder with $\ell$ uprights is by definition a labeled
trivalent graph obtained by vertically glueing basic ladders of type
$\wE^{(c)}_i1_{\Bp}$ and $\wF^{(c)}_j1_{\Bp'}$, for any combination of
$i,j\in [1, \ell-1]$, $c\in [0,n]$ and $\Bp,\Bp'\in\Gamma^n_\ell$. 

\begin{definition}
Let $\mathcal{L}ad^n_\ell$ denote the category of all $n$-ladders with $\ell$ uprights. The objects are all $\Bp\in\Gamma^n_\ell$ and $\Hom(\Bp,\Bp')$ is the $Q(q)$-vector space freely generated by the isotopy classes of all $n$-ladders with $\ell$ uprights whose lower and upper boundary are labeled $\Bp$ and $\Bp'$ respectively. Composition is defined by vertically glueing and the unit of $\Bp$, denoted $1_{\Bp}$, is the ladder without rungs whose uprights are labeled $\Bp$.      
\end{definition}

\noindent In our notation, we will usually suppress the labels of the
uprights. They will only be specified when needed.

Cautis, Kamnitzer and Morrison~\cite [Prop. 5.1.2]{CKM} defined an ideal generated by certain relations in $\mathcal{L}ad^n_\ell$, which we denote by $\mathcal{I}^n_\ell$. Their main results (Thm. 3.3.1 and Prop. 5.1.2) can then be paraphrased as follows:
\begin{theorem}[Cautis-Kamnitzer-Morrison]\label{thm:CKMmain}
There exists an equivalence of categories  
$$\mathcal{L}ad^n_\ell/\mathcal{I}^n_\ell \cong \bigoplus_{\Bp,\Bp'\in\Gamma^n_\ell, |\Bp|=|\Bp'|}\Hom_{\dot{\mathbf{U}}_q(\mathfrak{gl}_n)} \left(\wedgepq{\Bp},\wedgepq{\Bp'}\right)
$$
such that composition with the functor in~\eqref{eq:qskewhom} gives an essentially surjective full functor  
$$
\dot{\mathbf{U}}_q(\mathfrak{gl}_\ell)\to \mathcal{L}ad^n_\ell/\mathcal{I}^n_\ell 
$$
defined by $E_i^{(c)}1_{\Bp}\to \wE^{(c)}_i1_{\Bp}$ and $F_i^{(c)}1_{\Bp}\to \wF^{(c)}_i1_{\Bp}$ 
for any $i\in [1,\ell-1]$, $c\in [0,n]$ and $\Bp\in\Gamma^n_\ell$.
\end{theorem}
Khovanov and Lauda~\cite{KLIII} gave a categorification $\tU_\ell$ of  $\dot{\mathbf{U}}_q(\mathfrak{sl}_\ell)$, which can easily be 
extended to $\dot{\mathbf{U}}_q(\mathfrak{gl}_\ell)$ and
$\dot{\mathbf{U}}_q^n(\mathfrak{gl}_\ell)$. In this paper, we give
a categorification $\WB^n_\ell$ of $\mathcal{L}ad^n_\ell/\mathcal{I}^n_\ell$
(introduced in Section \ref{sec:2-categories}).
 In
this categorification we replace:
\begin{itemize}
\item $\Bp$ with the derived category $D^b(T^\Bp\mmod)$ of
  modules over an algebra defined by the second author in
  \cite{Webmerged}.  
\item An $n$-ladder $L\in \Hom_{\mathcal{L}ad^n_\ell}(\Bp,\Bp')$ with
  an associated $T^{\Bp'}\operatorname{-}T^{\Bp}$-bimodule $W_L$.
  Tensor product with this bimodule induces a functor between derived
  categories. The space of 2-morphisms between two ladders is the morphism space between the
  bimodules in the derived category which are $\mathfrak{gl}_n$
  invariant in an appropriate sense. 
\item Each generating relation in $\mathcal{I}^n_\ell$ with a natural
  isomorphism between the associated derived functors. 
\end{itemize}
In Corollary~\ref{cor:catskewfunc}, we give the categorification of the
second functor in Theorem~\ref{thm:CKMmain}, that is, a 2-functor 
$\tU_\ell\to \WB^n_\ell$ 
which is one of the main results of this paper. Our proof of the well-definedness 
of this 2-functor is a bit roundabout: we analyze the structure of $\WB^n_\ell$ enough to define candidate 2-morphisms, and use results from category $\cO$ to show it suffices to check the relations between these 2-morphisms on certain objects.  We then show that the Hom spaces between these objects exactly match Khovanov and Lauda's 2-representation 
of $\tU$ using the cohomology rings of partial flag varieties~\cite{KLIII}. 
The remaining relations, which are required for the well-definedness of our 2-functor, then follow 
from Khovanov and Lauda's work.  

\subsection{Categorifications of tensor products}
\label{sec:categ-tens-prod}

\subsubsection{The 2-category $\tU_n$}
\label{sec:2-category-tu}

We fix an arbitrary $n\in\N_{>1}$ in this section. The object of interest for this subsection is a strict 2-category
$\tU_n$, due to Khovanov and Lauda~\cite{KLIII}. We will give a more compact definition of this category, shown to be
equivalent to that of earlier literature such as \cite{KLIII,CaLa,Webmerged} in a recent paper of
Brundan \cite{Brundandef}.  In order to define it, we will need to
define a class of diagrams.  Consider the set of diagrams in the
horizontal strip $\R\times [0,1]$ composed of embedded oriented
curves, whose endpoints lie on distinct points of $\R\times\{0\}$ and
$\R\times \{1\}$.  At each point, projection to the $y$-axis must
locally be a diffeomorphism, unless at that point it looks like one of
the diagrams:
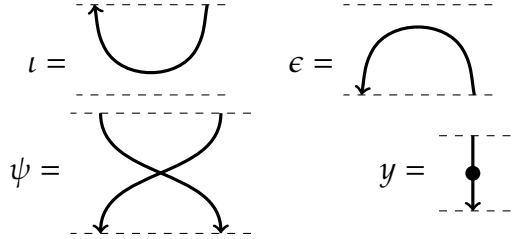
\[
\iota=\tikz[baseline,very thick,scale=3]{\draw[->] (.25,.3) to
  [out=-100, in=60]
(.2,.1)
  to[out=-120,in=-60] (-.2,.1) to [out=120,in=-80] 
(-.25,.3);
  \draw[thin,dashed] (.33,.3) -- (-.33,.3); \draw[thin,dashed] (.33,-.1) -- (-.33,-.1);}
\qquad 
\ep=\tikz[baseline,very thick,scale=3]{\draw[->] (.25,-.1) to
  [out=100, in=-60]   
 (.2,.1)
  to[out=120,in=60]
  (-.2,.1)  to [out=-120,in=80] 
  (-.25,-.1);
\draw[thin,dashed] (.33,-.1) -- (-.33,-.1); \draw[thin,dashed] (.33,.3) -- (-.33,.3); }
\]
\[\psi=
\tikz[baseline=-2pt,very thick,scale=4]{\draw[->] (.2,.2)  to [out=-90,in=90]
(-.2,-.2) ; \draw[<-]
  (.2,-.2) to[out=90,in=-90] 
(-.2,.2); 
\draw[thin,dashed] (.3,-.2) -- (-.3,-.2); \draw[thin,dashed] (.3,.2) -- (-.3,.2);}
\qquad \qquad 
 y=\tikz[baseline=-2pt,very thick,scale=2.5]{\draw[->]
  (0,.2) -- (0,-.2) 
  node[midway,circle,fill=black,inner
  sep=2pt]{}; 
\draw[thin,dashed] (.2,-.2) -- (-.2,-.2); \draw[thin,dashed] (.2,.2) -- (-.2,.2);}\]
 
We'll consider
labelings of the components of these diagrams
by elements of $[1,n-1]$.  The {\bf top} of such a diagram is the sequence
where we read off the label of each of the endpoints on $\R\times \{1\}$ in
order  from left to right, taking them with positive sign if the
curve is oriented upward there, and a negative sign if it is oriented
downward.  The {\bf bottom} is defined similarly with the endpoints on
$\R\times \{0\}$.  The {\bf vertical composition} $ab$ of two diagrams where
the bottom of $a$ matches the top of $b$ is the stacking of $a$ on top
of $b$ and then scaling the $y$-coordinate by $1/2$ to lie again in
$\R\times [0,1]$.  The horizontal composition of two diagrams $a\circ
b$ places $a$ to the {\it right} of $b$ in the plane, and thus has the
effect of concatenating their tops and bottoms in the opposite of the usual order.

We think of elements of $\Z^n$ as weights of $\mathfrak{gl}_n$ in the
usual way, and let $\al_i=-\al_{-i}=(0,\dots, 0,1,-1,0,\dots,0)$ with
the non-zero entries in the $i$th and $i+1$st positions.  We let
$\la^i=\al_i^\vee(\la)=\la_i-\la_{i-1}$ for any weight $\la$.
\begin{definition}
  Let $\doubletilde{\tU}_n$ be the strict 2-category where
 \begin{itemize}
  \item the set of objects is $\Z^n$
\item 1-morphisms $\mu\to \nu$ are sequences
  $\Bi=(i_1,\dots, i_m)$ with each $i_j \in \pm
  [1,n-1]$, which we interpret as a list of simple roots and their
  negatives such that $\mu +\sum_{j=1}^m\al_{i_j}=\nu$.  Composition
  is given by concatenation.
\item 2-morphisms $h\to h'$ between sequences are $\K$-linear combinations
  of diagrams of the type defined above with $h$ as bottom
  and $h'$ as top.    
\end{itemize}
\end{definition}
Since the underlying objects in $\doubletilde{\tU}_n$ are fixed for
any 2-morphism, we incorporate them into the diagram by labeling each
region of the place with $\mu$ at the far left, $\nu$ at the far
right, and intermediate regions are labeled by the rule \[  \tikz[baseline,very thick]{
\draw[postaction={decorate,decoration={markings,
    mark=at position .5 with {\arrow[scale=1.3]{<}}}}] (0,-.5) -- node[below,at start]{$i$}  (0,.5);
\node at (-1,0) {$\mu$};
\node at (1,.05) {$\mu-\al_i$};.
}.\]
We'll typically use $\eE_i$ to denote the 1-morphism $(i)$ (leaving
the labeling of regions implicit) and $\eF_i$ to denote $(-i)$.

We can define a {\bf degree} function on diagrams.  The degrees are
given on elementary diagrams by \[
  \deg\tikz[baseline,very thick,scale=1.5]{\draw[->] (.2,.3) --
    (-.2,-.1) node[at end,below, scale=.8]{$i$}; \draw[<-] (.2,-.1) --
    (-.2,.3) node[at start,below,scale=.8]{$j$};}
  =
  \begin{cases}
    -2 & i=j\\
    1 & |i-j|=1\\
    0 & |i-j|>1\\
  \end{cases}
\qquad \deg\tikz[baseline,very
  thick,->,scale=1.5]{\draw (0,.3) -- (0,-.1) node[at
    end,below,scale=.8]{$i$} node[midway,circle,fill=black,inner
    sep=2pt]{};}=2 \]
  \[
  \deg\tikz[baseline,very thick,scale=1.5]{\draw[->] (.2,.1)
    to[out=-120,in=-60] node[at end,above left,scale=.8]{$i$} (-.2,.1)
    ;\node[scale=.8] at (0,.3){$\la$};} =\langle\la,\al_i\rangle-1
  \qquad \deg\tikz[baseline,very
  thick,scale=1.5]{\draw[->] (.2,.1) to[out=120,in=60] node[at
    end,below left,scale=.8]{$i$} (-.2,.1);\node[scale=.8] at
    (0,-.1){$\la$};} =-\langle\la,\al_i\rangle-1.
  \]
For a general diagram, we sum together the degrees of the elementary
diagrams it is constructed from.  This defines a grading on the 2-morphism spaces of $\doubletilde{\tU}$.
\begin{definition}
Let $\tU_n$ be the quotient of $\doubletilde{\tU}_n$ by the following
relations on 2-morphisms:
\begin{itemize}
\item $\ep$ and $\iota$  are the units and counits of an adjunction,
  i.e. critical points can cancel.

\item the endomorphisms of words only using $\eF_i$ (or by duality only $\eE_i$'s) satisfy the relations of the {\bf quiver Hecke algebra} $R$.\newseq

\begin{equation*}\subeqn\label{first-QH}
    \begin{tikzpicture}[scale=.7,baseline]
      \draw[very thick,postaction={decorate,decoration={markings,
    mark=at position .2 with {\arrow[scale=1.3]{<}}}}](-4,0) +(-1,-1) -- +(1,1) node[below,at start]
      {$i$}; \draw[very thick,postaction={decorate,decoration={markings,
    mark=at position .2 with {\arrow[scale=1.3]{<}}}}](-4,0) +(1,-1) -- +(-1,1) node[below,at
      start] {$j$}; \fill (-4.5,.5) circle (3pt);
      \node at (-2,0){=}; \draw[very thick,postaction={decorate,decoration={markings,
    mark=at position .8 with {\arrow[scale=1.3]{<}}}}](0,0) +(-1,-1) -- +(1,1)
      node[below,at start] {$i$}; \draw[very thick,postaction={decorate,decoration={markings,
    mark=at position .8 with {\arrow[scale=1.3]{<}}}}](0,0) +(1,-1) --
      +(-1,1) node[below,at start] {$j$}; \fill (.5,-.5) circle (3pt);
   \end{tikzpicture}
 \qquad 
    \begin{tikzpicture}[scale=.7,baseline]
      \draw[very thick,postaction={decorate,decoration={markings,
    mark=at position .2 with {\arrow[scale=1.3]{<}}}}](-4,0) +(-1,-1) -- +(1,1) node[below,at start]
      {$i$}; \draw[very thick,postaction={decorate,decoration={markings,
    mark=at position .2 with {\arrow[scale=1.3]{<}}}}](-4,0) +(1,-1) -- +(-1,1) node[below,at
      start] {$j$}; \fill (-3.5,.5) circle (3pt);
      \node at (-2,0){=}; \draw[very thick,postaction={decorate,decoration={markings,
    mark=at position .8 with {\arrow[scale=1.3]{<}}}}](0,0) +(-1,-1) -- +(1,1)
      node[below,at start] {$i$}; \draw[very thick,postaction={decorate,decoration={markings,
    mark=at position .8 with {\arrow[scale=1.3]{<}}}}](0,0) +(1,-1) --
      +(-1,1) node[below,at start] {$j$}; \fill (-.5,-.5) circle (3pt);
      \node at (3.5,0){unless $i=j$};
    \end{tikzpicture}
  \end{equation*}
\begin{equation*}\subeqn\label{nilHecke-1}
    \begin{tikzpicture}[scale=.8,baseline]
      \draw[very thick,postaction={decorate,decoration={markings,
    mark=at position .2 with {\arrow[scale=1.3]{<}}}}](-4,0) +(-1,-1) -- +(1,1) node[below,at start]
      {$i$}; \draw[very thick,postaction={decorate,decoration={markings,
    mark=at position .2 with {\arrow[scale=1.3]{<}}}}](-4,0) +(1,-1) -- +(-1,1) node[below,at
      start] {$i$}; \fill (-4.5,.5) circle (3pt);
      \node at (-2,0){$-$}; \draw[very thick,postaction={decorate,decoration={markings,
    mark=at position .8 with {\arrow[scale=1.3]{<}}}}](0,0) +(-1,-1) -- +(1,1)
      node[below,at start] {$i$}; \draw[very thick,postaction={decorate,decoration={markings,
    mark=at position .8 with {\arrow[scale=1.3]{<}}}}](0,0) +(1,-1) --
      +(-1,1) node[below,at start] {$i$}; \fill (.5,-.5) circle (3pt);
      \node at (1.8,0){$=$}; 
    \end{tikzpicture}\,\,
    \begin{tikzpicture}[scale=.8,baseline]
      \draw[very thick,postaction={decorate,decoration={markings,
    mark=at position .8 with {\arrow[scale=1.3]{<}}}}](-4,0) +(-1,-1) -- +(1,1) node[below,at start]
      {$i$}; \draw[very thick,postaction={decorate,decoration={markings,
    mark=at position .8 with {\arrow[scale=1.3]{<}}}}](-4,0) +(1,-1) -- +(-1,1) node[below,at
      start] {$i$}; \fill (-4.5,-.5) circle (3pt);
      \node at (-2,0){$-$}; \draw[very thick,postaction={decorate,decoration={markings,
    mark=at position .2 with {\arrow[scale=1.3]{<}}}}](0,0) +(-1,-1) -- +(1,1)
      node[below,at start] {$i$}; \draw[very thick,postaction={decorate,decoration={markings,
    mark=at position .2 with {\arrow[scale=1.3]{<}}}}](0,0) +(1,-1) --
      +(-1,1) node[below,at start] {$i$}; \fill (.5,.5) circle (3pt);
      \node at (2,0){$=$}; \draw[very thick,postaction={decorate,decoration={markings,
    mark=at position .5 with {\arrow[scale=1.3]{<}}}}](4,0) +(-1,-1) -- +(-1,1)
      node[below,at start] {$i$}; \draw[very thick,postaction={decorate,decoration={markings,
    mark=at position .5 with {\arrow[scale=1.3]{<}}}}](4,0) +(0,-1) --
      +(0,1) node[below,at start] {$i$};
    \end{tikzpicture}
  \end{equation*}
\begin{equation*}\subeqn\label{black-bigon}
    \begin{tikzpicture}[very thick,scale=.8,baseline]
      \draw[postaction={decorate,decoration={markings,
    mark=at position .55 with {\arrow[scale=1.3]{<}}}}] (-2.8,0) +(0,-1) .. controls (-1.2,0) ..  +(0,1)
      node[below,at start]{$i$}; \draw[postaction={decorate,decoration={markings,
    mark=at position .55 with {\arrow[scale=1.3]{<}}}}] (-1.2,0) +(0,-1) .. controls
      (-2.8,0) ..  +(0,1) node[below,at start]{$j$}; 
   \end{tikzpicture}=\quad
   \begin{cases}
0 & i=j\\
     \begin{tikzpicture}[very thick,yscale=.6,xscale=.8,baseline=-3pt]
       \draw[postaction={decorate,decoration={markings,
    mark=at position .5 with {\arrow[scale=1.3]{<}}}}] (2,0) +(0,-1) -- +(0,1) node[below,at start]{$j$};
       \draw[postaction={decorate,decoration={markings,
    mark=at position .5 with {\arrow[scale=1.3]{<}}}}] (1,0) +(0,-1) -- +(0,1) node[below,at start]{$i$};
     \end{tikzpicture} & i\neq j,j\pm 1\\
   \begin{tikzpicture}[very thick,yscale=.6,xscale=.8,baseline=-3pt]
       \draw[postaction={decorate,decoration={markings,
    mark=at position .8 with {\arrow[scale=1.3]{<}}}}] (2,0) +(0,-1) -- +(0,1) node[below,at start]{$j$};
       \draw[postaction={decorate,decoration={markings,
    mark=at position .8 with {\arrow[scale=1.3]{<}}}}] (1,0) +(0,-1) -- +(0,1) node[below,at start]{$i$};\fill (2,0) circle (4pt);
     \end{tikzpicture}-\begin{tikzpicture}[very thick,yscale=.6,xscale=.8,baseline=-3pt]
       \draw[postaction={decorate,decoration={markings,
    mark=at position .8 with {\arrow[scale=1.3]{<}}}}] (2,0) +(0,-1) -- +(0,1) node[below,at start]{$j$};
       \draw[postaction={decorate,decoration={markings,
    mark=at position .8 with {\arrow[scale=1.3]{<}}}}] (1,0) +(0,-1) -- +(0,1) node[below,at start]{$i$};\fill (1,0) circle (4pt);
     \end{tikzpicture}& i=j-1\\
  \begin{tikzpicture}[very thick,baseline=-3pt,yscale=.6,xscale=.8]
       \draw[postaction={decorate,decoration={markings,
    mark=at position .8 with {\arrow[scale=1.3]{<}}}}] (2,0) +(0,-1) -- +(0,1) node[below,at start]{$j$};
       \draw[postaction={decorate,decoration={markings,
    mark=at position .8 with {\arrow[scale=1.3]{<}}}}] (1,0) +(0,-1) -- +(0,1) node[below,at start]{$i$};\fill (1,0) circle (4pt);
     \end{tikzpicture}-\begin{tikzpicture}[very thick,yscale=.6,xscale=.8,baseline=-3pt]
       \draw[postaction={decorate,decoration={markings,
    mark=at position .8 with {\arrow[scale=1.3]{<}}}}] (2,0) +(0,-1) -- +(0,1) node[below,at start]{$j$};
       \draw[postaction={decorate,decoration={markings,
    mark=at position .8 with {\arrow[scale=1.3]{<}}}}] (1,0) +(0,-1) -- +(0,1) node[below,at start]{$i$};\fill (2,0) circle (4pt);
     \end{tikzpicture}& i=j+1
   \end{cases}
  \end{equation*}
 \begin{equation*}\subeqn\label{triple-dumb}
    \begin{tikzpicture}[very thick,scale=.8,baseline=-3pt]
      \draw[postaction={decorate,decoration={markings,
    mark=at position .2 with {\arrow[scale=1.3]{<}}}}] (-2,0) +(1,-1) -- +(-1,1) node[below,at start]{$k$}; \draw[postaction={decorate,decoration={markings,
    mark=at position .8 with {\arrow[scale=1.3]{<}}}}]
      (-2,0) +(-1,-1) -- +(1,1) node[below,at start]{$i$}; \draw[postaction={decorate,decoration={markings,
    mark=at position .5 with {\arrow[scale=1.3]{<}}}}]
      (-2,0) +(0,-1) .. controls (-3,0) ..  +(0,1) node[below,at
      start]{$j$}; \node at (-.5,0) {$-$}; \draw[postaction={decorate,decoration={markings,
    mark=at position .8 with {\arrow[scale=1.3]{<}}}}] (1,0) +(1,-1) -- +(-1,1)
      node[below,at start]{$k$}; \draw[postaction={decorate,decoration={markings,
    mark=at position .2 with {\arrow[scale=1.3]{<}}}}] (1,0) +(-1,-1) -- +(1,1)
      node[below,at start]{$i$}; \draw[postaction={decorate,decoration={markings,
    mark=at position .5 with {\arrow[scale=1.3]{<}}}}] (1,0) +(0,-1) .. controls
      (2,0) ..  +(0,1) node[below,at start]{$j$}; \end{tikzpicture}=\quad
      \begin{cases} 
    \begin{tikzpicture}[very thick,yscale=.6,xscale=.8,baseline=-3pt]
     \draw[postaction={decorate,decoration={markings,
    mark=at position .5 with {\arrow[scale=1.3]{<}}}}] (6.2,0)
      +(1,-1) -- +(1,1) node[below,at start]{$k$}; \draw[postaction={decorate,decoration={markings,
    mark=at position .5 with {\arrow[scale=1.3]{<}}}}] (6.2,0)
      +(-1,-1) -- +(-1,1) node[below,at start]{$i$}; \draw[postaction={decorate,decoration={markings,
    mark=at position .5 with {\arrow[scale=1.3]{<}}}}] (6.2,0)
      +(0,-1) -- +(0,1) node[below,at
      start]{$j$};     \end{tikzpicture}& i=k=j+1\\
    -\begin{tikzpicture}[very thick,yscale=.6,xscale=.8,baseline]
     \draw[postaction={decorate,decoration={markings,
    mark=at position .5 with {\arrow[scale=1.3]{<}}}}] (6.2,0)
      +(1,-1) -- +(1,1) node[below,at start]{$k$}; \draw[postaction={decorate,decoration={markings,
    mark=at position .5 with {\arrow[scale=1.3]{<}}}}] (6.2,0)
      +(-1,-1) -- +(-1,1) node[below,at start]{$i$}; \draw[postaction={decorate,decoration={markings,
    mark=at position .5 with {\arrow[scale=1.3]{<}}}}] (6.2,0)
      +(0,-1) -- +(0,1) node[below,at
      start]{$j$};     \end{tikzpicture}& i=k=j-1\\
     0& \text{otherwise}
      \end{cases}
  \end{equation*}
\item the composition \[
\tikz[very thick]{
\draw[postaction={decorate,decoration={markings,
    mark=at position .8 with {\arrow[scale=1.3]{>}}}}] (0,0) to[out=90,in=-90] node[below,at start]{$j$} node[above,at end]{$j$} (1,1);\draw[postaction={decorate,decoration={markings,
    mark=at position .8 with {\arrow[scale=1.3]{<}}}},postaction={decorate,decoration={markings,
    mark=at position .2 with {\arrow[scale=1.3]{<}}}}] (-.5,0)
  to[out=90,  in=180] node[below,at start]{$i$} (0,.7) to[out=0, in=180] (1,.3) to[out=0,in=-90] node[above,at end]{$i$} (1.5,1); }
\] possesses an inverse.
\item if $\la^i\geq 0$, then the map   $\sigma_{\la,i}\colon \eE_i\eF_i\to \eF_i \eE_i\oplus
  \id_\la^{\oplus\la^i}$ given by  
\[
\tikz[very thick,scale=1.4]{\node[scale=1.5] at (-1,.5){$\la$};
\draw[postaction={decorate,decoration={markings,
    mark=at position .8 with {\arrow[scale=1.3]{>}}}}] (0,0) to[out=90,in=-90] node[below,at start]{$i$} node[above,at end]{$i$} (1,1);\draw[postaction={decorate,decoration={markings,
    mark=at position .8 with {\arrow[scale=1.3]{<}}}},postaction={decorate,decoration={markings,
    mark=at position .2 with {\arrow[scale=1.3]{<}}}}] (-.5,0)
  to[out=90,  in=180] node[below,at start]{$i$} (0,.7) to[out=0,
  in=180] (1,.3) to[out=0,in=-90] node[above,at end]{$i$} (1.5,1); 
\node at (2,.5){$\oplus$};
\draw[postaction={decorate,decoration={markings,
    mark=at position .8 with {\arrow[scale=1.3]{<}}}},postaction={decorate,decoration={markings,
    mark=at position .3 with {\arrow[scale=1.3]{<}}}}] (2.5,0)
  to[out=90,  in=180] node[below,at start]{$i$} (2.75,.5)
  to[out=0,in=90] node[below,at end]{$i$} (3,0); 
\node at (3.5,.5){$\oplus$};
\draw[postaction={decorate,decoration={markings,
    mark=at position .8 with {\arrow[scale=1.3]{<}}}},postaction={decorate,decoration={markings,
    mark=at position .3 with {\arrow[scale=1.3]{<}}}}] (4,0)
  to[out=90,  in=180] node[below,at start]{$i$} node[at
  end,circle,fill=black,inner sep=2pt]{} (4.25,.5)
  to[out=0,in=90] node[below,at end]{$i$} (4.5,0); 
\node at (5,.5){$\oplus$};
\node at (5.5,.5){$\dots$};
\node at (6,.5){$\oplus$};
\draw[postaction={decorate,decoration={markings,
    mark=at position .8 with {\arrow[scale=1.3]{<}}}},postaction={decorate,decoration={markings,
    mark=at position .3 with {\arrow[scale=1.3]{<}}}}] (6.5,0)
  to[out=90,  in=180] node[below,at start]{$i$} node[at
  end,circle,fill=black,inner sep=2pt]{} (6.75,.5)
  to[out=0,in=90] node[below,at end]{$i$} (7,0);\node at (6.75,.75){$\lambda^i-1$};  }
\] possesses an inverse.
\item if $\la^i\leq 0$, then the map   $\sigma_{\la,i}\colon \eE_i\eF_i\oplus
  \id_\la^{\oplus-\la^i}\to \eF_i \eE_i$ given by  
\[
\tikz[very thick,scale=1.4]{\node[scale=1.5] at (-1,.5){$\la$};
\draw[postaction={decorate,decoration={markings,
    mark=at position .8 with {\arrow[scale=1.3]{>}}}}] (0,0) to[out=90,in=-90] node[below,at start]{$i$} node[above,at end]{$i$} (1,1);\draw[postaction={decorate,decoration={markings,
    mark=at position .8 with {\arrow[scale=1.3]{<}}}},postaction={decorate,decoration={markings,
    mark=at position .2 with {\arrow[scale=1.3]{<}}}}] (-.5,0)
  to[out=90,  in=180] node[below,at start]{$i$} (0,.7) to[out=0,
  in=180] (1,.3) to[out=0,in=-90] node[above,at end]{$i$} (1.5,1); 
\node at (2,.5){$\oplus$};
\draw[postaction={decorate,decoration={markings,
    mark=at position .8 with {\arrow[scale=1.3]{<}}}},postaction={decorate,decoration={markings,
    mark=at position .3 with {\arrow[scale=1.3]{<}}}}] (2.5,1)
  to[out=-90,  in=180] node[above,at start]{$i$} (2.75,.5)
  to[out=0,in=-90] node[above,at end]{$i$} (3,1); 
\node at (3.5,.5){$\oplus$};
\draw[postaction={decorate,decoration={markings,
    mark=at position .8 with {\arrow[scale=1.3]{<}}}},postaction={decorate,decoration={markings,
    mark=at position .3 with {\arrow[scale=1.3]{<}}}}] (4,1)
  to[out=-90,  in=180] node[above,at start]{$i$} node[at
  end,circle,fill=black,inner sep=2pt]{} (4.25,.5)
  to[out=0,in=-90] node[above,at end]{$i$} (4.5,1); 
\node at (5,.5){$\oplus$};
\node at (5.5,.5){$\dots$};
\node at (6,.5){$\oplus$};
\draw[postaction={decorate,decoration={markings,
    mark=at position .8 with {\arrow[scale=1.3]{<}}}},postaction={decorate,decoration={markings,
    mark=at position .3 with {\arrow[scale=1.3]{<}}}}] (6.5,1)
  to[out=-90,  in=180] node[above,at start]{$i$} node[at
  end,circle,fill=black,inner sep=2pt]{} (6.75,.5)
  to[out=0,in=-90] node[above,at end]{$i$} (7,1); \node at (6.75,.25){$-\lambda^i-1$}; }
\] possesses an inverse.
\end{itemize}
\end{definition}
We can define an adjunction of $\eF_i$ and $\eE_i$ in the opposite
order by defining §§
\[\iota'=\tikz[baseline,very thick,scale=3]{\draw[<-]   (.25,.3) to
  [out=-100, in=60] node[at start,above,scale=.8]{$i$} (.2,.1)
  to[out=-120,in=-60] (-.2,.1) to [out=120,in=-80] 
  node[at end,above,scale=.8]{$i$} (-.25,.3);\node[scale=.8] at
  (0,.18){$\la$}; \node[scale=.8] at (0,-.1){$\la+\al_i$};
  \draw[thin,dashed] (.33,.3) -- (-.33,.3);\draw[thin,dashed] (.33,-.2) -- (-.33,-.2);}\qquad\qquad 
\ep'=\tikz[baseline,very thick,scale=3]{\draw[<-] (.25,-.1) to
  [out=100, in=-60]   node[at start,below,scale=.8]{$i$} (.2,.1)
  to[out=120,in=60]
  (-.2,.1)  to [out=-120,in=80] node[at end,below,scale=.8]{$i$} (-.25,-.1);\node[scale=.8] at
  (0,.3){$\la$};\node[scale=.8] at (0,.02){$\la+\al_i$};
  \draw[thin,dashed] (.33,-.1) -- (-.33,-.1); \draw[thin,dashed]
  (.33,.4) -- (-.33,.4);}\] according to the rule of \cite[(1.14-18)]{Brundandef}.

\subsubsection{The algebras $T$ and $\tilde{T}$}
\label{sec:algebras-t-tildet}

We first recall the definition of the $\mathfrak{gl}_n$ tensor
algebras $T$ and $\tilde{T}$ from~\cite{Webmerged}.  In order to prove our results, we have
also needed to add some new material, especially concerning the
relationship between  $T$ and $\tilde{T}$ and Hochschild homology.

\begin{definition}
  A {\bf Stendhal diagram} is a collection of finitely many oriented curves in
  $\R\times [0,1]$. Each curve is either
  \begin{itemize}
  \item colored red and labeled with a dominant weight of $\mathfrak{gl}_n$, or
  \item colored black and labeled with $i\in [1,n-1]$ and decorated with finitely many dots.
  \end{itemize}
 The diagram must be locally of the form 
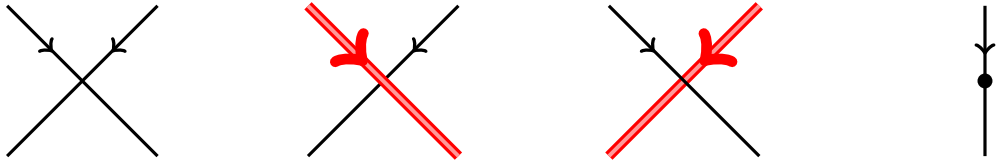
\begin{equation*}
\begin{tikzpicture}
  \draw[very thick,postaction={decorate,decoration={markings,
    mark=at position .75 with {\arrow[scale=1.3]{<}}}}] (-4,0) +(-1,-1) -- +(1,1);
  \draw[very thick,postaction={decorate,decoration={markings,
    mark=at position .75 with {\arrow[scale=1.3]{<}}}}](-4,0) +(1,-1) -- +(-1,1);


  \draw[very thick,postaction={decorate,decoration={markings,
    mark=at position .75 with {\arrow[scale=1.3]{<}}}}](0,0) +(-1,-1) -- +(1,1);
  \draw[wei, very thick,postaction={decorate,decoration={markings,
    mark=at position .75 with {\arrow[scale=1.3]{<}}}}](0,0) +(1,-1) -- +(-1,1);

  \draw[wei,very thick,postaction={decorate,decoration={markings,
    mark=at position .75 with {\arrow[scale=1.3]{<}}}}](4,0) +(-1,-1) -- +(1,1);
  \draw [very thick,postaction={decorate,decoration={markings,
    mark=at position .75 with {\arrow[scale=1.3]{<}}}}](4,0) +(1,-1) -- +(-1,1);

  \draw[very thick,postaction={decorate,decoration={markings,
    mark=at position .75 with {\arrow[scale=1.3]{<}}}}](8,0) +(0,-1) --  node
  [midway,circle,fill=black,inner sep=2pt]{}
  +(0,1);
\end{tikzpicture}
\end{equation*}
with each curve oriented in the negative direction.  In
particular, no red strands can ever cross.  Each curve must
meet both $y=0$ and $y=1$ at distinct points from the other curves.
\end{definition}
We'll typically only consider Stendhal diagrams up to isotopy. Since the orientation on a diagram is clear, we typically won't draw
it.  

We call the lines $y=0,1$ the {\bf  bottom} and {\bf top} of the
diagram.  Reading across the bottom and top from left to right, we
obtain a sequence of dominant weights and elements of $[1,n-1]$. We
record this data as
\begin{itemize}
\item the list $\Bi=(i_1, \dots, i_n)$ of elements of $[1,n-1]$, read
  from the left;
\item the  list $\bla=(\la_1,\dots,\la_\ell)$ of fundamental $\mathfrak{gl}_n$ weights, read
  from the left;
\item  the weakly increasing function
  $\kappa\colon [1,\ell]\to [0,n]$ such that $\kappa(m)$ is the position of
  the rightmost black strand which is left of $m$th red strand (both counted
  from the left). By
convention, we write $\kappa(i)=0$ if the $i$th red strand is left of all
black strands. 
\end{itemize}
We call such a triple of data a {\bf Stendhal triple}.  We will often
want to partition the sequence $\Bi$ in the groups of black strands
between two consecutive reds, that is, the groups \[\Bi_0=(i_1,\dots,
i_{\kappa(1)}), \Bi_1=(i_{\kappa(1)+1},\dots, i_{\kappa(2)}),\dots,
\Bi_\ell=(i_{\kappa(\ell)+1},\dots,i_n).\]  We call these {\bf black blocks}.

  Here are two examples of Stendhal diagrams:
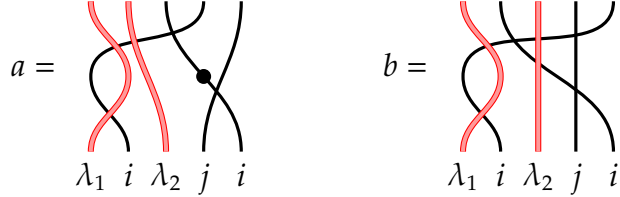
\[a=
\begin{tikzpicture}[baseline,very thick]
  \draw (-.5,-1) to[out=90,in=-90] node[below,at start]{$i$} (-1,0) to[out=90,in=-90](.5,1);
  \draw (.5,-1) to[out=90,in=-90] node[below,at start]{$j$} (1,1);
  \draw  (1,-1) to[out=90,in=-90] node[below,at start]{$i$}
  node[midway,circle,fill=black,inner sep=2pt]{} (0,1);
  \draw[wei] (-1, -1) to[out=90,in=-90]node[below,at start]{$\la_1$} (-.5,0) to[out=90,in=-90] (-1,1);
  \draw[wei] (0,-1) to[out=90,in=-90]node[below,at start]{$\la_2$} (-.5,1);
\end{tikzpicture}\qquad \qquad 
b=
\begin{tikzpicture}[baseline,very thick]
  \draw (-.5,-1) to[out=90,in=-90] node[below,at start]{$i$} (-1,0) to[out=90,in=-90](1,1);
  \draw (.5,-1) to[out=90,in=-90] node[below,at start]{$j$} (.5,1);
  \draw  (1,-1) to[out=90,in=-90] node[below,at start]{$i$} (-.5 ,1);
  \draw[wei] (0,-1) to[out=90,in=-90] node[below,at start]{$\la_2$} (0,1);
  \draw[wei] (-1, -1) to[out=90,in=-90] node[below,at start]{$\la_1$} (-.5,0) to[out=90,in=-90] (-1,1);
\end{tikzpicture}
\]
\begin{itemize}
\item At the top of $a$, we have $\Bi=(i,i,j)$, $\bla=(\la_1,\la_2)$
  and $\kappa=(1\mapsto 0,2\mapsto 0)$.
\item At the top of $b$ and bottom of $a$ and $b$, $\Bi=(i,j,i)$,
  $\bla=(\la_1,\la_2)$ and $\kappa=(1\mapsto 0,2\mapsto 1)$.
\end{itemize}


\begin{definition}
  Given Stendhal diagrams $a$ and $b$, their {\bf composition} $ab$ is
  given by stacking $a$ on top of $b$ and attempting to join the
  bottom of $a$ and top of $b$. If the Stendhal triples
  from the bottom of $a$ and top of $b$ don't match, then the
  composition is not defined and by convention is 0, which is not a
  Stendhal diagram, just a formal symbol.
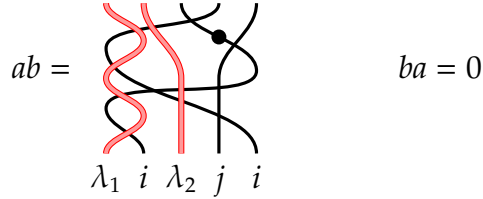
\[
ab=
\begin{tikzpicture}[baseline,very thick,yscale=.5]
  \draw (-.5,-2) to[out=90,in=-90] node[below,at start]{$i$} (-1,-.8)
  to[out=90,in=-90](1,.2) to[out=90,in=-90] node[midway,circle,fill=black,inner sep=2pt]{}  (0,2);
  \draw (.5,-2) to[out=90,in=-90] node[below,at start]{$j$} (.5,0) to[out=90,in=-90] (1,2);
  \draw  (1,-2) to[out=90,in=-90] node[below,at start]{$i$} (-1,.8) to[out=90,in=-90] (.5,2);
  \draw[wei] (0,-2) to[out=90,in=-90] node[below,at start]{$\la_2$}
  (0,0) to[out=90,in=-90]
  (-.5,2);
  \draw[wei] (-1, -2) to[out=90,in=-90] node[below,at start]{$\la_1$}
  (-.5,-1) to[out=90,in=-90] (-1,0) to[out=90,in=-90] (-.5,1)
  to[out=90,in=-90]   (-1,2);
\end{tikzpicture}\qquad \qquad ba=0
\]
Fix a field $\K$ and let $\ttalg$ be the formal span over
$\K$ of Stendhal diagrams (up to isotopy).  The composition law
induces an algebra structure on $\ttalg$.
\end{definition}

 Let $e(\Bi,\bla,\kappa)$ be the unique
crossingless, dotless  diagram where the triple read off from both top and
bottom is $(\Bi,\bla,\kappa)$.
Composition on the left/right with $e(\Bi,\bla,\kappa)$ is an idempotent
operation; it sends a diagram $a$ to itself if the top/bottom of $a$
matches $(\Bi,\bla,\kappa)$ and to 0 otherwise.  We'll often fix $\bla$, and thus leave it out from the notation, just
writing $e(\Bi,\kappa)$ for this diagram.  Let $P^\kappa_\Bi=T^\bla
e(\Bi,\kappa)$ and $\tilde{P}^\kappa_\Bi=\tilde{T}^\bla e(\Bi,\kappa)$.

Considered as
elements of $\ttalg$, the diagrams $e(\Bi,\bla,\kappa)$ are
orthogonal idempotents. 
The algebra $\ttalg$ is not unital, but it is {\bf locally
  unital}, that is for any element $a$, there is an idempotent such
that $ea=ae=a$.  This can be taken to be the sum of $e(\Bi,\bla,\kappa)$ for all
triples that occur at the top or bottom of one of the diagrams in $a$.

Alternatively, we can organize these diagrams into a category whose objects are Stendhal triples
$(\Bi,\bla,\kappa)$ and whose morphisms are Stendhal diagrams read
from bottom to top.  In this perspective, the idempotents
$e(\Bi,\bla,\kappa)$ are the identity morphisms of different objects.

\begin{definition}
  We call a black strand in a Stendhal diagram {\bf violating} if at some
  horizontal slice $y=c$ for $c\in [0,1]$, it is the leftmost strand.  A Stendhal
  diagram which possesses a violating strand is called {\bf violated}.
\end{definition}
Both the diagrams $a$ and $b$ above are violated.  The diagrams
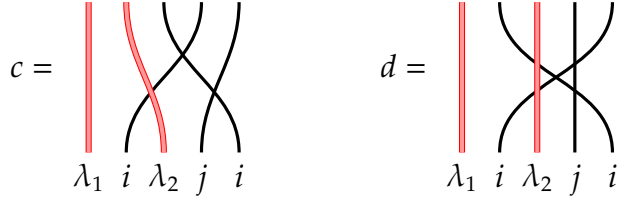
\[c=
\begin{tikzpicture}[baseline,very thick]
  \draw (-.5,-1) to[out=90,in=-90] node[below,at start]{$i$} (.5,1);
  \draw (.5,-1) to[out=90,in=-90] node[below,at start]{$j$} (1,1);
  \draw  (1,-1) to[out=90,in=-90] node[below,at start]{$i$} (0,1);
  \draw[wei] (-1, -1) to[out=90,in=-90]node[below,at start]{$\la_1$} (-1,1);
  \draw[wei] (0,-1) to[out=90,in=-90]node[below,at start]{$\la_2$} (-.5,1);
\end{tikzpicture}\qquad \qquad 
d=
\begin{tikzpicture}[baseline,very thick]
  \draw (-.5,-1) to[out=90,in=-90] node[below,at start]{$i$} (1,1);
  \draw (.5,-1) to[out=90,in=-90] node[below,at start]{$j$} (.5,1);
  \draw  (1,-1) to[out=90,in=-90] node[below,at start]{$i$} (-.5 ,1);
  \draw[wei] (0,-1) to[out=90,in=-90] node[below,at start]{$\la_2$} (0,1);
  \draw[wei] (-1, -1) to[out=90,in=-90] node[below,at start]{$\la_1$} (-1,1);
\end{tikzpicture}
\]
are not violated. The diagram $e({\Bi},\bla,\kappa)$ is violated if
and only if $\kappa(1) >0$.  

\begin{definition}  The {\bf degree} of a Stendhal diagram is the sum over
  crossings and dots in the diagram of 
  \begin{itemize}
  \item$-\langle\al_i,\al_j\rangle$ for each crossing of a black strand
    labeled $i$ with one labeled $j$;
  \item $\langle\al_i,\al_i\rangle=2$ for each dot on a black
    strand labeled $i$;
  \item $\langle\al_i,\la\rangle=\la^i$ for each crossing of a
    black strand labeled $i$ with a red strand labeled $\la$.
  \end{itemize}
The degree of diagrams is additive under composition.  Thus, the
algebra $\ttalg$ inherits a grading from this degree function.
\end{definition}

The reflection through the horizontal axis of a Stendhal diagram $a$
is again a Stendhal diagram, which we denote $\dot{a}$.  Note that
$\dot{(ab)}=\dot b\dot a$, so $\dot{a}$ induces an anti-automorphism of $\ttalg$.

\begin{definition}
  Let $\tilde{T}$ be the quotient  of $\ttalg$ by
  the following local
  relations between Stendhal diagrams:
  \begin{itemize}
  \item the KLR relations  (\ref{first-QH}--\ref{triple-dumb}) 
\item  All black crossings and dots can pass through red lines, with a
  correction term similar to Khovanov and Lauda's (for the latter two
  relations, we also include their mirror images). By convention, the term with the summation below is taken to 
be zero if $\lambda^i=0$ :
\newseq
  \begin{equation*}\subeqn \label{red-triple-correction}
    \begin{tikzpicture}[very thick,baseline]
      \draw (-3,0)  +(1,-1) -- +(-1,1) node[at start,below]{$i$};
      \draw (-3,0) +(-1,-1) -- +(1,1)node [at start,below]{$j$};
      \draw[wei] (-3,0)  +(0,-1) .. controls (-4,0) .. node[below, at start]{$\la$}  +(0,1);
      \node at (-1,0) {=};
      \draw (1,0)  +(1,-1) -- +(-1,1) node[at start,below]{$i$};
      \draw (1,0) +(-1,-1) -- +(1,1) node [at start,below]{$j$};
      \draw[wei] (1,0) +(0,-1) .. controls (2,0) ..  node[below, at start]{$\la$} +(0,1);   
\node at (2.6,0) {$+ $};
      \draw (6.5,0)  +(1,-1) -- +(1,1) node[midway,circle,fill,inner sep=2.5pt,label=right:{$a$}]{} node[at start,below]{$i$};
      \draw (6.5,0) +(-1,-1) -- +(-1,1) node[midway,circle,fill,inner sep=2.5pt,label=left:{$b$}]{} node [at start,below]{$j$};
      \draw[wei] (6.5,0) +(0,-1) -- node[below, at start]{$\la$} +(0,1);
\node at (3.8,-.2){$\displaystyle \delta_{i,j}\sum_{a+b+1=\la^i} $}  ;
 \end{tikzpicture}
  \end{equation*}
\begin{equation*}\subeqn\label{dumb}
    \begin{tikzpicture}[very thick,baseline=2.85cm]
      \draw[wei] (-3,3)  +(1,-1) -- +(-1,1);
      \draw (-3,3)  +(0,-1) .. controls (-4,3) ..  +(0,1);
      \draw (-3,3) +(-1,-1) -- +(1,1);
      \node at (-1,3) {=};
      \draw[wei] (1,3)  +(1,-1) -- +(-1,1);
  \draw (1,3)  +(0,-1) .. controls (2,3) ..  +(0,1);
      \draw (1,3) +(-1,-1) -- +(1,1);    \end{tikzpicture}
  \end{equation*}
\begin{equation*}\subeqn\label{red-dot}
    \begin{tikzpicture}[very thick,baseline]
  \draw(-3,0) +(-1,-1) -- +(1,1);
  \draw[wei](-3,0) +(1,-1) -- +(-1,1);
\fill (-3.5,-.5) circle (3pt);
\node at (-1,0) {=};
 \draw(1,0) +(-1,-1) -- +(1,1);
  \draw[wei](1,0) +(1,-1) -- +(-1,1);
\fill (1.5,.5) circle (3pt);
    \end{tikzpicture}
  \end{equation*}
\item  The ``cost'' of separating a red and a black line is adding $\la^i=\al_i^\vee(\la)$ dots to the black strand.
  \begin{equation}\label{cost}
  \begin{tikzpicture}[very thick,baseline=1.6cm]
    \draw (-2.8,0)  +(0,-1) .. controls (-1.2,0) ..  +(0,1) node[below,at start]{$i$};
       \draw[wei] (-1.2,0)  +(0,-1) .. controls (-2.8,0) ..  +(0,1) node[below,at start]{$\la$};
           \node at (-.3,0) {$=$};
    \draw[wei] (2.8,0)  +(0,-1) -- +(0,1) node[below,at start]{$\la$};
       \draw (1.2,0)  +(0,-1) -- +(0,1) node[below,at start]{$i$};
       \fill (1.2,0) circle (3pt) node[left=3pt]{$\la^i$};
          \draw[wei] (-2.8,3)  +(0,-1) .. controls (-1.2,3) ..  +(0,1) node[below,at start]{$\la$};
  \draw (-1.2,3)  +(0,-1) .. controls (-2.8,3) ..  +(0,1) node[below,at start]{$i$};
           \node at (-.3,3) {$=$};
    \draw (2.8,3)  +(0,-1) -- +(0,1) node[below,at start]{$i$};
       \draw[wei] (1.2,3)  +(0,-1) -- +(0,1) node[below,at start]{$\la$};
       \fill (2.8,3) circle (3pt) node[right=3pt]{$\la^i$};
  \end{tikzpicture}
\end{equation}
  \end{itemize}
\end{definition}

\begin{definition}
  Let $T$ be the quotient of $\tilde{T}$ by the 2-sided ideal $K$
  generated by all violated diagrams.
\end{definition}
 
Now, as before, fix a sequence of dominant weights
$\bla=(\la_1,\dots,\la_\ell)$ and let $\la=\sum_{i=1}^\ell \la_i$.
\begin{definition}
 We let
  $\alg^\bla$  (resp. $\tilde{\alg}^\bla$) be the subalgebra of $\alg$
  (resp. $\tilde{\alg}$) where
  the red lines are labeled, from left to right, with the elements of
  $\bla$.  Let $\alg^\bla_\al$, for any element of the $\gln$ weight lattice $\al$, be the subalgebra of
  $\alg^\bla$ where the sum of the roots associated to the black strands
  is $\la-\al$, and let $\alg^\bla_m$ (resp. $\tilde{\alg}^\bla_m$) be the subalgebra of diagrams
  with $m$ black strands.
\end{definition}

To give a simple illustration of the behavior of our algebra, let us
consider $\gln=\mathfrak{gl}_2$, and $\bla=(\om_1,\om_1)$ where $\om_1=(1,0)$.  Thus, our diagrams have 2 red lines, both labeled with $\om_1$'s. In this case, the algebras $\alg^\bla_\al$ are  easily described as follows:
\begin{itemize}
\item $\alg^{\bla}_{(2,0)}\cong \K$: it is spanned by the diagram
  $\tikz[baseline=-1pt,xscale=.8, yscale=.6]{\draw[wei]
    (0,-.5)--(0,.5); \draw[wei] (.5,-.5)--(.5,.5); }$ .
\item $\alg^{\bla}_{(1,1)}$ is spanned by 
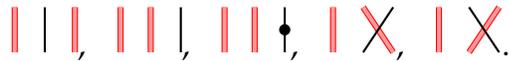
\begin{center}
\tikz[xscale=.8, yscale=.6]{\draw[wei] (0,0)--(0,1); \draw[thick] (.5,0) --(.5,1);\draw[wei] (1,0)--(1,1); },\quad \tikz[xscale=.8,yscale=.6]{\draw[wei] (0,0)--(0,1); \draw[wei] (.5,0) --(.5,1);\draw[thick] (1,0)--(1,1);},\quad \tikz[xscale=.8,yscale=.6]{\draw[wei] (0,0)--(0,1); \draw[wei] (.5,0) --(.5,1);\draw[thick] (1,0)--(1,1) node[midway,fill, circle,inner sep=1.5pt]{};},\quad \tikz[xscale=.8,yscale=.6]{\draw[wei] (0,0)--(0,1); \draw[wei] (1,0) --(.5,1);\draw[thick] (.5,0)--(1,1);},\quad \tikz[xscale=.8,yscale=.6]{\draw[wei] (0,0)--(0,1); \draw[wei] (.5,0) --(1,1);\draw[thick] (1,0)--(.5,1);}.
\end{center}
One can easily check that this is the standard presentation of a regular block of category $\cO$ for $\mathfrak{gl}_2$ as a quotient of the path algebra of a quiver (see, for example, \cite{Str03}).
\item $\alg^{\bla}_{(0,2)}\cong \operatorname{End}(\K^3)$: The algebra
  is spanned by the diagrams
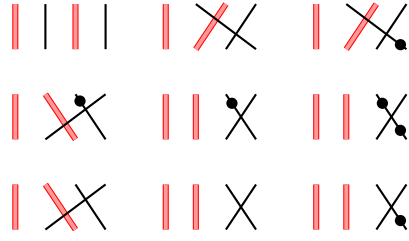
\begin{center}
\begin{tikzpicture}[yscale=1.2,xscale=2]
\node at (0,0){ \tikz[xscale=.8, yscale=.6]{\draw[wei] (0,0)--(0,1); \draw[thick] (.5,0) --(.5,1);\draw[wei] (1,0)--(1,1);\draw[thick] (1.5,0) --(1.5,1); }};

\node at (0,-1) {\tikz[xscale=.8,yscale=.6]{\draw[wei] (0,0)--(0,1); \draw[wei] (1,0) --(.5,1);\draw[thick] (1.5,0)--(1,1) node[pos=.85,fill, circle,inner sep=1.5pt]{}; \draw[thick] (.5,0) --(1.5,1) ;}};

\node at (1,-1) {\tikz[xscale=.8,yscale=.6]{\draw[wei] (0,0)--(0,1); \draw[wei] (.5,0) --(.5,1);\draw[thick] (1.5,0)--(1,1) node[pos=.8,fill, circle,inner sep=1.5pt]{}; \draw[thick] (1,0) --(1.5,1) ;}};
\node at (1,0) {\tikz[xscale=.8,yscale=.6]{\draw[wei] (0,0)--(0,1); \draw[wei] (.5,0) --(1,1);\draw[thick] (1.5,0)--(.5,1); \draw[thick] (1,0) --(1.5,1) ;}};

\node at (0,-2) {\tikz[xscale=.8,yscale=.6]{\draw[wei] (0,0)--(0,1); \draw[wei] (1,0) --(.5,1);\draw[thick] (1.5,0)--(1,1); \draw[thick] (.5,0) --(1.5,1) ;}};

\node at (2,-2) {\tikz[xscale=.8,yscale=.6]{\draw[wei] (0,0)--(0,1); \draw[wei] (.5,0) --(.5,1);\draw[thick] (1.5,0)--(1,1) node[pos=.2,fill, circle,inner sep=1.5pt]{}; \draw[thick] (1,0) --(1.5,1);}};

\node at (2,0) {\tikz[xscale=.8,yscale=.6]{\draw[wei] (0,0)--(0,1); \draw[wei] (.5,0) --(1,1);\draw[thick] (1.5,0)--(.5,1) node[pos=.1,fill, circle,inner sep=1.5pt]{}; \draw[thick] (1,0) --(1.5,1);}};

\node at (2,-1) {\tikz[xscale=.8,yscale=.6]{\draw[wei] (0,0)--(0,1); \draw[wei] (.5,0) --(.5,1);\draw[thick] (1.5,0)--(1,1) node[pos=.8,fill, circle,inner sep=1.5pt]{} node[pos=.2,fill, circle,inner sep=1.5pt]{}; \draw[thick] (1,0) --(1.5,1) ;}};

\node at (1,-2) {\tikz[xscale=.8,yscale=.6]{\draw[wei] (0,0)--(0,1); \draw[wei] (.5,0) --(.5,1); \draw[thick] (1.5,0)--(1,1) ; \draw[thick] (1,0) --(1.5,1);}};

\excise{
\node at (1,2){ \tikz[xscale=.8,yscale=.6]{\draw[wei] (0,0)--(0,1); \draw[wei] (1,0) --(.5,1);\draw[thick] (.5,0)--(1,1);}};

 \tikz[xscale=.8,yscale=.6]{\draw[wei] (0,0)--(0,1); \draw[wei] (.5,0) --(1,1);\draw[thick] (1,0)--(.5,1);}
}
\end{tikzpicture}
\end{center}
\end{itemize}
which one can easily check multiply (up
  to sign) as the elementary generators of $\operatorname{End}(\K^3)$.

Perhaps a more interesting example is the case of
$\mathfrak{gl}_3$ with $\bla=(\om_1,\om_2)$ and $\mu=(1,1,1)$.  Based on the
construction of a cellular basis in \cite{SWschur}, we can calculate
that this algebra is 19 dimensional, with a basis given by 
 
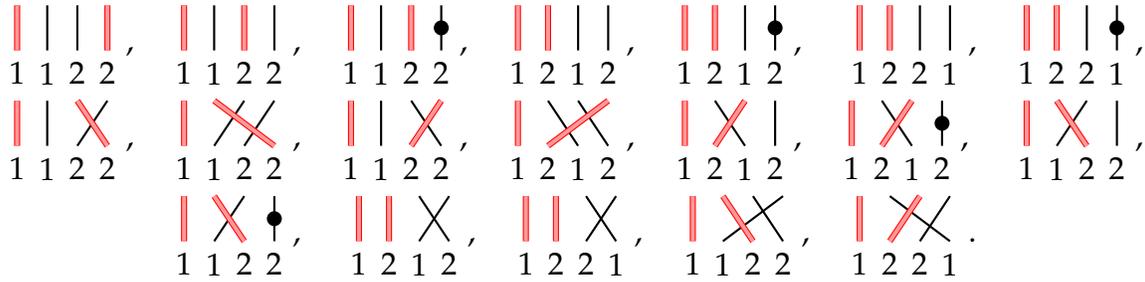
\begin{center}
\tikz[xscale=.8, yscale=.6,baseline]{\draw[wei] (0,0)-- node[below, at
  start]{$1$}(0,1); \draw[thick] (.5,0) -- node[below, at start] {$1$}(.5,1);\draw[thick] (1,0) -- node[below, at start] {$2$} (1,1);\draw[wei] (1.5,0)--node[below, at start]{$2$} (1.5,1); },\quad \tikz[xscale=.8, yscale=.6,baseline]{\draw[wei] (0,0)-- node[below, at
  start]{$1$}(0,1); \draw[thick] (.5,0) -- node[below, at start]
  {$1$}(.5,1);\draw[thick] (1.5,0) -- node[below, at start] {$2$}
  (1.5,1);\draw[wei] (1,0)--node[below, at start]{$2$} (1,1); },\quad \tikz[xscale=.8, yscale=.6,baseline]{\draw[wei] (0,0)-- node[below, at
  start]{$1$}(0,1); \draw[thick] (.5,0) -- node[below, at start]
  {$1$}(.5,1);\draw[thick] (1.5,0) -- node[below, at start] {$2$} node[midway,circle,fill=black,inner sep=2pt] {}
  (1.5,1);\draw[wei] (1,0)--node[below, at start]{$2$}
  (1,1); },\quad \tikz[xscale=.8, yscale=.6,baseline]{\draw[wei] (0,0)-- node[below, at
  start]{$1$}(0,1); \draw[thick] (1,0) -- node[below, at start]
  {$1$}(1,1);\draw[thick] (1.5,0) -- node[below, at start] {$2$}
  (1.5,1);\draw[wei] (.5,0)--node[below, at start]{$2$} (.5,1); },\quad \tikz[xscale=.8, yscale=.6,baseline]{\draw[wei] (0,0)-- node[below, at
  start]{$1$}(0,1); \draw[thick] (1,0) -- node[below, at start]
  {$1$}(1,1);\draw[thick] (1.5,0) -- node[below, at start] {$2$} node[midway,circle,fill=black,inner sep=2pt] {}
  (1.5,1);\draw[wei] (.5,0)--node[below, at start]{$2$}
  (.5,1); }, \quad \tikz[xscale=.8, yscale=.6,baseline]{\draw[wei] (0,0)-- node[below, at
  start]{$1$}(0,1); \draw[thick] (1,0) -- node[below, at start]
  {$2$}(1,1);\draw[thick] (1.5,0) -- node[below, at start] {$1$}
  (1.5,1);\draw[wei] (.5,0)--node[below, at start]{$2$} (.5,1); },\quad \tikz[xscale=.8, yscale=.6,baseline]{\draw[wei] (0,0)-- node[below, at
  start]{$1$}(0,1); \draw[thick] (1,0) -- node[below, at start]
  {$2$}(1,1);\draw[thick] (1.5,0) -- node[below, at start] {$1$} node[midway,circle,fill=black,inner sep=2pt] {}
  (1.5,1);\draw[wei] (.5,0)--node[below, at start]{$2$}
  (.5,1); },

\tikz[xscale=.8, yscale=.6,baseline]{\draw[wei] (0,0)-- node[below, at
  start]{$1$}(0,1); \draw[thick] (.5,0) -- node[below, at start] {$1$}(.5,1);\draw[thick] (1,0) -- node[below, at start] {$2$} (1.5,1);\draw[wei] (1.5,0)--node[below, at start]{$2$} (1,1); },\quad \tikz[xscale=.8, yscale=.6,baseline]{\draw[wei] (0,0)-- node[below, at
  start]{$1$}(0,1); \draw[thick] (.5,0) -- node[below, at start] {$1$}(1,1);\draw[thick] (1,0) -- node[below, at start] {$2$} (1.5,1);\draw[wei] (1.5,0)--node[below, at start]{$2$} (.5,1); },\quad \tikz[xscale=.8, yscale=.6,baseline]{\draw[wei] (0,0)-- node[below, at
  start]{$1$}(0,1); \draw[thick] (.5,0) -- node[below, at start] {$1$}(.5,1);\draw[thick] (1.5,0) -- node[below, at start] {$2$} (1,1);\draw[wei] (1,0)--node[below, at start]{$2$} (1.5,1); },\quad \tikz[xscale=.8, yscale=.6,baseline]{\draw[wei] (0,0)-- node[below, at
  start]{$1$}(0,1); \draw[thick] (1,0) -- node[below, at start] {$1$}(.5,1);\draw[thick] (1.5,0) -- node[below, at start] {$2$} (1,1);\draw[wei] (.5,0)--node[below, at start]{$2$} (1.5,1); },\quad \tikz[xscale=.8, yscale=.6,baseline]{\draw[wei] (0,0)-- node[below, at
  start]{$1$}(0,1); \draw[thick] (1,0) -- node[below, at start]
  {$1$}(.5,1);\draw[thick] (1.5,0) -- node[below, at start] {$2$}
  (1.5,1);\draw[wei] (.5,0)--node[below, at start]{$2$} (1,1); },\quad \tikz[xscale=.8, yscale=.6,baseline]{\draw[wei] (0,0)-- node[below, at
  start]{$1$}(0,1); \draw[thick] (1,0) -- node[below, at start]
  {$1$}(.5,1);\draw[thick] (1.5,0) -- node[below, at start] {$2$} node[midway,circle,fill=black,inner sep=2pt] {}
  (1.5,1);\draw[wei] (.5,0)--node[below, at start]{$2$}
  (1,1); }, \quad \tikz[xscale=.8, yscale=.6,baseline]{\draw[wei] (0,0)-- node[below, at
  start]{$1$}(0,1); \draw[thick] (.5,0) -- node[below, at start]
  {$1$}(1,1);\draw[thick] (1.5,0) -- node[below, at start] {$2$}
  (1.5,1);\draw[wei] (1,0)--node[below, at start]{$2$} (.5,1); },\quad \tikz[xscale=.8, yscale=.6,baseline]{\draw[wei] (0,0)-- node[below, at
  start]{$1$}(0,1); \draw[thick] (.5,0) -- node[below, at start]
  {$1$}(1,1);\draw[thick] (1.5,0) -- node[below, at start] {$2$} node[midway,circle,fill=black,inner sep=2pt] {}
  (1.5,1);\draw[wei] (1,0)--node[below, at start]{$2$}
  (.5,1); }, \quad \tikz[xscale=.8, yscale=.6,baseline]{\draw[wei] (0,0)-- node[below, at
  start]{$1$}(0,1); \draw[thick] (1.5,0) -- node[below, at start]
  {$2$}(1,1);\draw[thick] (1,0) -- node[below, at start] {$1$}
  (1.5,1);\draw[wei] (.5,0)--node[below, at start]{$2$} (.5,1); },\quad  \tikz[xscale=.8, yscale=.6,baseline]{\draw[wei] (0,0)-- node[below, at
  start]{$1$}(0,1); \draw[thick] (1,0) -- node[below, at start]
  {$2$}(1.5,1);\draw[thick] (1.5,0) -- node[below, at start] {$1$}
  (1,1);\draw[wei] (.5,0)--node[below, at start]{$2$} (.5,1); },\quad \tikz[xscale=.8, yscale=.6,baseline]{\draw[wei] (0,0)-- node[below, at
  start]{$1$}(0,1); \draw[thick] (1.5,0) -- node[below, at start]
  {$2$}(1,1);\draw[thick] (.5,0) -- node[below, at start] {$1$}
  (1.5,1);\draw[wei] (1,0)--node[below, at start]{$2$} (.5,1); },\quad  \tikz[xscale=.8, yscale=.6,baseline]{\draw[wei] (0,0)-- node[below, at
  start]{$1$}(0,1); \draw[thick] (1,0) -- node[below, at start]
  {$2$}(1.5,1);\draw[thick] (1.5,0) -- node[below, at start] {$1$}
  (.5,1);\draw[wei] (.5,0)--node[below, at start]{$2$} (1,1); }.
\end{center}
We leave the calculation of the multiplication in this basis to the
reader; it is a useful exercise for those wishing to become more 
comfortable with these sorts of calculations.  For example, when we multiply the last two vectors in
the basis above, we get that (for $Q_{21}(u,v)=u-v$)
\begin{equation*}
  \tikz[xscale=.8, yscale=.6,baseline=4pt]{\draw[wei] (0,0)-- node[below, at
  start]{$1$}(0,1); \draw[thick] (1.5,0) -- node[below, at start]
  {$2$}(1,1);\draw[thick] (.5,0) -- node[below, at start] {$1$}
  (1.5,1);\draw[wei] (1,0)--node[below, at start]{$2$} (.5,1); } \cdot
\tikz[xscale=.8, yscale=.6,baseline=4pt]{\draw[wei] (0,0)-- node[below, at
  start]{$1$}(0,1); \draw[thick] (1,0) -- node[below, at start]
  {$2$}(1.5,1);\draw[thick] (1.5,0) -- node[below, at start] {$1$}
  (.5,1);\draw[wei] (.5,0)--node[below, at start]{$2$} (1,1); } =
\tikz[xscale=.8, yscale=.6,baseline=4pt]{\draw[wei] (0,0) to[out=90,in=-90] node[below, at
  start]{$1$}(0,1); \draw[thick] (1.5,0) to[out=90,in=-90]  node[below, at start]
  {$2$} (1,.5) to[out=90,in=-90] (1.5,1);\draw[thick] (.5,0) to[out=90,in=-90]  node[below, at start] {$1$}
  (1.5,.5) to[out=90,in=-90] (.5,1);\draw[wei] (1,0) to[out=90,in=-90]
  node[below, at start]{$2$} (.5,.5) to[out=90,in=-90] (1,1); }  =\tikz[xscale=.8, yscale=.6,baseline=4pt]{\draw[wei] (0,0) to[out=90,in=-90] node[below, at
  start]{$1$}(0,1); \draw[thick] (1.5,0) to[out=90,in=-90]  node[below, at start]
  {$2$}  node[midway,circle,fill=black,inner sep=2pt] {}(1.5,1);\draw[thick] (.5,0) to[out=90,in=-90]  node[below, at start] {$1$}
  (1,.5) to[out=90,in=-90] (.5,1);\draw[wei] (1,0) to[out=90,in=-90]
  node[below, at start]{$2$} (.5,.5) to[out=90,in=-90] (1,1); } - \tikz[xscale=.8, yscale=.6,baseline=4pt]{\draw[wei] (0,0) to[out=90,in=-90] node[below, at
  start]{$1$}(0,1); \draw[thick] (1.5,0) to[out=90,in=-90]  node[below, at start]
  {$2$}  (1.5,1);\draw[thick] (.5,0) to[out=90,in=-90]  node[below,
at start] {$1$} node[at end,circle,fill=black,inner sep=2pt] {}
  (1,.5) to[out=90,in=-90] (.5,1);\draw[wei] (1,0) to[out=90,in=-90]
  node[below, at start]{$2$} (.5,.5) to[out=90,in=-90] (1,1); }=\tikz[xscale=.8, yscale=.6,baseline=4pt]{\draw[wei] (0,0) to[out=90,in=-90] node[below, at
  start]{$1$}(0,1); \draw[thick] (1.5,0) to[out=90,in=-90]  node[below, at start]
  {$2$}  node[midway,circle,fill=black,inner sep=2pt] {}(1.5,1);\draw[thick] (.5,0) to[out=90,in=-90]  node[below, at start] {$1$}
  (.5,1);\draw[wei] (1,0) to[out=90,in=-90]
  node[below, at start]{$2$} (1,1); } - \tikz[xscale=.8, yscale=.6,baseline=4pt]{\draw[wei] (0,0) to[out=90,in=-90] node[below, at
  start]{$1$} (.5,.5) to[out=90,in=-90] (0,1); \draw[thick] (1.5,0) to[out=90,in=-90]  node[below, at start]
  {$2$}  (1.5,1);\draw[thick] (.5,0) to[out=90,in=-90]  node[below,
at start] {$1$} 
  (0,.5) to[out=90,in=-90] (.5,1);\draw[wei] (1,0) to[out=90,in=-90]
  node[below, at start]{$2$} (1,1); }=\tikz[xscale=.8, yscale=.6,baseline=4pt]{\draw[wei] (0,0) to[out=90,in=-90] node[below, at
  start]{$1$}(0,1); \draw[thick] (1.5,0) to[out=90,in=-90]  node[below, at start]
  {$2$}  node[midway,circle,fill=black,inner sep=2pt] {}(1.5,1);\draw[thick] (.5,0) to[out=90,in=-90]  node[below, at start] {$1$}
  (.5,1);\draw[wei] (1,0) to[out=90,in=-90]
  node[below, at start]{$2$} (1,1); }. 
\end{equation*}
\begin{definition}
  Let $\cata^\bla$ (resp. $\tcata^\bla$) be the abelian category of finitely generated
  graded $T^\bla$ (resp. $\tcata^\bla$) modules.  Let
  $\cat^\bla=D^b(\cata^\bla)$ and  $\tcat^\bla=D^b(\tcata^\bla)$ be
  the bounded derived categories of these categories of modules.
\end{definition}
We let $[n]$ be the homological shift in the derived category, reindexing a complex so that the $k$th term of
$C^\bullet[n]$ is $C^{n+k}$, that is, decreasing the homological
degree of each term by $n$.  

Let $(n)$ be the shift of internal grading on $\cata^\bla$ and its
related categories.  As with the homological shift, this {\it decreases}
the grading by $n$.  

Finally, we let $\langle n\rangle $ be the ``Tate twist'' which is the
composition $[n](-n)$.

\subsubsection{Induction and restriction functors}
\label{sec:induct-restr-funct}

There are maps $\iota^k_L,\iota_R^k\colon \tilde{T}^\bla\to
\tilde{T}^\bla$ which act by adding a black strand at the left
(resp. right) side of the diagram.
\begin{definition}
  We'll let $\tilde{\fF}_k^*,\tilde{\fF}_k\colon
  \tcata^\bla\to\tcata^\bla$ denote the extension of scalars (that is,
  induction) by the maps $\iota^k_L,\iota_R^k$ (respectively).  After
  imposing the violating relation, the map induced by $ \iota^k_L$ is
  0, but we can still consider the map induced by $\iota^k_R$, and let
  ${\fF}_k\colon \cata^\bla\to \cata^\bla$ be the extension of scalars
  along this induced map.  We let $\fE_k\colon \cata^\bla\to
  \cata^\bla$ be the adjoint of $\fF_k$.  
\end{definition}

These functors give the compatibility of the categories $\cata^\bla$
and $\tcata^\bla$ with our constructions from the previous
sections as follows:  the category of $\cata^\bla$ carries a
representation of $\tU_n$ via the functors $\fE_i$ and $\fF_i$, as shown in
\cite[\ref{m-full-action}]{Webmerged}, whereas  $\tcata^\bla$ 
carries an action of two commuting copies of the lower half of this
category $\tU^-_n\times \tU_n^-$ via $\tilde{\fF}_k^*,\tilde{\fF}_k$. 
While these functors possess left and right
adjoints, these adjoints are not isomorphic and do not preserve the
category of finitely generated modules, so there is no hope of
extending this to a $\tU_n\times \tU_n$ action on this category.

For any highest weight $\mu$, we also have a map $\tilde{I}_\mu\colon
\tilde{T}^{\bla}\to \tilde{T}^{(\la_1,\dots,\la_\ell,\mu)}$ which adds
a red line at the far right.  This preserves violated diagrams, and
thus induces a map $I_\mu \colon
{T}^{\bla}\to {T}^{(\la_1,\dots,\la_\ell,\mu)}$.
\begin{definition}\label{I-def}
Let $\tilde{\fI},\fI$ denote extension of scalars along the maps
$\tilde{I}_\mu,I_\mu$.  That is:  \[\tilde{\fI}_{\mu}(M):=
\tilde{T}^{(\la_1,\dots,\la_\ell,\mu)}\otimes_{\tilde{T}^{\bla}}M\qquad
\fI_{\mu}(M):=  {T}^{(\la_1,\dots,\la_\ell,\mu)}\otimes_{T^{\bla}}M.\]  
\end{definition}

These functors have several important homological properties:
\begin{proposition}\label{exactfunctors}
  The functors $\tilde{\fF}_k^*,\tilde{\fF}_k,\fF_k,\fE_k,
  \tilde{\fI},\fI$ are all exact and send projectives to projectives.
\end{proposition}
\begin{proof}
  Tensor with a bimodule will be exact and send projectives to
  projectives if and only if the bimodule is sweet, that is projective
  as a left module and as a right module.  For
  $\tilde{\fF}_k^*,\tilde{\fF}_k,\fF_k,\tilde{\fI},\fI$, the bimodule
  under consideration is of the form ${T}^{\bla}e$ or
  $\tilde{T}^{\bla}e$ for some idempotent $e$, and thus is obviously
  projective as a left module.  
After applying the Morita equivalence of $T^\bla$ with the double
  Stendhal algebra $DT^\bla$ \cite[\ref{m-Morita}]{Webmerged} the bimodule
  $\fE_k$ has the same property.
For
  $\tilde{\fF}_k^*,\tilde{\fF}_k,\tilde{\fI}$, projectivity as a right
  module follows immediately from \cite[Prop.~\ref{m-basis}]{Webmerged}.  For $\fI$, projectivity on the right
  follows from \cite[Prop. \ref{m-FI-IF}]{Webmerged}, since it shows
  that $e(\Bi,\kappa) \tilde{\fI}$ is isomorphic to $e(\Bi,\kappa')T$
  where $\kappa'$ is $\kappa$ restricted to
  $[1,\ell-1]$. Finally, for $\fE_k,\fF_k$, this follows since these functors are biadjoint, and the right
  adjoint of a functor sending projectives to projectives is exact.
\end{proof}

\subsubsection{Derived pushforward and pullback}
\label{sec:deriv-pushf-pullb}

When we have a $T^\bla$-module $M$, we can abuse notation and use $M$
to also denote the pullback $\tilde{T}^\bla$-module. In this case, it is worth noting
that the pullback of ${\fF}_kM$ is not the same as the
$\tilde{T}^\bla$-module $\tilde{\fF}_kM$, but we have a surjective map
$\tilde{\fF}_kM\to {\fF}_kM$.  The difference is that in
$\tilde{\fF}_kM $, we have not
killed the submodule of elements where the new strand we've added passes through 
the region at the far left, whereas in ${\fF}_kM$, we have.

Given a sequence $\Bi$, we have a surjective map 
$\tilde{P}^\kappa_{\Bi}\to P^\kappa_\Bi$, whose kernel is the
submodule generated by the elements $c_m$
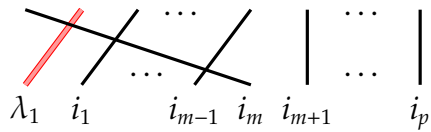
\[\tikz[very thick,xscale=1.5]{\draw[wei] (0,-.5) -- node[at start,below]{$\la_1$}(.5,.5); \draw (.5,-.5)
  -- node[at start,below]{$i_1$} (1,.5);\draw (1.5,-.5)
  -- node[at start,below]{$i_{m-1}$}(2,.5); \draw (2,-.5)
  -- node[at start,below]{$i_m$}(0,.5); \draw (2.5,-.5)
  -- node[at start,below]{$i_{m+1}$}(2.5,.5);\draw (3.5,-.5)
  -- node[at start,below]{$i_{p}$} (3.5,.5); \node at
  (1.1,-.4){$\cdots$};\node at (1.4,.4){$\cdots$};\node at
  (3,-.4){$\cdots$};\node at (3,.4){$\cdots$};}  \]
Let $P^{\kappa}_{\Bi}(s)$ for $s\in  [1,p]$ be the quotient of
$\tilde{P}^\kappa_{\Bi}$ by the submodule generated by $c_m$ for
$m\leq s$; in particular,
$P^{\kappa}_{\Bi}(0)=\tilde{P}^\kappa_{\Bi}$ and
$P^{\kappa}_{\Bi}(p)={P}^\kappa_{\Bi}$.  Note that
\begin{equation}
\tilde{\fF}_i
P^{\kappa}_{\Bi}(s)\cong P^{\kappa}_{\Bi\cup\{i\}}(s)\qquad \fI_\la
P^{\kappa}_{\Bi}(s)\cong P^{\kappa\cup \{\la\}}_{\Bi}(s).\label{eq:3}
\end{equation}
Fix $s\in [1,p]$.  Given a Stendhal triple $(\bla,\Bi,\kappa)$, let $\Bi^-$ be the sequence given by removing
$i_s$, $\kappa^-$ the corresponding function 
$\kappa^-(q)=\kappa(q)$ if $\kappa(q)\leq s$ and
$\kappa^-(q)=\kappa(q)-1$ if $\kappa(q) >s$. 

\begin{lemma}\label{change-S}
  If $s\in S$, then we have a surjective map $P^{\kappa}_{\Bi}(s-1)\to
  P^{\kappa}_{\Bi}(s)$ with kernel isomorphic to
  $\tilde{\fF}^*_{i_s}P^{\kappa}_{\Bi^-}(s-1)$, via the map attaching
  the diagram $c_s$ to the bottom of the diagram.  
\end{lemma}
\begin{proof}
  First note that this attachment map is in fact a map of modules:
  that if we take $c_m$ for $m<s$, add a black strand at the far left, and
  then attach $c_s$ to its bottom, we obtain an element of the
  submodule generated by $c_q$ for $q<s$.  

 In this diagram, we slide the crossing between the $m$th and $s$th
 strands (currently between the $m+1$st and $m+2$nd strands) to the
 far left, so that it is the crossing closest to the top.  In reduced
 words for permutations (reading from bottom), we go from $s_2\cdots
 s_{m+1}s_1\cdots s_s$ to $s_1\cdots s_ss_1\cdots s_m$.

This may have correction terms, but these will all lie in the
submodule generated by $c_q$ for $q<s$, as does this new leading term
diagram.  Thus the map
$\tilde{\fF}^*_{i_s}\tilde{P}^{\kappa^-}_{\Bi^-}\to
\tilde{P}^{\kappa}_{\Bi}$ given by attaching $c_s$ at the bottom
induces a map $\tilde{\fF}^*_{i_s}P^{\kappa}_{\Bi^-}(s-1) \to
P^{\kappa}_{\Bi}(s-1)$ whose image is clearly the submodule generated
by $c_s$, and thus has cokernel $ P^{\kappa}_{\Bi}(s)$.  Thus, we need
only prove that this map is injective.  

Since the functors $\tilde{\fF}_i$ and $\fI_\la$ are both exact, using
\eqref{eq:3}, we
can reduce to the case where $s=p$.  In this case, both
$P^{\kappa}_{\Bi}(p-1)\cong \tilde{\fF}_{i_p}P^{\kappa}_{\Bi^-}(p-1)$ and 
  $\tilde{\fF}^*_{i_p}P^{\kappa}_{\Bi^-}(p-1)$ are free modules over
  $\K[e_1(\By)]$ of rank
  $(p-1+\ell)\dim(P^{\kappa}_{\Bi^-}(p-1))$. Here $e_1(\By)=y_1+\ldots+y_p$, with $y_k$ being given 
by the sum of all $e(\Bj,\kappa')$ with one dot on the $k$th black
strand, where $\Bj$ is a permutation of $\Bi$ and $\kappa'$ is
arbitrary. Note that $p-1+\ell$ is the number of possible endpoints at
the top of the diagram for the $p$th black strand (i.e. the number of
shuffles). The claim follows from the basis
theorem~\cite[\ref{m-basis}]{Webmerged} by the same argument as
in~\cite[Prop. 2.16-8]{KLI}. 

A map between free modules of the same rank of a polynomial ring is
  injective if and only if the cokernel is finite dimensional.
 So the map is injective by the finite dimensionality of $P^{\kappa}_{\Bi}(s)$.
\end{proof}

\begin{lemma}\label{tensor-same}
  We have an isomorphism of bimodules $T^\bla\Lotimes_{\tilde{T}^\bla}T^\bla\cong T^\bla$.
\end{lemma}
\begin{proof}
  This is clear for the naive tensor product, so we only need to check that
  higher Tor's vanish.  That is, we need to check that
  $\Tor^i_{\tilde{T}^\bla}(T^\bla,P^{\kappa}_{\Bi})=0$ for
  $i>0$.  We'll prove the stronger statement that
  $\Tor^i_{\tilde{T}^\bla}(T^\bla,P^{\kappa}_{\Bi}(s))=0$ for all $s$
  by induction on increasing $s$.  Thus, the base case is when $s=0$;
  in this case $P^{\kappa}_{\Bi}(0)$ is projective over
  $\tilde{T}^\bla$, so the statement is clear.  Otherwise, we have a
  long exact sequence on $\Tor$ given by 
\[\cdots \to \Tor^i_{\tilde{T}^\bla}(T^\bla,\tilde{\fF}^*_{i_s}P^{\kappa}_{\Bi^-}(s-1))\to  \Tor^i_{\tilde{T}^\bla}(T^\bla, P^{\kappa}_{\Bi}(s-1))\to
  \Tor^i_{\tilde{T}^\bla}(T^\bla, P^{\kappa}_{\Bi}(s))\to \cdots\] 
By the inductive assumption, the middle term vanishes when $i>0$.  On
the other
hand, for all $i\geq 0$, we have \[\Tor^i_{\tilde{T}^\bla}(T^\bla,\tilde{\fF}^*_{i_s}P^{\kappa}_{\Bi^-}(s-1))\cong
\Tor^i_{\tilde{T}^\bla}(\tilde{\fE}^*_{i_s}T^\bla,P^{\kappa}_{\Bi^-}(s-1))\cong
0\] since $\tilde{\fE}^*_{i_s}T^\bla=0$.

Applied to the long exact sequence, this shows that
$\Tor^i_{\tilde{T}^\bla}(T^\bla, P^{\kappa}_{\Bi}(s))$ for $i>0$ is
trapped between two 0's and thus itself 0.  
\excise{We either have that: (1) $e_{\Bi}T^\bla\cong
\fF_{i_n}(e_{\Bi^-}T^\bla)$ for some Stendal triple $\Bi^-$ or (2) $e_{\Bi}T^\bla\cong \fI_{\lambda_\ell}(e_{\Bi}T^{\bla^-})$.
In the former case, we have an exact sequence \[0\to \tilde{\fF}^*_{i_n}(e_{\Bi^-}T^\bla)\to
\tilde{\fF}_{i_n}(e_{\Bi^-}T^\bla)\to \fF_{i_n}(e_{\Bi^-}T^\bla)\cong e_{\Bi}T^\bla\to 0.\]

We have a vanishing tensor
product $\tilde{\fF}^*_{i_n}(e_{\Bi^-}T^\bla)\Lotimes_{\tilde{T}^\bla} T^\bla
e_{\Bj}=0$
so we have an induced isomorphism $ e_{\Bi}T^\bla \Lotimes_{\tilde{T}^\bla} T^\bla
e_{\Bj}\cong\tilde{\fF}_{i_n}(e_{\Bi^-}T^\bla) \Lotimes_{\tilde{T}^\bla} T^\bla
e_{\Bj}\cong e_{\Bi^-}T^\bla  \Lotimes_{\tilde{T}^\bla}\eE_{i_n}(T^\bla
e_{\Bj})$.  The last term in this isomorphism is a tensor product of
projective modules over a $T^\bla_\mu$ with lower rank, and thus is 0
by inductive assumption.

In the case (2), a similar argument holds reducing to the case of
fewer black strands, though there is no need to
use an exact sequence, since adding a new red line has the same effect
over $T^\bla$ or $\tilde{T}^\bla$.  }
\end{proof}
In particular, this shows that if we have a left $T^\bla$ module and
right $T^\bla$ module, their derived tensor product over $T^\bla$ is
the same as that of their pullbacks over $\tilde{T}^\bla$.

\subsubsection{Grothendieck groups}
\label{sec:grothendieck-groups-1}

We can understand the Grothendieck groups of the categories
$\cata^\Bp$ and $\tcata^\Bp$.  

The vector space $U_q^-\otimes
  V_{\la_1}\otimes \cdots \otimes V_{\la_\ell}$ has left and right
  actions of $U_q^-$  given by
  \begin{align*}
    F_i\cdot (u\otimes w_1\otimes \cdots \otimes w_\ell)&=F_iu\otimes
    w_1\otimes \cdots \otimes w_\ell \\
    (u\otimes w_1\otimes \cdots \otimes w_\ell)\cdot F_i&=uF_i\otimes
    K^{-1}_i(w_1\otimes \cdots \otimes w_\ell)+u\otimes F_i(w_1\otimes
    \cdots \otimes w_\ell)
  \end{align*}
The results 
\cite[\ref{m-Uq-action}, \ref{m-tilde-iso}]{Webmerged} show that:
\begin{proposition}\label{prop:GG-iso}
  We have isomorphisms \[K^0_q(\tcata^\bla)\cong U_q^-\otimes
  V_{\la_1}\otimes \cdots \otimes V_{\la_\ell}\qquad K^0_q(\cata^\bla)\cong 
  V_{\la_1}\otimes \cdots \otimes V_{\la_\ell}.\]
This isomorphism sends 
\[[\tilde{\fF}_i^*](u\otimes w)\mapsto F_i\cdot(u\otimes w) \quad
[\tilde{\fF}_i](u\otimes w)\mapsto (u\otimes w)\cdot
F_i  \]
\[[\tilde{\fI}_\la^*](u\otimes w)\mapsto u_{(1)}\otimes (u_{(2)}v_{\la}\otimes
w) \quad [\tilde{\fI}_\la](u\otimes w)\mapsto  u\otimes w\otimes
v_{\la} \]
\end{proposition}
These are compatible with pullback from $T^\Bp$- to
$\tilde{T}^\Bp$-modules (resp. its adjoint pushforward), in the sense that
these functors categorify the inclusion \[w_1\otimes \cdots \otimes
w_\ell\mapsto 1\otimes  w_1\otimes \cdots \otimes w_\ell\] (resp. its
adjoint projection, induced by the counit map of $U_q^-$).

\subsubsection{Standard modules}
\label{sec:standard-modules}

In this section, we summarize results from Section \ref{m-sec:SS},
specialized to the case where $\fg=\mathfrak{gl}_n$ and all
highest weights are fundamental.  By  \cite[\ref{m-SS}]{Webmerged}, the algebra $T^{\Bp}$ is {\bf
  quasi-hereditary}, that is, the category $T^{\Bp}\mmod$ is {\bf
  highest weight.}

In this case, the labeling set for simple and projective modules is
given by $\ell$-tuples $\bmu=(\mu_1,\dots, \mu_\ell)$ where $\mu_i$ is
a weight of $\wedgep{p_i}$.  We order these elements according to
reduced dominance order
\begin{equation}
(\mu_1,\dots,
\mu_\ell)\geq (\mu_1',\dots,
\mu_\ell') \quad\text{ if }\quad\sum_{i=j}^\ell\mu_i\geq
\sum_{i=j}^\ell\mu_i' \quad\text{ for all $j\in[1,\ell]$}.\label{eq:7}
\end{equation}
Each weight $\mu$ of the representation $\wedgep {p}$ corresponds
to a $p$-element subset $J_{\mu}$ of $[1,n]$, and this weight space is spanned
by the wedge $v_\mu:=v_{j_1}\wedge \cdots \wedge v_{j_p}$ where $j_1<\cdots
<j_p$ are the elements of $J_p$ in order.  We let
$v_{\bmu}=v_{\mu_1}\otimes \cdots \otimes v_{\mu_\ell}$.
We can characterize $P_\bmu$ as the unique indecomposable projective
of the form $[P_\bmu]=v_{\bmu}+\sum_{\bmu'>\bmu}m_i(q) v_{\bmu'}$;  similarly, $L_\bmu$ is the unique simple
such that $[L_\bmu]=v_{\bmu}+\sum_{\bmu'<\bmu}n_i(q) v_{\bmu'}$. 

In brief, this means that $T^{\Bp}\mmod$ contains a distinguished
collection of objects $\nabla_\bmu$ which are ``intermediate'' between
the simples and projectives.  In this case, the labeling is fixed by
the fact that $[\nabla_\bmu]=v_\bmu$.  

The induced structure on modules over $\tilde{T}^\Bp$ is slightly more
complicated.  By  \cite[\ref{m-tilde-SS}]{Webmerged}, that category is
standardly stratified and the standard modules are precisely the
summands of $\fF_{\Bi}^*\nabla_\bmu$ for all different $\Bi$ and
$\bmu$ (where as always, we use the same symbol for a
$T^\Bp$-module and its pullback to $\tilde{T}^\Bp$).

\subsection{\texorpdfstring{$A_\infty$}{A infinity} algebras}
\label{sec:a_infty-algebras}

While we will not use them in an extremely deep way, it will be useful
to us to think about an action of $A$ in the $A_\infty$ sense.  The
basic situation in which this comes up for us is this: assume $M$ is
an object in a $\K$-linear abelian category, and consider a projective
resolution
$\cdots \longrightarrow M_{i-1}\longrightarrow M_i\longrightarrow \cdots \longrightarrow
M_0$, where all indices are non-positive. If a $\K$-algebra $A$ acts on $M$ on the right, then every element $a\in A$
induces an endomorphism of the chain complex $M_*$ which is unique up
to homotopy.  Thus, one can easily see that these endomorphisms will
induce an action of $A$ on this complex if we strictly identify
homotopic maps.  

This sort of crude identification of homotopic maps is impractical for
actually thinking of $M_*$ as an $A$-module.  A context that allows us
more control over the situation is to instead define an action of $A$
on $M_*$ in the $A_\infty$-sense.  That is, we define higher
multiplications
$\alpha_t\colon M_i\otimes A^{\otimes {u-1}}\to M_{i+u}$ for all
$i>0$; by convention $\alpha_1=\partial$. 
Since $A$ is an associative algebra, the relations these must satisfy
take the form:
\begin{multline}
  \label{eq:5}
  \sum_{s+t=u}(-1)^{su}\alpha_{t+1}(\alpha_s(m,a_1,\dots,
  a_{s-1}),a_{s},\dots, a_{u-1}) \\ +\sum_{r=1}^{u-2}(-1)^r\alpha_{u-1}(m,a_1,\dots,a_ra_{r+1},\dots,a_{u-1})=0
\end{multline}
One can prove that such an action always exists since the sum of terms
above when $s>1$ and $t>0$ acts trivially on the module $M$, and thus
is nullhomotopic.  The equation \eqref{eq:5} then says that $\alpha_u$
can be chosen to be any choice of null-homotopy.  
This is one low-tech
way of understanding the fact that the $A_\infty$ operad is a
cofibrant replacement of the associative algebra operad.   

We can always turn a module or complex of modules in the $A_\infty$
sense into an honest module by taking the $A_\infty$ tensor product with
the bimodule $A$ itself. Let $SA:=A[-1]$ denote the suspension of
$A$. For simplicity, let $\mathbf{m}:=m\otimes a_1\otimes
\cdots\otimes a_{u-1}\otimes a_u$.

\begin{definition}\label{Atensor}
  The $A_\infty$ tensor product $M\overset{\infty}\otimes A$ is given
  by the complex $M\otimes \mathsf{T} SA\otimes A$ whose $u$-th term is given by 
$\bigoplus_{i+j=u}(M_i\otimes  (A[-1])^{\otimes j}\otimes A)$ 
with the differential
  \begin{multline}\label{Adiff}
    \partial(\mathbf{m})=
\sum_{s=1}^{u}(-1)^{s(u-s)}\alpha_s(m,a_1,\dots, a_{s-1}) \otimes a_{s}\otimes
    \cdots\otimes a_{u}\\+\sum_{r=1}^{u-1}(-1)^{r} m\otimes a_1\otimes
    \cdots\otimes a_ra_{r+1}\otimes \cdots\otimes a_{u} .
  \end{multline}
\end{definition}
The relations \eqref{eq:5} are equivalent to this differential
squaring to 0.

\begin{lemma}\label{Aequiv}
  The map sending $\mathbf{m}\mapsto (-1)^{u+1}\alpha_{u+1}(m,a_1,\dots, a_{u-1},a_{u})$ gives a
  homotopy equivalence between the complexes $M\overset{\infty}\otimes
  A$ and $M$ with homotopy inverse $m\mapsto m\otimes 1$.
\end{lemma}
\begin{proof}
  To show this, we need to show that there is a homotopy $h$ such that
  \[(\partial\circ h+h\circ \partial)(\mathbf{m})=\mathbf{m}+(-1)^u\alpha_{u+1}(m,a_1,\dots, a_{u-1},a_{u})\otimes 1\]
We claim that this homotopy is given by \[h(\mathbf{m})=(-1)^um\otimes a_1\otimes \cdots \otimes a_{u-1}\otimes
  a_{}\otimes 1,\] since we have 
  \begin{align*}
   (\partial\circ h+h\circ \partial)(\mathbf{m})&=\sum_{s=1}^{u+1}(-1)^{s(u+1-s)+u}\alpha_s(m, a_1,
  \dots,a_{s-1})\otimes a_s\otimes \dots \otimes
  a_{u}\otimes 1\\ &\quad +\sum_{r=1}^{u} (-1)^{r+u} m\otimes a_1\otimes
    \cdots\otimes a_ra_{r+1}\otimes \cdots\otimes a_{u}\otimes
    1\\ &\quad +\sum_{s=1}^u (-1)^{s(u-s)+u-s+1} \alpha_s(m, a_1,
  \dots,a_{s-1})\otimes a_s\otimes \dots \otimes
  a_{u}\otimes 1\\ &\quad +\sum_{r=1}^{u-1} (-1)^{r+u-1} m\otimes a_1\otimes
    \cdots\otimes a_ra_{r+1}\otimes \cdots\otimes a_{u}\otimes
    1\\
&=\mathbf{m}+(-1)^u\alpha_{u+1}(m,a_1,\dots, a_{u-1},a_{u})\otimes 1.\qedhere
  \end{align*}
\end{proof}

\subsection{Hochschild cohomology}
\label{sec:hochsch-cohom}

For any algebra $A$ the Hochschild cohomology is the Ext algebra
$HH^\bullet(A)\cong \Ext^\bullet_{A\otimes A^{\op}}(A,A)$ of the diagonal
bimodule with itself.  It's the ``center'' of this algebra, computed
in the derived category.

For us, the most important thing about Hochschild cohomology is its
connection to deformations.  For a $\K$-algebra $A$, a {\bf
  deformation} of $A$ is a free $\K[h]/(h^2)$-algebra $\tilde{A}$ with a fixed isomorphism $A\cong \tilde{A}/h \tilde{A}$ (up to the
  obvious equivalence). A classical theorem of Hochschild relates
  these to Hochschild cohomology:
\begin{theorem}[\mbox{\cite[6.2]{Hoch}}]
  There is a bijection between deformations of $A$ and $HH^2(A)$.  
\end{theorem}
For a more modern treatment of this theorem, see
\cite[9.3.1]{Weibel}.  The important point for us is that a Hochschild
class $s$ can be represented by a 2-cocycle $r\colon A\otimes A\to A$,
and the associated deformation is the product on $A[h]/(h^2)$ defined by
$a\star_r b= ab+hr(a,b)$.

\excise{\begin{proof}
  Since the form of this equivalence is important for us, let us give
  a brief proof.  We can compute $HH^2(A)$ using the Hochschild
  resolution
\[\cdots A\otimes A\otimes A\otimes A\to A\otimes A\otimes A\to
A\otimes A\to A\to 0\]
In this presentation, an element of $HH^2(A)$ is a map $r\colon A\otimes A\to A$ such that 
\begin{equation}
a_1r(a_2,a_3)-r(a_1a_2,a_3)+r(a_1,a_2a_3)-r(a_1,a_2)a_3=0\label{eq:6}
\end{equation}
modulo
those of the form $r(a_1,a_2)=a_1s(a_2)-s(a_1a_2)+s(a_1)a_2$ for some
$s\colon A\to A$.  

To obtain a deformation from a 2-cocycle $r$, we consider multiplication on $A[h]/(h^2)$ by
$a\star_r b= ab+hr(a,b)$, with the cocycle equation \eqref{eq:6} assuring
associativity.   If two such products are isomorphic, the isomorphism
is of the form $a\mapsto a+hs(a)$, so $r$ and $r'$ differ by a
coboundary.

 On the other hand, assume we have a deformation.  If we fix a vector
space splitting $\phi\colon A\to \tilde{A}$, then we have that
$\phi(a)\phi(b)=\phi(ab)+h\phi(r(a,b))$ for some 2-cocycle $r$.  Thus
$\tilde{A}$ is isomorphic to the deformation attached to this cocycle.
\end{proof}}

The algebras $\tilde{T}^\Bp$ have a natural deformation coming from changing
the relation \eqref{cost};  we consider the span of Stendhal diagrams
over the polynomial ring $\K[z_1,\dots z_\ell]$, and then impose the
relations as before, except replacing \eqref{cost} with the relation   \begin{equation*}\label{d-cost}
  \begin{tikzpicture}[very thick,baseline]
    \draw (-2.8,0)  +(0,-1) .. controls (-1.2,0) ..  +(0,1) node[below,at start]{$i$};
       \draw[wei] (-1.2,0)  +(0,-1) .. controls (-2.8,0) ..  +(0,1) node[below,at start]{$\om_i$};
           \node at (-.3,0) {=};
    \draw[wei] (2.8,0)  +(0,-1) -- +(0,1) node[below,at start]{$\om_i$};
       \draw (1.2,0)  +(0,-1) -- +(0,1) node[below,at start]{$i$};
       \fill (1.2,0) circle (3pt);
 \node at (3.8,0) {$+$};\node at (4.5,0){$z_k$};
        \draw[wei] (6.8,0)  +(0,-1) -- +(0,1) node[below,at start]{$\om_i$};
       \draw (5.2,0)  +(0,-1) -- +(0,1) node[below,at start]{$i$};
  \end{tikzpicture}
\end{equation*}
where the red strand shown is the $k$th from the left.

Since $\tilde{T}$ has a basis given by diagrams in
\cite[\ref{m-basis}]{Webmerged}, we can use this basis to define a
splitting from $\tilde{T}^\Bp$ into its deformation.  
We can then use this splitting to compute a 2-cocycle for this
deformation.  By definition,  we take the product of two diagrams in this basis, and
expanding them in terms of the basis using the deformed relations; the
sum of the terms divisible by $h$ with the $h$ removed is $r$ applied
to these diagrams.

This is a special case of the canonical deformation discussed in
\cite[\S 2.6]{WebwKLR}, since by \cite[\S 3.5]{WebwKLR}, the algebra
$\tilde{T}^{\Bp}$ is a reduced weighted KLR algebra.  The deformation
is induced by deforming the polynomials $Q_{{*,*}}$ from the definition
of the weighted KLR algebra, and thus \cite[2.7]{WebwKLR} shows
that the same set of diagrams provides a basis for the deformed and
undeformed algebras.  The key point is that the polynomial
representation of $\tilde{T}$ can easily be deformed compatibly with
this deformation, which allows to show that these diagrams remain
linearly independent.  

Thus, this deformation defines a map $\mathbf{y}\colon
\K^\ell\to HH^2(\tilde{T}^\Bp)$. 
 We'll denote the images of the coordinate
vectors by $\ssy_k$; these
correspond to the specialization $z_k=h,z_j=0$ for $j\neq k$.  We'll
choose a specific Hochschild cochain $r_k\colon T^\Bp\otimes T^\Bp\to
T^{\Bp}$ which represents this deformation.

Another important issue we'll consider is how these classes act on
$\Ext(M,M)$ for an $A$-module $M$.  Let $A_h$ be a deformation of $A$
over $\K[h]/(h^2)$ and let $s\in HH^2(A)$ be the corresponding Hochschild
class, and $s_M\in \Ext^2(M,M)$ the induced action on $M$.

\begin{proposition}[\mbox{Gerstenhaber and Schack, \cite[\S 3]{GerSch}}]
  The module $M$ has a deformation to a $\K[h]/(h^2)$-free module over $A_h$ if and
  only if $s_M=0$. 
\end{proposition}
\begin{proof}
Let $\varphi\colon A\to \End_{\K}(M)$ be the action map.
We make $\End_{\K}(M)$ into an $A$-$A$-bimodule, with the left action
being precomposition with $\varphi(a)$ and the right action post-composition.
If we freely resolve $M$ using the tensor product with the Hochschild
resolution, we can identify $\Ext^2(M,M)$ with the Hochschild cohomology
$HH^2(A,\End_{\K}(M) )$ of the bimodule $\End_{\K}(M)$.  More
concretely, this group is
the subquotient of
$\Hom_\K(M\otimes A^{\otimes 2},M)\cong \Hom_\K(A^{\otimes 2},\End_{\K}(M))$ of cocycles $\xi$  satisfying
\[\varphi(a_1)\xi(a_2\otimes a_3)-\xi(a_1a_2\otimes
a_3)+\xi(a_1\otimes a_2a_3)-\xi(a_1\otimes a_2
)\varphi(a_3),\] modulo those of the form \[\partial\eta(a_1\otimes
a_2)=\varphi(a_1)\eta(a_2)-\eta(a_1a_2)+\eta(
a_1)\varphi(a_2).\]  If $r$ is a Hochschild 2-cocycle representing $s$, then the induced cocycle
$r_M$ is $r_M(a_1\otimes
a_2)=\varphi(r(a_1,a_2))$.

  The module $M$ has a flat deformation if the map $\varphi\colon A\to
  \End_{\K}(M)$ has a deformation to a map $\varphi_h=\varphi+h\varphi_1\colon A_h\to \End_{\K}(M)[h]/(h^2)$.  The
  argument of \cite{GerSch} shows that this happens when
  $r_M(a_1,a_2)=\partial(
  \varphi_1)$. Thus, we have
 $s_M=0$. 
\end{proof}

In general we can compute the class $s_M$ by studying {\it how} a
module fails to deform flatly.  Let us make this more precise.
Consider a complex 
\[\cdots \overset{{\partial}}\longrightarrow  M_i
\overset{{\partial}}\longrightarrow  M_{i+1}
  \overset{{\partial}}\longrightarrow\cdots
  \overset{{\partial}}\longrightarrow  M_0\] of $A$-modules. Assume there exists a precomplex 
(that is, a sequence of maps where the differential
is not assumed to square to $0$.)
\[\cdots \overset{\tilde{\partial}}\longrightarrow \tilde M_i
\overset{\tilde{\partial}}\longrightarrow \tilde M_{i+1}
  \overset{\tilde{\partial}}\longrightarrow\cdots
  \overset{\tilde{\partial}}\longrightarrow \tilde M_0\] of $\K[h]/(h^2)$-free modules over the deformation $A_h$,
 such that $\tilde M_i/h \tilde M_i\cong M_i$ and $\partial$ is the
 reduction of $\tilde{\partial}$ modulo $h$. For example, if each
 $M_i$ is projective, then such a precomplex always exists. If this
 precomplex is actually a complex, that is,  $\tilde{\partial}^2=0$,
 then this is a deformation in the derived category; if the $M_i$'s
 are a projective resolution of a module $N$, then the cohomology of
 the deformation will be a flat deformation of $N$ and $s_N$ must be 0.

Even if $\tilde{\partial}^2\neq 0$, then we still have that $\tilde{\partial}^2$ is
  equal to 0 $\pmod h$, so we have an induced map
  $\tilde{\partial}^2/h \colon M_i\to M_{i+2}$.  This is a chain map, since \[\partial\circ
  \tilde{\partial}^2/h =\tilde{\partial}^3/h =\tilde{\partial}^2/h
  \circ \partial,\] and we can think of its class up to homotopy as an
  obstruction to deforming the complex.  Since $s_M$ is another such
  obstruction, it seems logical they would agree, as indeed is true:
\begin{lemma}\label{Hochschild-action}
  The class of $s_M\in \Ext^2(M,M)$ agrees with the action induced by 
  $\tilde{\partial}^2/h$ on the complex $M$.
\end{lemma}
\begin{proof}
First, we will prove that this induced map commutes with all chain maps, and
in particular, its action on an object in the derived category is well
defined and independent of the choice of precomplex.

Assume that $M$ and $N$ are two complexes of $A$-modules, and we
deform them to precomplexes $\tilde{M}$ and $\tilde{N}$ with
``differentials'' $\tilde{\partial}$ and $\tilde{\partial}'$.  Assume
that $f\colon M\to N$ is a chain map.  Then,
modulo $h$, we have
\[f\circ \tilde{\partial}^2/h-(\tilde{\partial}')^2/h\circ f={\partial}'\circ
(f\circ \tilde{\partial}-\tilde{\partial}'\circ
f)/h+(f\circ\tilde{\partial}-\tilde{\partial}'\circ f)/h
\circ 
{\partial}\] so $(f\circ\tilde{\partial}-\tilde{\partial}'\circ f)/h$
exhibits a homotopy between the two maps. Most importantly, this shows that the answer
will be the same up to homotopy for any deformation of the same
differential on a single complex.  In fact, this shows that
$\tilde{\partial}^2/h$ must be the action of some element of
Hochschild cohomology.

Thus, we need only show that it agrees with $s_M$ for some
quasi-isomorphic complex and one particular choice of deformed differential. Choose a splitting
$\phi\colon A\to A_h$, and consider the deformed
Hochschild complex whose $k$th term is $A_h\otimes A^{\otimes k}\otimes
M$. The deformed differential on this is given by 
\[\tilde \partial(a'\otimes a_1\otimes \cdots\otimes a_k\otimes
m)=a'\phi(a_1)\otimes a_2\otimes \cdots \otimes m-a'\otimes
a_1a_2\otimes \cdots \otimes m+\cdots \]
One can easily check that 
\[ \tilde \partial^2(a'\otimes a_1\otimes \cdots\otimes a_k\otimes
m)=ha'r(a_1,a_2)\otimes \cdots\otimes m.\]
Thus, the resulting endomorphism agrees with that induced by the
cocycle $r$.
\end{proof}

We note that if the action of $A$ on $M$ is an
$A_\infty$ action, then we can consider any deformation of the
$A_\infty$ operations: these induce a precomplex structure on
$\tilde M\otimes \mathsf{T} SA\otimes A$ whose square calculates the same
Hochschild homology class.

\section{Ladder bimodules}
\label{sec:ladder-bimodules}

\subsection{The case of \texorpdfstring{$Y$}{Y}-ladders}
\label{sec:case-y-ladders}

Fix a triple of integers $a+b=c\leq n$.  

Let $L_{a,b}$ be the unique highest weight simple module over
$T^{a,b}_{\omega_c}$. We'll instead want to think of this simple as a module
over $\tilde{T}^{a,b}_{\omega_c}$.  Let $e_{a,b}$ be the idempotent in $\tilde{T}^{a,b}$
summing straight line diagrams with the red strands at far left and
far right, and all black strands between them.

By
  \cite[\ref{m-add-red}]{Webmerged}, we have the isomorphism
  $e_{a,b}T^{a,b}e_{a,b}\cong T^a_{\omega_c-\omega_b}$, and thus a
  surjective map $e_{a,b}\tilde{T}^{a,b}e_{a,b}\to
  T^a_{\omega_c-\omega_b}$.  Since we have that
  $\omega_c-\omega_b=(0,\dots, 0,1,\dots, 1,0,\dots,0)$ with $b$ 0's,
  $a$ 1's and $n-c$ 0's, the corresponding weight space in
  $\wedgep{a}$ is 1-dimensional and extremal. The algebra $T^a_{\omega_c-\omega_b}$
  is thus a matrix algebra by \cite[\ref{m-proj-irr}]{Webmerged}.

\begin{proposition}\label{prop-a-b-only}
We have that $L_{a,b}=e_{a,b}L_{a,b}$ and $L_{a,b}$ is thus isomorphic
to the unique simple module over
$e_{a,b}T^{a,b}e_{a,b}\cong T^a_{\omega_c-\omega_b}$. 
\end{proposition}
\begin{proof}
  Since the simple module $L_{a,b}$ satisfies $\eE_i L_{a,b}=0$, we have
  that this simple is killed by the idempotent $e_{\kappa,\Bi}$ if the
  second red line is not at the far right.  Thus, by process of
  elimination, $e_{a,b}$ acts by the identity. \end{proof}

We can describe this simple more explicitly using the fact that
$T^a_{\omega_c-\omega_b}$ 
  has a graded cellular basis by work of Hu and Mathas \cite{HM}. This basis is indexed by pairs of standard tableaux on an $a\times b$ rectangle. 

In this paper, we read Young tableaux from the bottom left corner to the bottom right corner, and then in rows from 
left to right up the diagram. The content of a box in a Young tableau we define to be equal to $a+j-i$, where $i$ is the row number of the box and $j$ its column number and $a$ is the number of rows in $\sS$. We further identify this content with the corresponding root
$\al_{a+j-i}$ of $\gln$.

\begin{definition}
  Given a standard tableau $\sS$, we let the {\bf content reading word} of
  $\sS$ be the word whose $i$th entry is the content of the box
  containing $i$ in $\sS$.  The {\bf row reading word} of $\sS$ is the
  entries of $\sS$ read in the natural order explained above. 
\end{definition}
\begin{example}
Consider the following two tableaux:
$$\tikz[thick,scale=.8]{\node at (-1,1) {$\sS_1=$};\draw (0,0) -- (2,0);\draw (0,1) -- (2,1);\draw (0,2)
  -- (2,2);\draw (0,0) -- (0,2);\draw (1,0) -- (1,2);\draw (2,0) --
  (2,2); \node at (.5,.5){$1$};\node at (1.5,.5){$2$};\node at
  (.5,1.5){$3$};\node at (1.5,1.5){$4$};}\qquad \qquad
\tikz[thick,scale=.8]{\node at (-1,1) {$\sS_2=$};\draw (0,0) -- (2,0);\draw (0,1) -- (2,1);\draw
  (0,2) -- (2,2);\draw (0,0) -- (0,2);\draw (1,0) -- (1,2);\draw (2,0)
  -- (2,2);\node at (.5,.5){$1$};\node at (1.5,.5){$3$};\node at
  (.5,1.5){$2$};\node at (1.5,1.5){$4$};}
$$
The content reading word of $\sS_1$ is equal to $(2,3,1,2)$ and its row reading word is equal to $(1,2,3,4)$. 
The content reading word of $\sS_2$ is equal to $(2,1,3,2)$ and its row reading word is equal to $(1,3,2,4)$. 
\end{example}

\excise{We'll index the content of this diagram by putting the top
left corner equal to $1$; that is, the content of the box $(i,j)$ is
$a+j-i$.  We'll use French notation for tableaux, so $(1,1)$ is the
bottom left corner of the tableau.  For example, if $a=2,b=3$, then the contents are:
\[\tikz[thick,scale=.6,baseline]{\draw (0,0) -- (3,0);\draw (0,1) -- (3,1);\draw (0,2)
  -- (3,2);\draw (0,0) -- (0,2);\draw (1,0) -- (1,2);\draw (2,0) --
  (2,2); \draw (3,0) --
  (3,2);\node at (.5,.5){$2$};\node at (1.5,.5){$3$};\node at
  (.5,1.5){$1$};\node at (1.5,1.5){$2$};\node at
  (2.5,.5){$4$};\node at (2.5,1.5){$3$};}\]}

Let $\sR$ be the tableau in which the box $(i,j)$
has filling $(i-1)b+j$, that is, with row reading $(1,2,3,\dots, ab)$. Attached to $\sS$, there is a unique permutation $w_\sS$ sending
$(1,2,3,\dots)$ to the
row reading word of $\sS$.  Note that $w_\sR$ is the identity. 

The basis vector $C_{\sS,\sT}$ for a pair $\sS,\sT$ is given by the diagram that traces out a string
diagram for $w_\sS$ read upward from $y=\nicefrac 12$ to $y=1$ and a string
diagram for $w_\sT$ read downward from $y=\nicefrac 12$ to $y=0$.  The
labels on strands are determined by the property that
\begin{itemize}
\item at $y=1$, reading left to right gives the content reading word of $\sS$.
\item at $y=\nicefrac 12$, reading left to right gives the content reading word of $\sR$,
  that is, $(\al_a,\al_{a+1},\dots, \al_{c-1},\al_{a-1},\dots, \al_{c-2},\dots,
  \al_1,\dots, \al_b)$.
\item at $y=0$, reading left to right gives the content reading word of $\sT$.  
\end{itemize}

This follows the general principle that each strand is attached to a
box in the tableau, and the strand connects the entries
of the sequences at $y=0,\nicefrac 12 ,1$ associated to the same box
of the diagrams. 

\begin{proposition}[\mbox{Hu-Mathas}]\label{prop:Lab-basis}
  The vectors $C_{\sS,\sT}$ form a basis of
  $T^{a}_{\omega_c-\omega_b}$ with the multiplication rule 
\[C_{\sS,\sT}C_{\sS',\sT'}=
\begin{cases}
  C_{\sS,\sT'} & \sS'=\sT\\
0 & \sS'\neq\sT
\end{cases}
\]
\end{proposition}
This is a special case of a much more general result for all type A
cyclotomic quotients \cite[Main Theorem]{HM}.

In particular $C_{\sS,\sS}$ is an idempotent.  Note that all the
vectors $C_{\sS,\sT}$ are of degree 0, showing that all elements of
$T^{a}_{\omega_c-\omega_b}$ of positive or negative degree vanish.  

\begin{example}
  Consider the case where $a=b=2$, so $c=4$.  We're considering
  tableaux on a $2\times 2$ rectangle, of which there are two:
\[\tikz[thick,scale=.8]{\draw (0,0) -- (2,0);\draw (0,1) -- (2,1);\draw (0,2)
  -- (2,2);\draw (0,0) -- (0,2);\draw (1,0) -- (1,2);\draw (2,0) --
  (2,2); \node at (.5,.5){$1$};\node at (1.5,.5){$2$};\node at
  (.5,1.5){$3$};\node at (1.5,1.5){$4$};}\qquad \qquad
\tikz[thick,scale=.8]{\draw (0,0) -- (2,0);\draw (0,1) -- (2,1);\draw
  (0,2) -- (2,2);\draw (0,0) -- (0,2);\draw (1,0) -- (1,2);\draw (2,0)
  -- (2,2);\node at (.5,.5){$1$};\node at (1.5,.5){$3$};\node at
  (.5,1.5){$2$};\node at (1.5,1.5){$4$};}\]
Thus, we have 4 basis vectors in $e_{a,b}T^{a,b}e_{a,b}$, one for each
pair of these tableaux
\[\tikz[thick,scale=.6,baseline]{\draw (0,0) -- (2,0);\draw (0,1) -- (2,1);\draw (0,2)
  -- (2,2);\draw (0,0) -- (0,2);\draw (1,0) -- (1,2);\draw (2,0) --
  (2,2); \node at (.5,.5){$1$};\node at (1.5,.5){$2$};\node at
  (.5,1.5){$3$};\node at (1.5,1.5){$4$};}\quad
\tikz[thick,scale=.6,baseline]{\draw (0,0) -- (2,0);\draw (0,1) -- (2,1);\draw
  (0,2) -- (2,2);\draw (0,0) -- (0,2);\draw (1,0) -- (1,2);\draw (2,0)
  -- (2,2);\node at (.5,.5){$1$};\node at (1.5,.5){$2$};\node at
  (.5,1.5){$3$};\node at (1.5,1.5){$4$};}\qquad 
\tikz[thick,yscale=.6,baseline,xscale=1.2]{\draw[wei] (0,0) -- node[below,at start]{$2$}(0,2);\draw (.5,0) -- node[below,at start]{$2$} (.5,2);\draw (1,0) -- node[below,at start]{$3$} (1,2);\draw (1.5,0) -- node[below,at start]{$1$} (1.5,2);\draw
  (2,0) -- node[below,at start]{$2$} (2,2);\draw[wei] (2.5,0) --
  node[below,at start]{$2$} (2.5,2);}\qquad \qquad \tikz[thick,scale=.6,baseline]{\draw (0,0) -- (2,0);\draw (0,1) -- (2,1);\draw
  (0,2) -- (2,2);\draw (0,0) -- (0,2);\draw (1,0) -- (1,2);\draw (2,0)
  -- (2,2);\node at (.5,.5){$1$};\node at (1.5,.5){$3$};\node at
  (.5,1.5){$2$};\node at (1.5,1.5){$4$};}\quad
\tikz[thick,scale=.6,baseline]{\draw (0,0) -- (2,0);\draw (0,1) -- (2,1);\draw
  (0,2) -- (2,2);\draw (0,0) -- (0,2);\draw (1,0) -- (1,2);\draw (2,0)
  -- (2,2);\node at (.5,.5){$1$};\node at (1.5,.5){$2$};\node at
  (.5,1.5){$3$};\node at (1.5,1.5){$4$};}\qquad 
\tikz[thick,yscale=.6,baseline,xscale=1.2]{\draw[wei] (0,0) --
  node[below,at start]{$2$}(0,2);\draw (.5,0) -- node[below,at
  start]{$2$} (.5,2);\draw (1,0) to[out=90,in=-90] node[below,at start]{$3$}  (1,1) to[out=90,in=-90] (1.5,2);\draw (1.5,0)  to[out=90,in=-90]node[below,at start]{$1$}  (1.5,1) to[out=90,in=-90]  (1,2);\draw
  (2,0) -- node[below,at start]{$2$} (2,2);\draw[wei] (2.5,0) --
  node[below,at start]{$2$} (2.5,2);}\]
\[\tikz[thick,scale=.6,baseline]{\draw (0,0) -- (2,0);\draw (0,1) -- (2,1);\draw
  (0,2) -- (2,2);\draw (0,0) -- (0,2);\draw (1,0) -- (1,2);\draw (2,0)
  -- (2,2);\node at (.5,.5){$1$};\node at (1.5,.5){$2$};\node at
  (.5,1.5){$3$};\node at (1.5,1.5){$4$};}\quad
\tikz[thick,scale=.6,baseline]{\draw (0,0) -- (2,0);\draw (0,1) -- (2,1);\draw
  (0,2) -- (2,2);\draw (0,0) -- (0,2);\draw (1,0) -- (1,2);\draw (2,0)
  -- (2,2);\node at (.5,.5){$1$};\node at (1.5,.5){$3$};\node at
  (.5,1.5){$2$};\node at (1.5,1.5){$4$};}\qquad
\tikz[thick,yscale=.6,baseline,xscale=1.2]{\draw[wei] (0,0) --
  node[below,at start]{$2$}(0,2);\draw (.5,0) -- node[below,at
  start]{$2$} (.5,2);\draw (1.5,0) to[out=90,in=-90] node[below,at start]{$3$}  (1,1) to[out=90,in=-90] (1,2);\draw (1,0)  to[out=90,in=-90]node[below,at start]{$1$}  (1.5,1) to[out=90,in=-90]  (1.5,2);\draw
  (2,0) -- node[below,at start]{$2$} (2,2);\draw[wei] (2.5,0) --
  node[below,at start]{$2$} (2.5,2);}\qquad\qquad \tikz[thick,scale=.6,baseline]{\draw (0,0) -- (2,0);\draw (0,1) -- (2,1);\draw
  (0,2) -- (2,2);\draw (0,0) -- (0,2);\draw (1,0) -- (1,2);\draw (2,0)
  -- (2,2);\node at (.5,.5){$1$};\node at (1.5,.5){$3$};\node at
  (.5,1.5){$2$};\node at (1.5,1.5){$4$};}\quad
\tikz[thick,scale=.6,baseline]{\draw (0,0) -- (2,0);\draw (0,1) -- (2,1);\draw
  (0,2) -- (2,2);\draw (0,0) -- (0,2);\draw (1,0) -- (1,2);\draw (2,0)
  -- (2,2);\node at (.5,.5){$1$};\node at (1.5,.5){$3$};\node at
  (.5,1.5){$2$};\node at (1.5,1.5){$4$};}\qquad
\tikz[thick,yscale=.6,baseline,xscale=1.2]{\draw[wei] (0,0) --
  node[below,at start]{$2$}(0,2);\draw (.5,0) -- node[below,at
  start]{$2$} (.5,2);\draw (1.5,0) to[out=90,in=-90] node[below,at start]{$3$}  (1,1) to[out=90,in=-90] (1.5,2);\draw (1,0)  to[out=90,in=-90]node[below,at start]{$1$}  (1.5,1) to[out=90,in=-90]  (1,2);\draw
  (2,0) -- node[below,at start]{$2$} (2,2);\draw[wei] (2.5,0) --
  node[below,at start]{$2$} (2.5,2);}\]
The relation (\ref{black-bigon}) shows that this is isomorphic to a
$2\times 2$ matrix algebra, as expected.
\end{example}
As with any matrix algebra, there is a single irreducible
representation, given by the left ideal
$L_{a,b}\cong T^a_{\omega_c-\omega_b}C_{\sR,\sR}$; in fact, we could replace $\sR$
by any $\sS$, but we prefer to have a fixed choice. 

 We will use a graphical notation for elements of this
 representation.  To avoid confusion with the algebra, we'll represent
 $C_{\sR,\sR}$ thought of as an element of $L_{a,b}$ 
as the diagram
\begin{equation}
C_{\sR,\sR}=\tikz[baseline,very thick,yscale=2,xscale=2.5]{\draw (0,0) -- (.35,.25) to node
  [above, at end]{$b$}(.7,.5); \draw
  (0,0) -- node
  [above, at end]{$a$} (-.7,.5);\draw (0,0) -- node
  [above, at end]{$b-1$} (.4,.5); \draw (0,0) -- node
  [above, at end]{$a+1$} (-.4,.5);
  \draw[wei] (0,0) -- (0,-.3) ; \draw[wei] (0,0)-- (.5,.25) to node
  [above, at end]{$b$} (1,.5);\draw[wei]
  (0,0) -- node
  [above, at end]{$a$} (-1,.5); \node at (0,.4){$\cdots$};}\label{CRR}
\end{equation}

By our usual conventions, the labelling of the strands reminds us that
all idempotents corresponding to other tableaux act by zero, since the
labels on strands won't match when we compose the diagrams.
If, as before, we take $a=b=2$, then this module is 2-dimensional,
with basis given by 
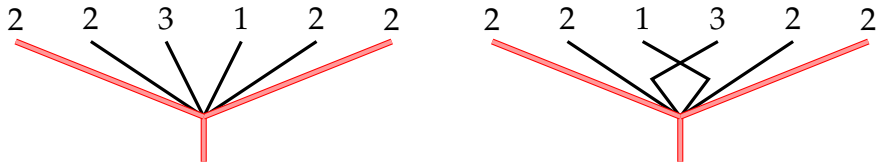
\[\tikz[baseline,very thick,yscale=2,xscale=2.5]{\draw (0,0) to node
  [above, at end]{$2$}(.6,.5); \draw
  (0,0) -- node
  [above, at end]{$2$} (-.6,.5);\draw (0,0) -- node
  [above, at end]{$1$} (.2,.5); \draw (0,0) -- node
  [above, at end]{$3$} (-.2,.5);
  \draw[wei] (0,0) -- (0,-.3) ; \draw[wei] (0,0)-- (.5,.25) to node
  [above, at end]{$2$} (1,.5);\draw[wei]
  (0,0) -- node
  [above, at end]{$2$} (-1,.5);}\qquad \tikz[baseline,very thick,yscale=2,xscale=2.5]{\draw (0,0) to node
  [above, at end]{$2$}(.6,.5); \draw
  (0,0) -- node
  [above, at end]{$2$} (-.6,.5);\draw (0,0) --  (.15,.25) --node
  [above, at end]{$1$} (-.2,.5); \draw (0,0) -- (-.15,.25) --node
  [above, at end]{$3$} (.2,.5);
  \draw[wei] (0,0) -- (0,-.3) ; \draw[wei] (0,0)-- (.5,.25) to node
  [above, at end]{$2$} (1,.5);\draw[wei]
  (0,0) -- node
  [above, at end]{$2$} (-1,.5);} \]

All the relations of this module are encapsulated in the fact that
$(1-e_{a,b})L_{a,b}=0$ and the relations in
$T^a_{\omega_c-\omega_b}$.  However, it will be useful to record some
of them here for later use:
\newseq
\begin{equation*}\subeqn
\label{Lbc-relations}\tikz[baseline,very thick,scale=2]{\draw (0,0) -- (.35,.25) to node
  [above, at end]{$b$} (1,.5); \draw
  (0,0) --node
  [above, at end]{$a$} (-.7,.5);\draw (0,0) -- (.4,.5); \draw (0,0) -- (-.4,.5);
  \draw[wei] (0,0) -- (0,-.5) ; \draw[wei] (0,0)-- (.5,.25) to node
  [above, at end]{$b$} (.7,.5);\draw[wei]
  (0,0) -- node
  [above, at end]{$a$} (-1,.5); \node at (0,.4){$\cdots$};}=
\tikz[baseline,very thick,scale=2]{\draw (0,0) -- (-.35,.25) to node
  [above, at end]{$a$} (-1,.5); \draw
  (0,0) -- node
  [above, at end]{$b$} (.7,.5);\draw (0,0) -- (.4,.5); \draw (0,0) -- (-.4,.5);
  \draw[wei] (0,0) -- (0,-.5) ; \draw[wei] (0,0)-- (-.5,.25) to node
  [above, at end]{$a$} (-.7,.5);\draw[wei]
  (0,0) -- node
  [above, at end]{$b$} (1,.5); \node at (0,.4){$\cdots$};}=0
\end{equation*}
\begin{equation*}\subeqn
  \label{eq:4}
  \tikz[baseline,very thick,scale=2]{\draw (0,0) -- (.1,.25) to node
  [above, at end]{$i$} (-.2,.5); \draw (0,0) -- (-.1,.25) to  node
  [above, at end]{$i-1$} (.2,.5); 
  \draw[wei] (0,0) -- (0,-.5) ; \draw[wei] (0,0)-- (1,.5);\draw[wei]
  (0,0) -- (-1,.5); \node at (.4,.4){$\cdots$};\node at (-.4,.4){$\cdots$};}=  \tikz[baseline,very thick,scale=2]{\draw (0,0) -- (.3,.25) to node
  [above, at end]{$i$} (-.2,.5); \draw (0,0) -- (-.3,.25) to  node
  [above, at end]{$i+1$} (.2,.5); 
  \draw[wei] (0,0) -- (0,-.5) ; \draw[wei] (0,0)-- (1,.5);\draw[wei]
  (0,0) -- (-1,.5); \node at (.45,.4){$\cdots$};\node at (-.45,.4){$\cdots$};\node at (0,.2){$\cdots$};}=0
\end{equation*}

\excise{Let $v_{I}$ for $I\subset [1,n]$ denote the wedge product
$v_{i_1}\wedge \cdots \wedge v_{i_m}$.  

\begin{proposition}
  The class of $L_{a,b}$ in $K(\tilde{T}^{a,b}\mmod)$ is $1\otimes
  v_{a,b}$ where $v_{a,b}$ is the
  unique highest weight vector of the form $v_{[b+1,c]}\otimes
  v_{[1,b]}+\cdots$  
\end{proposition}
\begin{proof}
  The defining property of $[L_{a,b}]$ is that it is orthogonal to the
  image of the left and right actions of $F_i$ for all $i$, and that
  $\langle P_{\sR},L_{a,b}\rangle =1$.  

\bentodo{This needs to be fixed up.}
\end{proof}}

Note that~\eqref{Lbc-relations} and~\eqref{eq:4} imply that all
diagrams of positive or negative degree are zero in $L_{a,b}$; this 
also follows from the basis given by Proposition \ref{prop:Lab-basis}.  In particular, this 
also kills any diagram which contains a dot. If we let $d$ denote the portion of the
diagram including the dot and things below it and $d'$ denote the portion of the
diagram strictly below the dot, then $\deg(d')=\deg(d)+2$.  However
this means that one of these diagrams has non-zero degree, and thus is
0.

Ultimately, we'll attach a bimodule to any ladder, but let us start with
the case of a single trivalent vertex.  Recall $Y_i$ and $Y_i^\star$:
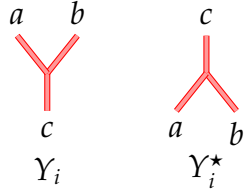
\[\tikz[baseline,very thick]{\draw[wei] (0,0) -- node [at end, below]{$c$} (0,-.5) ; \draw[wei] (0,0)-- node [at end, above]{$b$}  (.4,.5);\draw[wei]
  (0,0) --node [at end, above]{$a$} (-.4,.5); \node at
  (0,-1.3){$Y_i$};}\qquad \tikz[baseline,very thick]{
\draw[wei] (0,0)
  -- node [at end, above]{$c$} (0,.5); 
\draw[wei] (0,0)-- node [at end, below]{$b$}  (.4,-.5);
\draw[wei]
  (0,0) --node [at end, below]{$a$} (-.4,-.5);
\node at (0,-1.3){$Y_i^\star$};}\]
\begin{definition}
  A {\bf ladder diagram} for a ladder with a single trivalent vertex is a diagram
  with
  \begin{itemize}
  \item red lines that trace out the ladder, labeled with the appropriate
    fundamental weights (which we'll just denote with their number)
  \item black lines that map immersively to $[0,1]$; as usual,
    these are labeled with simple roots and constrained intersect generically,
    that is, with no triple points or tangencies.  These can carry any
    number of dots, which must avoid intersection points.  
  \item over the triple point of the ladder, we include a box which
    contains an element of $L_{a,b}$.  If the $Y$ opens upward, this
    box connects with 2 red and $ab$ different black strands if
    $a+b=c$ at distinct points on
    the top of the box, and connects with one red strand labeled $c$
    at the bottom.  If the $Y$ opens downwards, we reflect this
    configuration through a horizontal axis. 
\item every black strand is constrained to have its endpoints on
  $y=1,y=0$ or on the box.  By the requirement that the projection to
  the $y$-axis is an immersion, we see that the only possible
  combinations are one endpoint at $y=0$ and the other at $y=1$, or
  one at $y=1$ (resp. $y=0$) and the other on the box if the $Y$ opens
  upward (resp. downward).
  \end{itemize}
\end{definition}
Using the notation introduced above for elements of $L_{a,b}$, we can
represent a ladder diagram as a Stendhal diagram where at the trivalent
vertex, we have inserted a picture which looks like the diagram
\eqref{CRR}.  The relation moving a crossing or dot in or out of the box just
becomes an isotopy in this schema; the relations of $L_{a,b}$ are
encapsulated in the local relations (\ref{Lbc-relations}--\ref{eq:4}).

An example of such a diagram is 
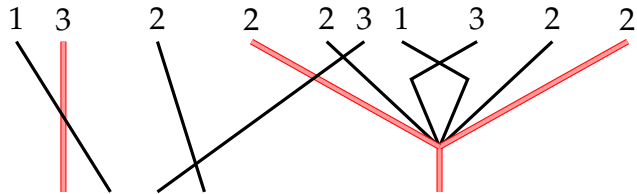
\[\tikz[baseline,very thick,yscale=2,xscale=2.5]{\draw (0,-.2) to node
  [above, at end]{$2$}(.6,.5); \draw
  (0,-.2) -- node
  [above, at end]{$2$} (-.6,.5);\draw (0,-.2) --  (.15,.25) --node
  [above, at end]{$1$} (-.2,.5); \draw (0,-.2) -- (-.15,.25) --node
  [above, at end]{$3$} (.2,.5);
  \draw[wei] (0,-.2) -- (0,-.5) ; \draw[wei] (0,-.2) to node
  [above, at end]{$2$} (1,.5);\draw[wei]
  (0,-.2) -- node
  [above, at end]{$2$} (-1,.5);  
\draw[wei]
  (-2,-.5) -- node
  [above, at end]{$3$} (-2,.5);\draw
  (-1.5,-.5) -- node
  [above, at end]{$3$} (-.4,.5);
\draw
  (-1.25,-.5) -- node
  [above, at end]{$2$} (-1.5,.5);
\draw
  (-1.75,-.5) -- node
  [above, at end]{$1$} (-2.25,.5);
} \]
Note that the $\K$-span of these diagrams is naturally a bimodule over
the algebras $\doubletilde{T}$ freely spanned by Stendhal diagrams for
the sequences at top and bottom of the ladder.

\begin{definition}
  The ladder bimodule $W_{Y_i}$ of a ladder with one trivalent vertex as above
  is the quotient of the $\K$-span of ladder diagrams by the local
  relations (\ref{first-QH}--\ref{triple-dumb}) and (\ref{red-triple-correction}--\ref{cost})
as well as:
\begin{itemize}
\item for any element of $\tilde{T}^{a,b}$, we obtain the
  same result by tracing out the Stendhal diagrams with the strands
  leading to the top of the box, or by
  acting on the element of $L_{a,b}$ in the box.
\newseq
\item black strands can pass through the ladder triple point up to a
  scalar.  Let the scalar 
$    \sigma(m,a,b)$ be $1$ if the number of elements of $\{a,b,c\}$ less
than $m$ is odd and $-1$ if it is even.  That is, $\sigma(m,a,b)=1$ if
$a\leq m< b$ or $b\leq m< a$ or $m\geq c$ and $-1$ otherwise.
\[\subeqn\label{through-box1}\tikz[baseline,very thick]{\draw (0,0) -- (.7,.5); \draw
  (0,0) -- (-.7,.5);\draw (0,0) -- (.4,.5); \draw (0,0) -- (-.4,.5);
  \draw[wei] (0,0) -- node[below,at end]{$c$} (0,-.5) ; \draw[wei] (0,0)-- node[above,at
  end]{$b$} (1,.5);\draw[wei]
  (0,0) -- node[above,at
  end]{$a$}  (-1,.5); \node at (0,.4){$\cdots$}; \draw (-1,-.5) --
  node[below,at start]{$m$}
  (.4,0)--(1.3,.5);}= \tikz[baseline,very thick]{\draw (0,0) -- (.7,.5); \draw
  (0,0) -- (-.7,.5);\draw (0,0) -- (.4,.5); \draw (0,0) -- (-.4,.5);
  \draw[wei] (0,0) -- node[below,at end]{$c$} (0,-.5) ; \draw[wei] (0,0)-- node[above,at
  end]{$b$}  (1,.5);\draw[wei]
  (0,0) --node[above,at
  end]{$a$}  (-1,.5); \node at (0,.4){$\cdots$}; \draw (-1,-.5) --  node[below,at start]{$m$}
  (-.5,.1)--(1.3,.5);}\]
\[ \subeqn\label{through-box2}\tikz[baseline,very thick]{\draw (0,0) -- (.7,.5); \draw
  (0,0) -- (-.7,.5);\draw (0,0) -- (.4,.5); \draw (0,0) -- (-.4,.5);
  \draw[wei] (0,0) -- node[below,at end]{$c$} (0,-.5) ; \draw[wei] (0,0)-- node[above,at
  end]{$b$}  (1,.5);\draw[wei]
  (0,0) -- node[above,at
  end]{$a$}  (-1,.5); \node at (0,.4){$\cdots$}; \draw (1,-.5) -- node[below,at start]{$m$}
  (-.4,0)--(-1.3,.5);}=\sigma(m,a,b) \tikz[baseline,very thick]{\draw (0,0) -- (.7,.5); \draw
  (0,0) -- (-.7,.5);\draw (0,0) -- (.4,.5); \draw (0,0) -- (-.4,.5);
  \draw[wei] (0,0) -- node[below,at end]{$c$} (0,-.5) ; \draw[wei] (0,0)-- node[above,at
  end]{$b$}  (1,.5);\draw[wei]
  (0,0) -- node[above,at
  end]{$a$}  (-1,.5); \node at (0,.4){$\cdots$}; \draw (1,-.5) -- node[below,at start]{$m$}
  (.5,.1)--(-1.3,.5);}\]
\end{itemize}
We let $W_{Y_i^\star}$ be the bimodule $W_{Y_i}$ with left and right actions
exchanged by the anti-automorphism of $T^{\Bp}$ and $T^{\Bp'}$
reflecting diagrams through the $x$-axis with
internal and homological grading shifted downward by the quantity
$\eta(Y_i)=- ab$.
Pictorially, we can think
of this as a bimodule with the same definition as $W_{Y_i}$ but with the
$Y$ upside down (and a grading shift). 
\end{definition}
We fix $\Bp$ to be a sequence of indices given by the bottom of $Y_i$
(resp. top of $Y_i^\star$), with $p_i=c$ and $\Bp'$ to be the
sequence at the top of $Y_i$ (resp. bottom of $Y_i^\star$).  That is,
we have that $\Bp'=(\dots, p_{i-1},a,b,p_{i+1},\dots)$.  We can attach
a Stendhal diagram from $\tilde{T}^{\Bp'}$ to the top of $Y_i$ or
bottom of $Y_i^\star$; similarly, a diagram from $\tilde{T}^{\Bp}$ can be
attached to the bottom of $Y_i$ or top of $Y_i^\star$.  This endows
$W_{Y_i}$ (resp. $W_{Y_i^\star}$) with a
$\tilde{T}^{\Bp'}\operatorname{-}\tilde{T}^{\Bp}$-bimodule structure
(resp. $\tilde{T}^{\Bp}\operatorname{-}\tilde{T}^{\Bp'}$-bimodule structure) by
the stacking of diagrams. 

If $X$ and $Y$ are functors between categories with a categorical
action of $\tU_n$, then a {\bf strong equivariant} structure on $X$ is
a system of natural isomorphisms $u\circ X\cong X\circ u$ for every
1-morphism $u$ which satisfies the obvious commutative square for
every 2-morphism.  We say that a natural transformation $X\to Y$ of
strongly equivariant functors 
{\bf commutes with the action of $\tU_n$ (resp. $\tU_n^-$)} if the obvious square
commutes.  
\begin{proposition}\label{prop:trivalent-commute}
  The ladder bimodules $W_{Y_i},W_{Y_i^\star}$  commute with
  induction functors, that is, they have a strongly equivariant
  structure for the action of $\tU^-_n\times\tU^-_n$.
\end{proposition}
\begin{proof}
Consider the case of $W_{Y_i}$.  There is an obvious map from $\fF_{M} (W_A\otimes N)\to W_A\otimes
  \fF_MN$ given by the obvious inclusion of diagrams:  we can simply
  extend the new strands downward so that they are added before we
  tensor with $W_A$ rather than after.

Furthermore, this map is surjective, since in any diagram of $W_A\otimes
  \fF_MN$ where a new strand passes below the $Y$, we can use isotopy,
  and the relations (in particular (\ref{through-box1}--\ref{through-box2})) to slide the
  new strands above the $Y$.  Thus this diagram is in the image of
  this map.

Now, consider an element of its kernel.  This is a sum of diagrams
which become 0  if we allow relations where the new strands can pass
below the $Y$, but not if they must stay above.  Thus, in $W_A\otimes
  \fF_MN$, we can write it as a sum of relations.  Now, we take these
  relations and ``unzip the $Y$.'' That is, in each relation, we push
  the branch point in the $Y$ further down.  When we push through a
  black strand, we include the sign in (\ref{through-box1}--\ref{through-box2}); we are not
  using this relation, but rather we wish to show that doing this on both sides
  of a relation will result in a new relation, by the locality of
  relations, or will simply result in the two sides coinciding if the
  relation is  (\ref{through-box1}--\ref{through-box2}).  

This requires some care about the relations
(\ref{red-triple-correction},\ref{cost}).  For
(\ref{red-triple-correction}), we wish to show that 
\begin{equation}
\tikz[baseline,very thick,yscale=2,xscale=1.5]{\draw (.7,-.5) to node
  [above, at end]{$b$}(.7,.5); \draw
  (-.7,-.5) -- node
  [above, at end]{$a$} (-.7,.5);\draw (.4,-.5) -- (.4,.5); \draw (-.4,-.5) -- (-.4,.5);
\draw[wei] (1,-.5) to node
  [above, at end]{$b$} (1,.5);\draw[wei]  (-1,-.5) -- node
  [above, at end]{$a$} (-1,.5); \node at (0,0){$\cdots$};
\draw (1.4,-.5) to[out=90,in=-35] (1.15,0) to[out=145,in=-20] node
  [above, at end]{$k$}(-1.4,.5);\draw (-1.4,-.5) to[out=20,in=-145] (1.15,0) to[out=35,in=-90] node
  [above, at end]{$k$}(1.4,.5);}-\tikz[baseline,very thick,yscale=2,xscale=1.5]{\draw (.7,-.5) to node
  [above, at end]{$b$}(.7,.5); \draw
  (-.7,-.5) -- node
  [above, at end]{$a$} (-.7,.5);\draw (.4,-.5) -- (.4,.5); \draw (-.4,-.5) -- 
 (-.4,.5);
\draw[wei] (1,-.5) to node
  [above, at end]{$b$} (1,.5);\draw[wei]  (-1,-.5) -- node
  [above, at end]{$a$} (-1,.5); \node at (0,0){$\cdots$};
\draw (1.4,-.5) to[out=160,in=-35] (-1.15,0) to[out=145,in=-90] node
  [above, at end]{$k$}(-1.4,.5);\draw (-1.4,-.5) to[out=90,in=-145] (-1.15,0) to[out=35,in=-160] node
  [above, at end]{$k$}(1.4,.5);}=\delta_{k,c} e_{k,\omega_a,a,\dots,b,\omega_b,k}
\end{equation}
In order to do this calculation, we apply the relations
(\ref{triple-dumb}, \ref{red-triple-correction}) successively.
We leave the reader to check that whenever $c\neq k$, the correction
terms which appear create a bigon with both sides labeled $k$, and
thus are 0 by  (\ref{black-bigon}).  If $k=c$, then
the crossing commutes past every strand in the diagram except the one
labeled $c-1$.  In moving past this strand, we apply the relation
(\ref{triple-dumb});  now, we have a non-zero correction term which
breaks open the crossing.  The relation (\ref{black-bigon}) shows that the strands labeled $c$ can be pulled straight, and thus give the idempotent attached to
$c,\omega_a,a,\dots,b,\omega_b,c$.

Now we return to (\ref{cost}).  What we must show is that \begin{equation}\label{split-bigon}
\tikz[baseline,very thick,yscale=2.2,xscale=2.5]{\draw (0,0) to node
  [above, at end]{$b$}(.7,.7); \draw
  (0,0) -- node
  [above, at end]{$a$} (-.7,.7);\draw (0,0) -- node
  [above, at end]{$b-1$} (.4,.7); \draw (0,0) -- node
  [above, at end]{$a+1$} (-.4,.7);
  \draw[wei] (0,0) -- (0,-.7) ; \draw[wei] (0,0) to node
  [above, at end]{$b$} (1,.7);\draw[wei]
  (0,0) -- node
  [above, at end]{$a$} (-1,.7); \node at (0,.4){$\cdots$}; \draw
  (1,-.7) to[out=90,in=-10] (.3,.1) to[in=-30,out=170] (-.8,.25) to[in=-150,out=150] (-.8,.35) to[out=30,in=-150]node
  [above, at end]{$k$} (1.3,.7);}\hspace{-3mm}=\sigma(k,a,b)\hspace{-3mm}\tikz[baseline,very thick,yscale=2.2,xscale=2.5]{\draw (0,0) to node
  [above, at end]{$b$}(.7,.7); \draw
  (0,0) -- node
  [above, at end]{$a$} (-.7,.7);\draw (0,0) -- node
  [above, at end]{$b-1$} (.4,.7); \draw (0,0) -- node
  [above, at end]{$a+1$} (-.4,.7);
  \draw[wei] (0,0) -- (0,-.7) ; \draw[wei] (0,0) to node
  [above, at end]{$b$} (1,.7);\draw[wei]
  (0,0) -- node
  [above, at end]{$a$} (-1,.7); \node at (0,.4){$\cdots$}; \draw
  (.5,-.7) to[in=-60,out=140] (-.3,-.4) to[out=120,in=-120] (-.3,-.3)  to[out=60,in=-140] (1,.2) to[out=40,in=-90]node
  [above, at end]{$k$} (1.3,.7);}
\end{equation}
In order to do these calculations, let us note the following
relations:
\newseq
\[\label{tetris-long}\subeqn\begin{tikzpicture}[very thick,scale=1.5,baseline]
\draw (1.5,-.5) to[out=90,in=-90] node[at start,below]{$k$} (0,0) to[out=90,in=-90] (1.5,.5) ;
\draw (1,-.5) to[out=90,in=-90] node[at start,below]{${k\mp 1}$} (1.5,0) to[out=90,in=-90] (1,.5) ;
\draw (.5,-.5) to[out=90,in=-90] node[at start,below]{${k}$} (1,0) to[out=90,in=-90] (.5,.5);
\draw (0,-.5) to[out=90,in=-90] node[at start,below]{${k\pm 1}$} (.5,0) to[out=90,in=-90] (0,.5);
\end{tikzpicture}
=
\begin{tikzpicture}[very thick,scale=1.5,baseline]
\draw (1.5,-.5) to[out=90,in=-90] node[at start,below]{$k$} (1.5,.5);
\draw (1,-.5) to[out=90,in=-90] node[at start,below]{${k\mp 1}$} (1,.5);
\draw (.5,-.5) to[out=90,in=-90]node[at start,below]{${k}$}  (.5,.5) ;
\draw (0,-.5) to[out=90,in=-90] node[at start,below]{${k\pm 1}$}  (0,0) to[out=90,in=-90] (0,.5);
\end{tikzpicture}\mp
\begin{tikzpicture}[very thick,scale=1.5,baseline]
\draw (1.5,-.5) to[out=90,in=-90] node[at start,below]{$k$} (.5,.5);
\draw (1,-.5) to[out=90,in=-90] node[at start,below]{${k\mp 1}$} (.5,0) to[out=90,in=-90] (1,.5);
\draw (.5,-.5) to[out=90,in=-90] node[at start,below]{${k}$} (1.5,.5);
\draw (0,-.5) to[out=90,in=-90] node[at start,below]{${k\pm 1}$}  (0,0) to[out=90,in=-90] (0,.5);
\end{tikzpicture}
\]
\[\label{tetris-angle}\subeqn\begin{tikzpicture}[very thick,scale=1.5,baseline]
\draw (1.5,-.5) to[out=90,in=-90] node[at start,below]{$k$} (0,0) to[out=90,in=-90] (1.5,.5) ;
\draw (1,-.5) to[out=90,in=-90] node[at start,below]{${k\pm 1}$} (1.5,0) to[out=90,in=-90] (1,.5) ;
\draw (.5,-.5) to[out=90,in=-90] node[at start,below]{${k}$} (1,0) to[out=90,in=-90] (.5,.5);
\draw (0,-.5) to[out=90,in=-90] node[at start,below]{${k\pm 1}$} (.5,0) to[out=90,in=-90] (0,.5);
\end{tikzpicture}
= -
\begin{tikzpicture}[very thick,scale=1.5,baseline]
\draw (1.5,-.5) to[out=90,in=-90] node[at start,below]{$k$} (1.5,.5);
\draw (1,-.5) to[out=90,in=-90] node[at start,below]{${k\pm 1}$} (1,.5);
\draw (.5,-.5) to[out=90,in=-90]node[at start,below]{${k}$}  (.5,.5) ;
\draw (0,-.5) to[out=90,in=-90] node[at start,below]{${k\pm 1}$}  (0,0) to[out=90,in=-90] (0,.5);
\end{tikzpicture}\mp
\begin{tikzpicture}[very thick,scale=1.5,baseline]
\draw (1.5,-.5) to[out=90,in=-90] node[at start,below]{$k$} (.5,.5);
\draw (1,-.5) to[out=90,in=-90] node[at start,below]{${k\pm 1}$} (.5,0) to[out=90,in=-90] (1,.5);
\draw (.5,-.5) to[out=90,in=-90] node[at start,below]{${k}$} (1.5,.5);
\draw (0,-.5) to[out=90,in=-90] node[at start,below]{${k\pm 1}$}  (0,0) to[out=90,in=-90] (0,.5);
\end{tikzpicture}
\]
\[\label{tetris-red1}\subeqn\begin{tikzpicture}[very thick,scale=1.5,baseline]
\draw (1.5,-.5) to[out=90,in=-90] node[at start,below]{$k$} (0,0) to[out=90,in=-90] (1.5,.5) ;
\draw (1,-.5) to[out=90,in=-90] node[at start,below]{${k\pm 1}$} (1.5,0) to[out=90,in=-90] (1,.5) ;
\draw (.5,-.5) to[out=90,in=-90] node[at start,below]{$k$} (1,0) to[out=90,in=-90] (.5,.5);
\draw[wei] (0,-.5) to[out=90,in=-90] node[at start,below]{$k$} (.5,0) to[out=90,in=-90] (0,.5);
\end{tikzpicture}
= \pm
\begin{tikzpicture}[very thick,scale=1.5,baseline]
\draw (1.5,-.5) to[out=90,in=-90] node[at start,below]{$k$} (1.5,.5);
\draw (1,-.5) to[out=90,in=-90] node[at start,below]{${k\pm 1}$} (1,.5);
\draw (.5,-.5) to[out=90,in=-90]node[at start,below]{${k}$}  (.5,.5) ;
\draw[wei] (0,-.5) to[out=90,in=-90] node[at start,below]{${k}$}  (0,0) to[out=90,in=-90] (0,.5);
\end{tikzpicture}+
\begin{tikzpicture}[very thick,scale=1.5,baseline]
\draw (1.5,-.5) to[out=90,in=-90] node[at start,below]{$k$} (.5,.5);
\draw (1,-.5) to[out=90,in=-90] node[at start,below]{${k\pm 1}$} (.5,0) to[out=90,in=-90] (1,.5);
\draw (.5,-.5) to[out=90,in=-90] node[at start,below]{${k}$} (1.5,.5);
\draw[wei] (0,-.5) to[out=90,in=-90] node[at start,below]{${k}$}  (0,0) to[out=90,in=-90] (0,.5);
\end{tikzpicture}
\]
\[\label{tetris-red2}\subeqn\begin{tikzpicture}[very thick,scale=1.5,baseline]
\draw (1.5,-.5) to[out=90,in=-90] node[at start,below]{$k$} (0,0) to[out=90,in=-90] (1.5,.5) ;
\draw[wei] (1,-.5) to[out=90,in=-90] node[at start,below]{${k}$} (1.5,0) to[out=90,in=-90] (1,.5) ;
\draw (.5,-.5) to[out=90,in=-90] node[at start,below]{$k$} (1,0) to[out=90,in=-90] (.5,.5);
\draw (0,-.5) to[out=90,in=-90] node[at start,below]{$k\pm 1$} (.5,0) to[out=90,in=-90] (0,.5);
\end{tikzpicture}
= \pm
\begin{tikzpicture}[very thick,scale=1.5,baseline]
\draw (1.5,-.5) to[out=90,in=-90] node[at start,below]{$k$} (1.5,.5);
\draw[wei] (1,-.5) to[out=90,in=-90] node[at start,below]{${k}$} (1,.5);
\draw (.5,-.5) to[out=90,in=-90]node[at start,below]{${k}$}  (.5,.5) ;
\draw (0,-.5) to[out=90,in=-90] node[at start,below]{${k\pm 1}$}  (0,0) to[out=90,in=-90] (0,.5);
\end{tikzpicture}\pm
\begin{tikzpicture}[very thick,scale=1.5,baseline]
\draw (1.5,-.5) to[out=90,in=-90] node[at start,below]{$k$} (.5,.5);
\draw[wei] (1,-.5) to[out=90,in=-90] node[at start,below]{${k}$} (.5,0) to[out=90,in=-90] (1,.5);
\draw (.5,-.5) to[out=90,in=-90] node[at start,below]{${k}$} (1.5,.5);
\draw (0,-.5) to[out=90,in=-90] node[at start,below]{${k\pm 1}$}  (0,0) to[out=90,in=-90] (0,.5);
\end{tikzpicture}
\]
In order to calculate the LHS of (\ref{split-bigon}), we use the
equations above.  To organize this calculation, recall that the
strands that interact with one labeled $k$ correspond to boxes of
content $k,k\pm 1$ in a $a\times b$ box.  We can think about computing
the left hand side as a process of killing bigons and stripping off
the corresponding boxes.  The equations
(\ref{tetris-long}--\ref{tetris-red2}) show how this process can be
systematized into removing chunks which look like 3-box tetris
pieces.  In each case, the second term on the RHS contributes 0
because its product with $C_{\sR,\sR}$ is 0 by (\ref{Lbc-relations}).  

These pieces can be categorized as positive or negative
depending on how they affect the sign of the remaining term:
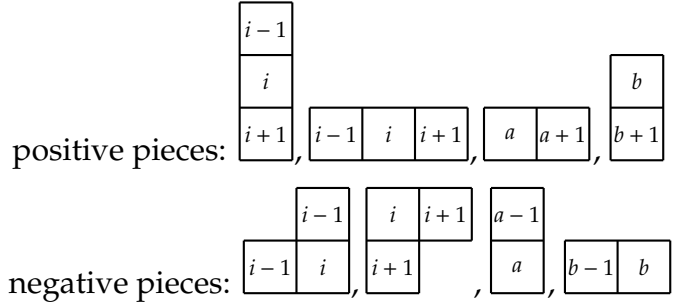
\[\text{positive pieces: }\tikz[thick,scale=.7,baseline]{
\draw (0,0) -- (1,0);
\draw (0,1) -- (1,1);
\draw (0,2)  -- (1,2);
\draw (0,3) -- (1,3);
\draw (0,0) -- (0,3);
\draw (1,0) -- (1,3); 
\node[scale=.7] at (.5,.5){$i+1$};
\node[scale=.7] at  (.5,1.5){$i$};
\node[scale=.7] at (.5,2.5){$i-1$};}, 
\tikz[thick,scale=.7,baseline]{\draw (0,0) -- (3,0);
\draw (0,1) -- (3,1);
\draw (2,0)-- (2,1);
\draw (3,0) -- (3,1);
\draw (0,0) -- (0,1);
\draw (1,0) -- (1,1); 
\node[scale=.7] at (2.5,.5){$i+1$};
\node[scale=.7] at (1.5,.5){$i$};
\node[scale=.7] at (.5,.5){$i-1$};},
\tikz[thick,scale=.7,baseline]{\draw (0,0) -- (2,0);
\draw (0,1) -- (2,1);
\draw (2,0)-- (2,1);
\draw (0,0) -- (0,1);
\draw (1,0) -- (1,1); 
\node[scale=.7] at (1.5,.5){$a+1$};
\node[scale=.7] at (.5,.5){$a$};},
\tikz[thick,scale=.7,baseline]{
\draw (0,0) -- (1,0);
\draw (0,1) -- (1,1);
\draw (0,2)  -- (1,2);
\draw (0,0) -- (0,2);
\draw (1,0) -- (1,2); 
\node[scale=.7] at (.5,.5){$b+1$};
\node[scale=.7] at  (.5,1.5){$b$};}\]
\[\text{negative pieces: }\tikz[thick,scale=.7,baseline]{
\draw (0,0) -- (2,0);
\draw (0,1) -- (2,1);
\draw (2,2)  -- (1,2);
\draw (2,0) -- (2,2);
\draw (0,0) -- (0,1);
\draw (1,0) -- (1,2); 
\node[scale=.7] at (.5,.5){$i-1$};
\node[scale=.7] at  (1.5,.5){$i$};
\node[scale=.7] at (1.5,1.5){$i-1$};}, 
\tikz[thick,scale=.7,baseline]{\draw (0,0) -- (1,0);
\draw (0,1) -- (2,1);
\draw (2,2)-- (2,1);
\draw (0,2) -- (2,2);
\draw (0,0) -- (0,2);
\draw (1,0) -- (1,2); 
\node[scale=.7] at (1.5,1.5){$i+1$};
\node[scale=.7] at (.5,1.5){$i$};
\node[scale=.7] at (.5,.5){$i+1$};}, 
\tikz[thick,scale=.7,baseline]{
\draw (0,0) -- (1,0);
\draw (0,1) -- (1,1);
\draw (0,2)  -- (1,2);
\draw (0,0) -- (0,2);
\draw (1,0) -- (1,2); 
\node[scale=.7] at (.5,.5){$a$};
\node[scale=.7] at  (.5,1.5){$a-1$};
}, 
\tikz[thick,scale=.7,baseline]{\draw (0,0) -- (2,0);
\draw (0,1) -- (2,1);
\draw (2,0)-- (2,1);
\draw (0,0) -- (0,1);
\draw (1,0) -- (1,1); 
\node[scale=.7] at (1.5,.5){$b$};
\node[scale=.7] at (.5,.5){$b-1$};}
\]
Assume for simplicity that $a<b$; if $k< a$, then the boxes of
content $k,k\pm 1$ can broken into one negative $L$ and then $k-1$
positive pieces.  This matches
$\sigma(k,a,b)=-1$.   

If $k=a$, we can use the postive piece with $a$ at
the corner, and then positive pieces.  
If $a<k\leq b$, then we can use all positive pieces.  This matches
$\sigma(k,a,b)=1$.  If $b<k<c$, then as when $k<a$, we use one $L$
and then positive pieces, matching
$\sigma(k,a,b)=-1$.  If $k>c$, then, there is no interaction
between any strands, matching
$\sigma(k,a,b)=1$.   
 
The remaining case of $k=c$ follows
from \ref{black-bigon}; the $c$-labeled strand only interacts with the
sole strand labeled $c-1$, producing the desired dot on the
strand labeled $c$.  The other term, that with a dot on the strand
labeled $c-1$, is 0  since any dot kills $L_{a,b}$.  

The ultimate result is a series
  of relations, all of which keep the new strands above the $Y$, which
  shows that this element of the kernel is 0.
\end{proof}

The right $\tU^-_n$-equivariant structure also induces a $\tU^-_n$-equivariant structure
on the ladder bimodules over  $T$. Let $\phi\colon \eF_jW_{Y_i}\to W_{Y_i}\eF_j$ denote the commutation isomorphism. Note that $\eE_jW_{Y_i}$ and $W_{Y_i}\eE_j$ are both adjoints of $W_{Y_i^\star}\eF_j\cong 
\eF_jW_{Y_i^\star}$ and therefore have to be naturally isomorphic. The unit and counit we pick are of course 
the ones induced by the units and counits of the adjunctions between $W_{Y_i}$ and $W_{Y_i^\star}$, and 
$\eE_j$ and $\eF_j$. Since the unit and counit of the first adjunction are $\tU^-_n$ equivariant, the induced units and counits for $\eE_jW_{Y_i}\cong W_{Y_i}\eE_j$ and $W_{Y_i^\star}\eF_j\cong \eF_jW_{Y_i^\star}$ are independent 
of the possible choices in their definition and they are compatible with the commutation isomorphisms. 
By the zig-zag relation for the unit and counit of $W_{Y_i}$ and $W_{Y_i^\star}$, we can simplify the definition of the commutation map $\phi'\colon  W_{Y_i}\eE_j\to \eE_jW_{Y_i}$ a bit. It is given by the composite
$$ W_{Y_i}\eE_j\to \eE_j\eF_jW_{Y_i}\eE_j\cong \eE_jW_{Y_i}\eF_j\eE_j\to  \eE_jW_{Y_i}.$$ 
This isomorphism is the "vertical reflection" of the commutation map $\eF_jW_{Y_i^*}\to W_{Y_i^*}\eF_j$.
\begin{lemma}
  The maps $\phi$ and $(\phi')^{-1}$ define an equivariant structure
  over $\tU_n$. 
\end{lemma}
\begin{proof}
  Since the morphisms in $\tU_n$ are generated by the KLR
  endomorphisms of $\eF^n$, the unit and counit of the adjunction
  $(\eF_j,\eE_j)$ and the inverse of a map build of these.  Thus
  commutation with the KLR endomorphisms and the unit and counit
  suffice.  The former follows by $\tU_n^-$-equivariance, and the
  latter is automatic from the definition of $\phi'$.  
\end{proof}
Note that from the definition of $\phi'$ above, it is also clear that if a given natural map between two ladder 
bimodules commutes with $\tU^-$, then it commutes with the whole $\tU$. 

 The bimodule $e_{\Bi;\kappa}W_{Y_i}e_{\Bj;\kappa'}$ has a basis in bijection with
  quadruples consisting of
  \begin{enumerate}
  \item a permutation $\pi\in S_m$ acting on the entries of $\Bj$,
   \item a vector $\mathbf{a}\in \Z_{\geq 0}^m$ 
   \item a tableau $ \sS$ on a $b\times c$
    rectangle together with 
\item a shuffle $\sigma$ of $\pi\cdot \Bj$ and
    the content word of $\sS$ which is equal to $\Bi$, such that all entries
    from the content word lie in the segment
    $[\kappa(i)+1,\kappa(i+1)]$. 
\end{enumerate}

We construct a basis vector from this data by taking the basis vector of
$L_{a,b}$ associated to the tableau in the box, connecting the strands
from the top of the box to $y=1$ without crossings, and then adding in
strands starting at $y=0$ corresponding to $\Bj$.  The strands trace
out a string diagram of the permutation $\pi$, and the interlace with
those from the top of the box according to the shuffle $\sigma$ while
creating a minimal number of crossings of all types.  We
then multiply at the bottom by $y_1^{a_1}\cdots y_m^{a_m}$. 
There is not a unique such diagram, but for each set of data, we
simply choose one. 

For $W_{Y_i^\star}$ we obtain a similar basis by reflecting the basis elements 
above in a horizontal axis. 
  
\begin{proposition}
The diagrams as above for all appropriate triples form a basis of
$W_{Y_i},W_{Y_i^\star}$.
\end{proposition}
\begin{proof}
  Consider $W_{Y_i}e_{\Bi,\kappa}$, the other case can be proved similarly. 
We need to show that this has
  the correct basis. 
If  $\Bi=\emptyset$ and
  $\ell=1$, then this is just the basis of $L_{a,b}$ we already know
  and the result holds. Otherwise,
 by Proposition \ref{prop:trivalent-commute},  the module
 $W_{Y_i}e_{\Bi,\kappa}$ is an induction, and the result follows since an
 induction has a basis given by all ways of shuffling together basis
 vectors of the two modules.  
\end{proof}
This basis further shows that $e_{\Bi,\kappa}W_{Y_i}$ is projective as
a right module:
\begin{corollary}
  The right module $e_{\Bi,\kappa}W_{Y_i}$ is isomorphic to the sum of
  $e_{\Bi',\kappa'}\tilde{T}$ with multiplicity given by the number of
  tableaux $\sS$   and shuffles $\sigma$ of $\Bi'$ with
    the content word of $\sS$ which are equal to $\Bi$.
\end{corollary}

\subsection{The case of \texorpdfstring{$a=1$}{a=1}}
\label{sec:case-a=1}

We'll consider further the structure of
$\tilde{T}^{1,c-1}_{\omega_c}$.  Every sequence that appears in this
algebra is a permutation of $(\omega_1,1,\dots,c-1,\omega_{c-1})$.
Note that the space of degree 0 endomorphisms of such a projective is
1-dimensional in each case, so these projectives are all
indecomposable.  Conversely, every indecomposable is necessarily of this form.
Furthermore, if two such sequences differ by permuting symbols which
are not consecutive in the sequence above, they are isomorphic by
(\ref{black-bigon}) or \eqref{cost}.  

Consider the collection of proper
subsets $S\subsetneq [1,c]$.  Let $w_S$ be the longest permutation of
$[0,c]$ such that we have $m<k$ and  $w_S(m)>w_S(k)$ if and only if $[m+1,k]\in S$.  Let $P_{S}$ be the
projective over $\tilde{T}^{1,c-1}$ associated to the unique sequence
obtained by applying $w_S$ to
$(\omega_1,1,\dots,c-1,\omega_{c-1})$.  
\begin{lemma}
  The projectives $P_S$ are a complete irredundant collection of
  indecomposable projectives for
  $\tilde{T}^{1,c-1}$.
\end{lemma}
\begin{proof}
  We've already observed that every indecomposable projective must be
  isomorphic to one of these, by a diagram crossing strands labeled
  with distant roots.

  On the other hand, if $S\neq S'$, there are no diagrams of degree 0
  joining them, so the corresponding projectives are not isomorphic.  
\end{proof}

Let $S\triangle S'=(S\cup S')\setminus(S\cap S')$ be the operation of
symmetric difference. 
For every $k$, we have a unique dotless diagram $D_{S,k}$ with a minimal number
of 
crossings with bottom $P_S$ and top $P_{S\triangle\{k\}}$.  One can
easily check that this diagram has degree 1.  We'll define
 a map
$x_{S,k}\colon P_{S}\to P_{S\triangle \{k\}}$ given by $D_{S,k}$
unless $k=c\in S$, in which case we take $-D_{S,k}$. 
One can easily calculate that:
\begin{lemma}
For all $k,m\in [1,c]$, we have 
  \begin{equation}
    x_{S\triangle\{k\},m}x_{S,k}=x_{S\triangle\{m\},k}x_{S,m} \qquad
    \sum_{k}x_{S\triangle\{k\},k}x_{S,k}=0\label{eq:2}
  \end{equation}
\end{lemma}
\begin{proof}
  The products $x_{S\triangle\{k\},m}x_{S,k}$ and
  $x_{S\triangle\{m\},k}x_{S,m}$ are both the unique diagram with a
  minimal number of crossings and no dots from $P_S$ to
  $P_{S\triangle\{k,m\}}$, unless $k=c$ or $m=c$, in which case both are
  the negatives of this diagram. In either case, they are equal.

The product $x_{S\triangle\{k\},k}x_{S,k}$ is the identity on $P_S$
times a dot on the strand labeled $k+1$ minus a dot on that labeled
$k$ if $0<k<c$.  If $k=c$, we only get the negative term due to the
sign change $x_{S,k}=-D_{S,k}$, and if
$k=0$ only the positive.  Thus, in the sum of (\ref{eq:2}), all terms
appear twice with opposite signs, and we get that the sum is 0.  
\end{proof}

\begin{proposition}\label{prop:endomorphisms}
  The endomorphism algebra $A=\End(\oplus_SP_S)$ is a quadratic algebra
  generated by the idempotents $e_S$, the degree 1 morphisms
  $x_{S,k}$, and with all relations given by \eqref{eq:2}.

The algebra $A$ is isomorphic to the algebra
$A^!_{\operatorname{pol}}(\eta,-)$ defined in \cite[\S 8.5]{BLPWtorico} 
corresponding to the intersection of the coordinate arrangement in
$\R^c$ with the affine subspace $\{(z_1,\dots, z_n)\mid z_1+\cdots
+z_n=1\}$. 
\end{proposition}
\begin{proof}
First, note that the morphisms $x_{S,k}$ generate $A$; obviously, any
diagram without dots can be factored into degree 1 morphisms, which
must correspond to a factorization as $x_{S,k}$.  Thus, we need only
show we can get all dots.  However, the products
$x_{S\triangle\{k\},k}x_{S,k}$ span the degree 2 dots, with one
relation (given by \eqref{eq:2}).  

  The equations \eqref{eq:2} correspond to the relations of
  $A^!_{\operatorname{pol}}(\eta,-)$ as described in \cite[\S
  3.3]{BLPWtorico}. The first set to the relation that the length two
  paths between $\al$ and $\be$ are the same, and the second to the
  relation killing the image of $\vartheta(\mathfrak{t})$, since in
  this case, $ \mathfrak{t}$ corresponds to the diagonal
  $\C^*$-action.  The fact that we only consider proper subsets
  matches the fact that the subspace $\{(z_1,\dots, z_n)| z_1+\cdots
+z_n=1\}$ misses one of the chambers of $\R^n$, that where all
$z_i<0$.  

Thus, we have an algebra map $A^!_{\operatorname{pol}}(\eta,-) \to
A$.  This map is surjective since the $x_{S,k}$ generate.  On the
other hand, $A$ and  $A^!_{\operatorname{pol}}(\eta,-)$ have the same
Hilbert series since in both cases, the morphism space between two
projectives is a free module over polynomials in $c-1$ variables
generated by a morphism of minimal degree equal to the length of the
shortest path between the two idempotents.  
\end{proof}

By \cite[8.25]{BLPWtorico}, this shows that $A$ is Koszul, and we can
construct a canonical projective resolution of every simple using the
Koszul resolution from \cite{BGS96}.  Let $B$ be the quadratic dual of
$A$.  The Koszul resolution of the simple quotient of $P_S$ has the
form \[\cdots \to A\otimes_{A_0}B^*_2e_S \to A\otimes_{A_0}B^*_1e_S \to
A\otimes_{A_0}B^*_0e_S.\]  If we identify $B^*_n$ with 
$A_1\otimes_{A_0}\cdots
\otimes_{A_0}R\otimes_{A_0}\cdots \otimes_{A_0}A_1$ where $R\subset
A_1\otimes_{A_0}A_1$ is the span of the relations \eqref{eq:2}, then
the differential just becomes the usual Hochschild differential.  

Let $x'_{S,k}=(-1)^{\#(S\cap [1,k-1])}x_{S,k}$.  A {\bf wall path} of length
$m$ from $S$ to $S'$ is a sequence $S'=S_0,S_1,\dots,
  S_m=S$ such that $S_p=S_{p-1}\triangle\{k_p\}$ for some sequence
  $k_1,\dots, k_m$.
  \begin{proposition}
    The space $e_{S'}B^*e_S$ has a basis given by the sums  \[b_m=\sum_{S=S_0,S_1,\dots,
  S_m=S'}x'_{S_{m-1},k_m}\otimes \cdots \otimes x'_{S_{0},k_1}\] 
over wall paths of length $m=\#(S\triangle S')+2p$ for some $0\leq p<c-\#(S\cup S')$
  \end{proposition}
  \begin{proof}
    By \cite[8.25]{BLPWtorico}, the dimension of $e_{S'}B^*_0e_S$ is
    $c-\#(S\cup S')$, the number of vertices of the intersection of
    the corresponding closed chambers.  Thus, we need only show that $b_m$ lies in this
    space (since these are obviously non-zero and linearly independent).  

    Of course, a wall paths of length $m$ from $S$ to $S'$ can be
    factored into paths from  $S$ to $S_1$ of length $g$, from $S_1$
    to $S_2$ of length 2, and from $S_2$ to $S'$ of length $m-g-2$.
    Fixing the first and last of these paths, we can break $b_m$ into
    a sum over pairs $S_1$ and $S_2$ and length two paths from $S_1$
    to $S_2$.  If $S_1=S_2$, then we obtain $
    \sum_{k}x_{S_1\triangle\{k\},k}'\otimes
    x_{S_1,k}'=\sum_{k}x_{S_1\triangle\{k\},k}\otimes x_{S_1,k}$ which
    lies in $R$ by \eqref{eq:2}.  If $S_1\neq S_2$ then as long as $S_1\cup S_2\neq
    [1,c]$, we arrive at
    \[x'_{S\triangle\{k\},k'}x'_{S,k}+x'_{S\triangle\{k'\},k}x'_{S,k'}=\pm
    (x_{S\triangle\{k\},k'}x_{S,k}-x_{S\triangle\{k'\},k}x_{S,k'})\]
    where $\{k,k'\}=S_1\triangle S_2$.  This is in $R$ by
    \eqref{eq:2}, completing the proof.  Note that we need an upper
    bound on $m$ precisely to avoid the case where  $S_1\cup S_2=
    [1,c]$.
  \end{proof}

  If $S=\emptyset$, then this is particularly simple;
  $e_{S'}B^*e_{\emptyset}$ is spanned by $b_m$ for $m=\#S',\#S'+2,\dots
  ,2c-2-\#S'$.  Let $Q_{-k}$ be the sum of $P_{S'}(-k)$ for all $S'$ such
  that $\#S'=k\pmod 2$ and $\#S'\leq k\leq 2c-2-\#S'$.  We have a differential $Q_{-k}\to Q_{-k+1}$
  where the component $P_S\to P_{S\triangle\{k\}}$ is just $x_{S,k}'$.
\begin{corollary}\label{W-res}
  As a left module $L_{1,c-1}$ has a resolution as a left module of the form:
\[Q_{-(2c-2)}\cong P_{\emptyset}(-2c+2)\to \cdots \to Q_{-k}\to \cdots \to Q_{-1}\cong
\bigoplus_{s\in [1,c]} P_{\{s\}}\to Q_0
\cong P_{\emptyset}.\]
\end{corollary}

For example, if $c=2$, then the resolution is 3 step, of the form:
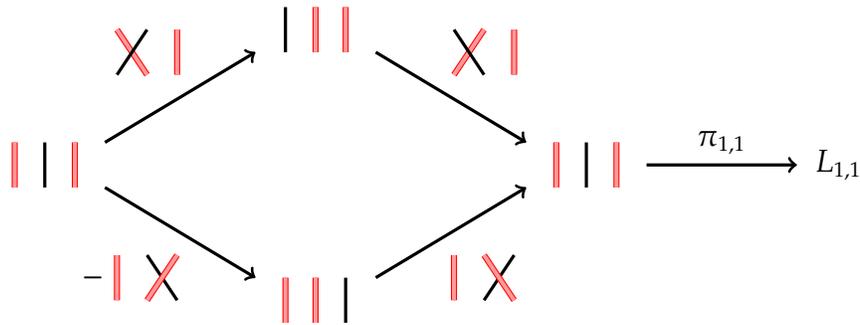
\begin{center}
\tikz[xscale=.8, yscale=.6,very thick]{
\draw[wei] (0,0)--(0,1); \draw[very thick] (.5,0) --(.5,1);\draw[wei] (1,0)--(1,1);
\node at (1.3,-2){$-$}; \draw[wei] (1.7,-1.5)--(1.7,-2.5); \draw[very thick] (2.7,-2.5) --(2.2,-1.5);
\draw[wei] (2.7,-1.5)--(2.2,-2.5);
\draw[->] (1.5,0)--(4,-2);
\draw[wei] (4.5,-3)--(4.5,-2); \draw[wei] (5,-3) --(5,-2);\draw[very thick] 
(5.5,-3)--(5.5,-2);
\draw[wei] (1.7,3.5)--(2.2,2.5); \draw[very thick] (1.7,2.5) --(2.2,3.5);\draw[wei] 
(2.7,2.5)--(2.7,3.5);
\draw[wei] (7.8,3.5)--(7.3,2.5); \draw[very thick] (7.3,3.5) --(7.8,2.5);\draw[wei] 
(8.3,2.5)--(8.3,3.5);
\draw[->] (1.5,1)--(4,3);
\draw[very thick](4.5,3)--(4.5,4);\draw[wei] (5,3)--(5,4); \draw[wei] 
(5.5,3)--(5.5,4);
\draw[->] (6,3)--(8.5,1);
\draw[wei] (7.3,-1.5)--(7.3,-2.5); \draw[very thick] (8.3,-1.5)--(7.8,-2.5);\draw[wei] 
(7.8,-1.5)--(8.3,-2.5);
\draw[->] (6,-2)--(8.5,0);
\draw[wei] (9,0)--(9,1); \draw[very thick] (9.5,0)--(9.5,1);\draw[wei] 
(10,0)--(10,1);
\draw[->] (10.5,0.5)--(13,0.5);
\node at (13.7,0.5) {$L_{1,1}$};
\node at (11.75,1) {$\pi_{1,1}$};
}
\end{center}
where 
\vskip0.5cm
\begin{center}
\tikz[xscale=.8, yscale=.6]{
\draw[wei] (0,0)--(0,1); \draw[very thick] (.5,0) --(.5,1);\draw[wei] (1,0)--(1,1);
\node at (1.7,0.5) {$=P_{\emptyset}$};}\qquad
\tikz[xscale=.8, yscale=.6]{\draw[very thick] (0,0)--(0,1); \draw[wei] (.5,0) --(.5,1);\draw[wei] (1,0)--(1,1);
\node at (1.9,0.5) {$=P_{\{1\}}$};}\qquad 
\tikz[xscale=.8, yscale=.6]{\draw[wei] (0,0)--(0,1); \draw[wei] (.5,0) --(.5,1);\draw[very thick] (1,0)--(1,1);
\node at (1.9,0.5) {$=P_{\{2\}}$};}
\end{center}

For $c=3$, we have a slightly more complicated example, which is illustrative:
 
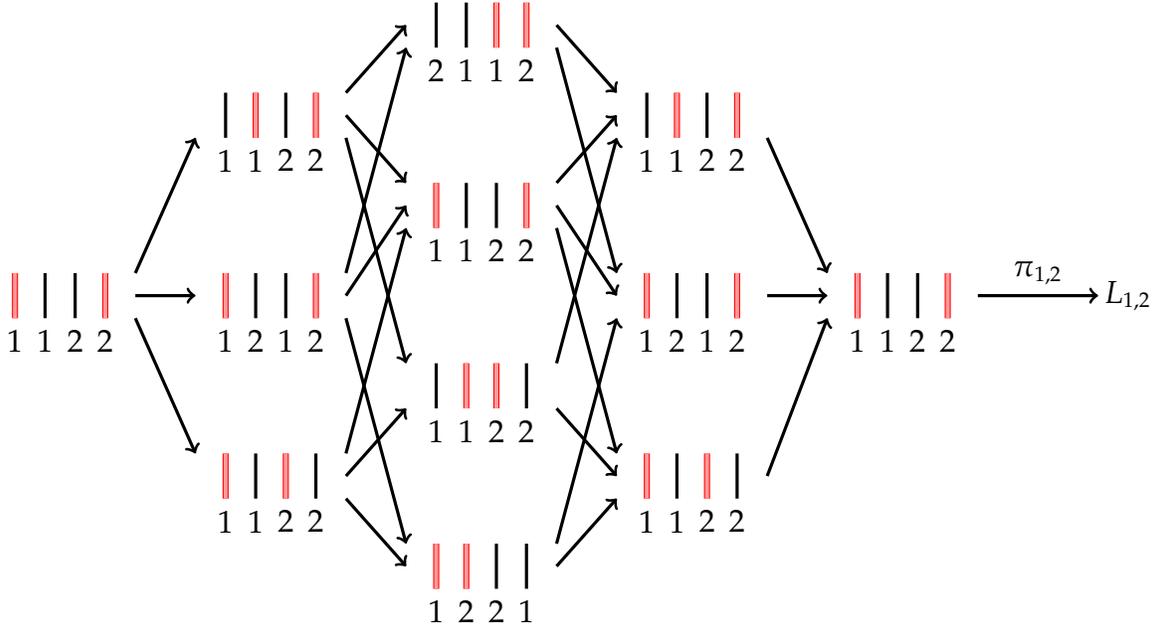
\begin{center}
\tikz[xscale=.8, yscale=.6,very thick]
{
\draw[wei] (0,0)--(0,1); \draw[very thick] (.5,0) --(.5,1);\draw[very thick] (1,0)--(1,1); \draw[wei] (1.5,0)--(1.5,1);
\node at (0,-0.5) {1};
\node at (.5,-0.5) {1};
\node at (1,-0.5) {2};
\node at (1.5,-0.5) {2};
\draw[->] (2,0)--(3,-3);
\draw[wei] (3.5,-4)--(3.5,-3); \draw[very thick] (4,-4)--(4,-3); \draw[wei] (4.5,-4) --(4.5,-3);\draw[very thick] (5,-4)--(5,-3);
\node at (3.5,-4.5) {1};
\node at (4,-4.5) {1};
\node at (4.5,-4.5) {2};
\node at (5,-4.5) {2};
\draw[->] (2,1)--(3,4);
\draw[very thick](3.5,4)--(3.5,5);\draw[wei](4,4)--(4,5);\draw[very thick] (4.5,4)--(4.5,5); \draw[wei] (5,4)--(5,5);
\node at (3.5,3.5) {1};
\node at (4,3.5) {1};
\node at (4.5,3.5) {2};
\node at (5,3.5) {2};
\draw[->] (2,0.5)--(3,0.5);
\draw[wei] (3.5,0)--(3.5,1); \draw[very thick] (4,0) --(4,1);\draw[very thick] (4.5,0)--(4.5,1); \draw[wei] (5,0)--(5,1);
\node at (3.5,-0.5) {1};
\node at (4,-0.5) {2};
\node at (4.5,-0.5) {1};
\node at (5,-0.5) {2};
\draw[->] (5.5,5)--(6.5,6.5);
\draw[->] (5.5,4.5)--(6.5,3);
\draw[->] (5.5,4)--(6.5,-1);
\draw[wei] (7,-6)--(7,-5); \draw[wei] (7.5,-6)--(7.5,-5); \draw[very thick] (8,-6) --(8,-5);\draw[very thick] (8.5,-6)--(8.5,-5);
\node at (7,-6.5) {1};
\node at (7.5,-6.5) {2};
\node at (8,-6.5) {2};
\node at (8.5,-6.5) {1};
\draw[very thick](7,6)--(7,7);\draw[very thick](7.5,6)--(7.5,7);\draw[wei] (8,6)--(8,7); \draw[wei] (8.5,6)--(8.5,7);
\node at (7,5.5) {2};
\node at (7.5,5.5) {1};
\node at (8,5.5) {1};
\node at (8.5,5.5) {2};
\draw[wei] (7,2)--(7,3); \draw[very thick] (7.5,2) --(7.5,3);\draw[very thick] (8,2)--(8,3); \draw[wei] (8.5,2)--(8.5,3);
\node at (7,1.5) {1};
\node at (7.5,1.5) {1};
\node at (8,1.5) {2};
\node at (8.5,1.5) {2};
\draw[very thick] (7,-2)--(7,-1); \draw[wei] (7.5,-2) --(7.5,-1);\draw[wei] (8,-2)--(8,-1); \draw[very thick] (8.5,-2)--(8.5,-1);
\node at (7,-2.5) {1};
\node at (7.5,-2.5) {1};
\node at (8,-2.5) {2};
\node at (8.5,-2.5) {2};
\draw[wei] (10.5,-4)--(10.5,-3); \draw[very thick] (11,-4)--(11,-3); \draw[wei] (11.5,-4) --(11.5,-3);\draw[very thick](12,-4)--(12,-3);
\node at (10.5,-4.5) {1};
\node at (11,-4.5) {1};
\node at (11.5,-4.5) {2};
\node at (12,-4.5) {2};
\draw[very thick] (10.5,4)--(10.5,5);\draw[wei](11,4)--(11,5);\draw[very thick] (11.5,4)--(11.5,5); \draw[wei] (12,4)--(12,5);
\node at (10.5,3.5) {1};
\node at (11,3.5) {1};
\node at (11.5,3.5) {2};
\node at (12,3.5) {2};
\draw[wei] (10.5,0)--(10.5,1); \draw[very thick] (11,0) --(11,1);\draw[very thick] (11.5,0)--(11.5,1); \draw[wei] (12,0)--(12,1);
\node at (10.5,-0.5) {1};
\node at (11,-0.5) {2};
\node at (11.5,-0.5) {1};
\node at (12,-0.5) {2};
\draw[wei] (14,0)--(14,1); \draw[very thick] (14.5,0) --(14.5,1);\draw[very thick] (15,0)--(15,1); \draw[wei] (15.5,0)--(15.5,1);
\node at (14,-0.5) {1};
\node at (14.5,-0.5) {1};
\node at (15,-0.5) {2};
\node at (15.5,-0.5) {2};
\draw[->] (5.5,-3)--(6.5,2);
\draw[->] (5.5,-3.5)--(6.5,-2);
\draw[->] (5.5,-4)--(6.5,-5.5);
\draw[->] (5.5,0)--(6.5,-5);
\draw[->] (5.5,0.5)--(6.5,2.5);
\draw[->] (5.5,1)--(6.5,6);
\draw[->] (9,-5.5)--(10,-4);
\draw[->] (9,-5)--(10,0);
\draw[->] (9,-2)--(10,-3.5);
\draw[->] (9,-1)--(10,4);
\draw[->] (9,2)--(10,-3);
\draw[->] (9,2.5)--(10,0.5);
\draw[->] (9,3)--(10,4.5);
\draw[->] (9,6)--(10,1);
\draw[->] (9,6.5)--(10,5);
\draw[->] (12.5,4)--(13.5,1);
\draw[->] (12.5,0.5)--(13.5,0.5);
\draw[->] (12.5,-3.5)--(13.5,0);
\draw[->] (16,0.5)--(18,0.5);
\node at (18.5,0.5) {$L_{1,2}$};
\node at (17,1) {$\pi_{1,2}$};
}
\end{center}
Reading from top to bottom, we have: in the first and the last column
$P_{\emptyset}$; in the second and the fourth column, the sum
$P_{\{1\}}\oplus P_{\{2\}}\oplus P_{\{3\}}$; and in the middle column $P_{\{1,2\}}\oplus P_{\emptyset}\oplus P_{\{1,3\}}\oplus P_{\{2,3\}}$.  
The differentials are defined by the diagrams $\pm D_{S,k}$.

Using induction, this is easily extended to a resolution of $W_{Y_i}$ as a
left module.  As explained in Sect.~\ref{sec:a_infty-algebras}, this resolution has a right $A_\infty$ action, whose existence is important to us but not its precise definition, which we therefore omit. 

\begin{lemma}\label{a-b-1}
If $a=1$ or $b=1$, we have \[\RHom_{\tilde{T}^{\Bp'}
  }(W_{Y_i},\tilde{T}^{\Bp'})\cong W_{Y_i^\star}\langle c-1\rangle\qquad
  \RHom_{\tilde{T}^{\Bp}
  }(W_{Y_i^\star},\tilde{T}^{\Bp})\cong W_{Y_i}\langle 1-c\rangle.\]  This defines adjunction
  maps \[\ep_{Y_i}\colon W_{Y_i}\Lotimes_{\tilde{T}^{\Bp}}
  W_{Y_i^\star}\langle c-1\rangle\to \tilde{T}^{\Bp'}\qquad
  \iota_{Y_i}\colon \tilde{T}^{\Bp}\to  W_{Y_i^\star}\Lotimes_{\tilde{T}^{\Bp'}}
  W_{Y_i}\langle c-1\rangle\]
and 
\[\ep_{Y_i^\star}\colon W_{Y_i^\star}\Lotimes_{\tilde{T}^{\Bp'}} W_{Y_i}\langle 1-c\rangle\to \tilde{T}^{\Bp}\qquad
  \iota_{Y_i^\star}\colon \tilde{T}^{\Bp'}\to  W_{Y_i}\Lotimes_{\tilde{T}^{\Bp}} W_{Y_i^\star}\langle 1-c\rangle\]
\end{lemma}

\begin{proof}
In both cases, we can reduce to assuming that there are no red or black
strands but those coming from the trivalent vertex.  Also, the cases
$a=1,b=c-1$ and $a=c-1,b=1$ are completely symmetric, so we can assume
we are in the former case.

 Applying $\RHom_{\tilde{T}^{\Bp}
  }(-,\tilde{T}^{\Bp})$  to the resolution of Lemma \ref{W-res}
  we get $\Hom(Q_k,\tilde{T}^{\Bp})\cong
  Q_{2c-2-k}\langle 2-2c\rangle$.  The part of the differential $P_{S}\to
  P_{S\triangle \{k\}}$ is given by $x'_{S,k}$ and its dual is the
  mirror image of this through the $x$-axis, which is $x'_{S,k}$
  (unless $k=0$, in which case we must fix the sign)
and thus this isomorphism matches $\partial_k^{*}$ with
  $\partial_{2c-k}$. 
Thus, we arrive at the same complex, but with degrees shifted upward by $2c-2$, since taking $\RHom$ negates
  the degrees of the complex.

On the other hand, as a left module $Y_i^\star$ is projective, so its
dual is again a bimodule.  In fact, we simply take the vector space dual of
$\dot{L}_{a,b}$, the same underlying simple but
with the action switched from left to right using reflection.  This
dual is just $L_{a,b}$.
\end{proof}

Thus, we have three functors, which are isomorphic up to shift, but not
canonically so: $Y_i^\star$, the left adjoint $Y_i^L\langle c-1\rangle$ and right
adjoint $Y_i^R\langle 1-c\rangle$.  

In order to fix an isomorphism $Y_i^\star\cong Y_i^R\langle
1-c\rangle$, we should define a map $\iota_{Y_i}\colon \operatorname{id}\to Y_i^\star
Y_i\langle c-1\rangle$ which we identify with the unit of the
adjunction. We let $\epsilon_{Y_i}$ be the induced counit, which is unique. Similarly, an isomorphism $Y_i^\star\cong Y_i^L\langle
1-c\rangle$ will be fixed by choosing a map $\epsilon_{Y_i^\star}\colon Y_i^\star
Y_i\langle 1-c\rangle \to\operatorname{id}$, which we match with the
counit of the other adjunction, and let $\iota_{Y_i^\star}$ be the
unique induced unit.  

\begin{lemma}\label{lem:full-length}
 If  $a=1$ (resp. $b=1$), then $\ssy_i^{b}$ (resp. $\ssy_{i+1}^{a}$)
 induces an isomorphism from the 
 the $-c+1$st homological degree to the $c-1$st homological degree. 
\end{lemma}
\begin{proof}
  We'll again apply Lemma \ref{Hochschild-action}.  Assume that
  $a=1$.  
We realize $Y_{i}^\star \Lotimes_{\tilde{T}^{\Bp_{i}}} Y_{i}$ as the tensor product
  of two copies of the resolution Corollary \ref{W-res}.  By general
  results, this is an $A_\infty$-bimodule over $ \tilde{T}^{\Bp'}$;
  the structure of the higher products will be irrelevant for us.

There is one homotopy representative of $\iota_{Y_{i}}$ which
  sends $1\mapsto X_{2c-2}\otimes X_0+\cdots +X_0\otimes
  X_{2c-2}$, where $X_0$ is the usual generator of $Q_0$, and
  $X_{2c-2}$ the usual generator of $Q_{2c-2}$.

We consider the precomplex over the deformation associated to $\ssy_i$
given by the same diagrams as the complex $Q_{2c-2}\to \cdots \to
Q_0$.  The differential is no longer 0.  Instead, the map
$\tilde{\partial}^2/h$ induces the map $Q_{j}\to Q_{j+2}$  which acts
by the identity on $P_S$ if $\#S\leq j-2<j\leq 2p_{i}-\#S$, and 0 on
all other summands.  Thus, the $p_{i}$th power of this chain map is just
the identity map on $Q_{2c-2}\cong P_\emptyset\cong Q_0$.  Thus,
$(\ssy_i^{p_{i}}\otimes 1)\iota_{Y_{i}}$ maps $1\mapsto X_0\otimes X_0 $.
The result then follows from the fact that $\ep_{Y_{i}^\star}(X_0\otimes X_0)$ is a non-zero scalar, 
which by definition we take equal to $1$. Note that this uniquely determines $\ep_{Y_{i}^\star}$. 

In the case where   $b=1$, the argument is almost the same,
except that the signs in the definition of the complex $Q_\bullet$ are
slightly different, since $x_{S,c}$ is {\it minus} the
corresponding diagram if $c\in S$.  Thus, $\tilde{\partial}^2/h$
gives the map $Q_{j}\to Q_{j+2}$ given by multiplication by $-1$. This also gives an
isomorphism, so it completes the proof.
\end{proof}

\excise{Thus, we can fix the sign of the biadjunction of $Y_i$ and
$Y_i^\star$ utilizing the isomorphism in Lemma~\ref{lem:full-length}. If  $a=1$, then we choose
these maps so that $\ep_{Y_i^\star}\ssy_i^{b}\iota_{Y_i}=1
,$  and if
 $b=1$, then we choose them so that $\ep_{Y_i^\star}\ssy_{i+1}^{a}\iota_{Y_i}=(-1)^a$.}   

From the proof of Lemma~\ref{lem:full-length} we immediately get the following result. 

\begin{lemma}\label{lem:bubble-dot} We can fix the biadjunction of $Y_i$ and
$Y_i^\star$ utilizing the isomorphism in Lemma~\ref{lem:full-length}, such that: 
 if  $a=1$, then \[\ep_{Y_{i}^\star}(\ssy_{i}^{q}\otimes 1)\iota_{Y_{i}}=
 \begin{cases}
   0& q<b\\
   1&q=b
 \end{cases}
;\] if $b=1$, then \[\ep_{Y_i^\star}(1\otimes \ssy_{i+1}^{q})\iota_{Y_i}=\begin{cases}
   0& q<a\\
  (-1)^{a}&q=a
 \end{cases}.\]
\end{lemma}
As remarked above, this also fixes $\iota_{Y_i^\star}$ and $\ep_{Y_i}$ uniquely by the usual zig-zag relations 
for units and counits. 

Note, for the case where $a=b=1$, these
conditions define the same adjunction: as shown in the proof of Lemma~\ref{lem:full-length}, the
action of $\ssy_{i}$ and $-\ssy_{i+1}$ agree when we identify
$\wE_{i-1}$ and $\wF_{i}$ (both of which are $Y_i$ when thought of as
webs, rather than ladders).

Finally, we wish to note the categorical version of the bigon
relation:  let $\beta=Y^\star_i Y_i$ be the bigon ladder on the $i$th red
strand when $p_i=c$, and the sides of the bigon are labeled with $c-1$
and $1$.
\begin{proposition}\label{bigon}
  $W_\beta\cong \tilde{T}^{\Bp}\langle c-1\rangle\oplus \tilde{T}^{\Bp}\langle c-3\rangle\oplus\cdots \oplus 
  \tilde{T}^{\Bp}\langle 1-c\rangle$.
\end{proposition}
\begin{proof}
  This is immediate from Lemma \ref{W-res}.  As usual, we can reduce
  to the case where at the outside there are no black strands.  If we resolve $W_{Y_i}$
  as a left module, and then tensor with $Y^\star_i$, then the terms
  corresponding to $P_S$ for $S\neq 0$ all vanish, and we are left
  with a 1-dimensional contribution from $P_\emptyset$ in the degrees
  $c-1,\dots -(c-1)$ and thus isomorphic to $\tilde{T}^{\Bp}$ as a bimodule.
\end{proof}
Note that in particular, if $a=b=1$, this implies that
$\Ext^{-2}(W_\beta,W_\beta)\cong Z(\tilde{T}^{\Bp})(-2)$.  In
particular, the space of elements of degree $-2$ in this space is
1-dimensional, spanned by $ \iota\circ \ep$.

\subsection{General ladders}
\label{sec:general-ladders}

\begin{definition}
  The {\bf ladder bimodule} for a general ladder is the composition of the
 $\tilde{T}$-bimodules attached to a slicing of the ladder into trivalent ladders.  
\end{definition}
\excise{We can think of these bimodules as forming a 2-category whose 
\begin{itemize}
\item objects are sequences $\Bp$ of integers. 
\item 1-morphisms are ladders up to isotopy.
\item 2-morphisms are homomorphisms between ladder bimodules, either in
  the abelian category of bimodules over $\tilde{T}$ or in its derived
  category, the latter being more interesting for our purposes.   
\end{itemize}

}
\excise{\begin{lemma}
  For any ladder bimodule, the cohomology of the complex $W\Lotimes P$
  for an indecomposable projective $P$ is concentrated in all even or
  all odd degrees.
\end{lemma}
\begin{proof}
We'll prove this by induction on the number of trivalent vertices.  If
the highest trivalent vertex is an upward opening $Y$, that is, $Y_i$
in our notation of the last section, then the cohomology will stay
concentrated in whatever degrees it was in before.  If instead, it is a $Y_i^\star$
\end{proof}}

\begin{proposition}
  For every ladder bimodule $W$ over $\tilde{T}^{\Bp'}\operatorname{-} \tilde{T}^{\Bp}$, we have that
  $T^{\Bp'}\otimes_{\tilde{T}^{\Bp'}}W\cong
  W\otimes_{\tilde{T}^{\Bp}}T^{\Bp}$ as ${T}^{\Bp'}\operatorname{-} {T}^{\Bp}$ bimodules 
  and all higher Tors vanish.
\end{proposition}
\begin{proof}
  We can reduce to the case of the ladder $Y_i$.  As a right module,
  $W_{Y_i}$ is projective, and thus the higher Tors always vanish.   
On the other hand, as a left module $W_{Y_i}$ is the simple module
$L_{a,b}$ induced on the left with $\tilde{T}^{\Bp_1}$ and on the right with
$\tilde{T}^{\Bp_1}$ where $\Bp'$ is the concatenation $(\Bp_1,a,b,\Bp_2)$.  We have that
$T^{a,b}\Lotimes_{\tilde{T}^{a,b}}L_{a,b}\cong L_{a,b}$ by Lemma
\ref{tensor-same}.   Thus, we have vanishing higher Tors in
general.

Thus, it only remains to prove the first claim of the proposition. The module  $T^{\Bp}\otimes_{\tilde{T}^{\Bp}}W$ kills
violating strands above the $Y$ and
$W\otimes_{\tilde{T}^{\Bp}}T^{\Bp}$ kills them below it.  Thus, we
need only show that these are the same subbimodule.  

Consider a
diagram with a violating strand above the $Y$, and let $y_0$ be the
lowest value of $y$ where a black strand crosses from right of the red
strand (not violating) to left of it (violating) when reading bottom
to top.  If this is below the $Y$, then we are done.  Otherwise,
follow the course of this black strand toward the bottom of the
diagram.  If it joins the top of the box, then we push this strand
through any crossings separating it from the box.  Once there are no
such crossings and dots, we will have a diagram which is 0 by
\eqref{Lbc-relations}, leaving only correction terms with fewer
crossings.  Eventually, we can rewrite this
diagram as a sum of others with fewer crossings where this strand is
still violating and doesn't reach the top of the box.  Thus, we can
now use relation (\ref{through-box1}--\ref{through-box2}) to move the violating strand
partially below the $Y$.

Similarly, any diagram with a violating strand below the $Y$ can be
rewritten as one with such a point above the $Y$, using
(\ref{through-box1}--\ref{through-box2}).  This completes the proof.
\end{proof}

This result shows that tensoring with the cyclotomic quotients
on the left or right defines a 2-functor from the 2-category above to the 2-category with
the same objects and 1-morphisms, but 2-morphisms given by maps between ladder bimodules over $T$.  In
particular, together with Lemma \ref{tensor-same}, it shows that:
\begin{corollary}\label{cor:reduction-tensor}
  For any ladder bimodules $W_1,W_2$ over $\tilde{T}$, we have that 
\[T^{\Bp_1}\Lotimes_{\tilde{T}^{\Bp_1}}(W_1\Lotimes_{\tilde{T}^{\Bp_2}}
W_2)\cong
(W_1\Lotimes_{\tilde{T}^{\Bp_2}}T^{\Bp_2})\Lotimes_{T^{\Bp_2}}(T^{\Bp_2}\Lotimes_{\tilde{T}^{\Bp_2}}W_2)
\cong (W_1\Lotimes_{\tilde{T}^{\Bp_2}}
W_2) \Lotimes_{\tilde{T}^{\Bp_3}}T^{\Bp_3}.\]
\end{corollary}

Assume that $c_1+c_2+c_3=c$.  We have two different ways of branching
from a strand labeled $c$ to $c_1,c_2,c_3$, depending on whether we
split off the left or right strand first:
\begin{equation}
\tikz[thick,baseline]{\draw[wei] (0,-1.5) -- (0,-.5);\draw[wei]
  (-.5,.5) -- (0,-.5); \draw[wei] (1,1.5) -- (0,-.5);\draw[wei]
  (0,1.5) -- (-.5,.5);\draw[wei] (-1,1.5) -- (-.5,.5);}\qquad
\tikz[thick,baseline]{\draw[wei] (0,-1.5) -- (0,-.5);\draw[wei]
  (.5,.5) -- (0,-.5); \draw[wei] (-1,1.5) -- (0,-.5);\draw[wei]
  (0,1.5) -- (.5,.5);\draw[wei] (1,1.5) -- (.5,.5);}\label{trees}
\end{equation}

\begin{proposition}\label{associative}
  The bimodules attached to the ladders in \eqref{trees}
are canonically isomorphic.  
\end{proposition}
\begin{proof}
  Since the $Y$-ladder is projective on the bottom, both of these ladders are
  attached to an honest bimodule. First we check that these
  bimodules are isomorphic.  Commutation with induction
  functors means that we only need to show this for the block of
  $\tilde{T}^c$ with no black lines. That is, we must show that applying these two
  ladders results in the same module over $\tilde{T}^{(c_1,c_2,c_3)}$. Let $L_1$ be
  the module resulting from the first ladder, and $L_2$ from the second.
  The resulting
  module in both cases is killed by all restriction functors. Since
  $\wedgep{c_1}\otimes \wedgep{c_2}\otimes \wedgep{c_3}$ contains
  $\wedgep{c}$ with multiplicity 1 (as is easily shown, for example,
  with {\it jeu de taquin}), there is a unique such
  simple $L$. Both $L_1$ and $L_2$ must be iterated extensions of
  this simple. Furthermore, if we let $\Bi_1$ be the idempotent for the
  sequence
  \begin{multline*}
    (\omega_{c_1},c_1,c_1+1,\dots, c_1+c_2-1, c_1-1, c_1, \dots, c_1+c_2-2, \dots,
    1,\dots, c_2,\\
 \omega_{c_2},c_1+c_2,c_1+c_2+1,\dots, c-1,
    c_1+c_2-1,\dots, c-2,\dots, 1, \dots, c_3,\omega_{c_3})    
  \end{multline*}

then $\dim e_{\Bi_1}L_1=1$, so $L_1$ must itself be simple.  Similarly,
if $\Bi_2$ corresponds to 
\begin{multline*}
  (\omega_{c_1},c_1,c_1+1,\dots, c-1, c_1-1, c_1, \dots, c-2, \dots,
  1,\dots, c_2+c_3,\\ \omega_{c_2},c_2,c_2+1,\dots, c_2+c_3-1, c_2-1,\dots,
  c_2+c_3-2,\dots, 1, \dots, c_3,\omega_{c_3})
\end{multline*}
then $\dim e_{\Bi_2}L_2=1$, so $L_2$ must be simple as well.  Thus,
these modules are isomorphic.

Thus, we can fix an isomorphism between these modules by producing an
isomorphism between the 1-dimensional spaces $  e_{\Bi_1}L_1$ and
$e_{\Bi_1}L_2$.     The  former space is generated by a straight line
diagram, so we only describe the image of this vector in
$e_{\Bi_1}L_2$.   First, let $v$ be the vector in  $L_{c_1,c_2+c_3}$
attached to the standard tableau on a $c_1\times (c_2+c_3)$ box where
we fill the first $c_2$ boxes in all the rows in order, and then the
last $c_3$.  That is, 
fills $(i,j)$ with $j+c_2(i-1)$ if $j\leq c_2$ and with
$j+(c_1-1)c_2+c_3(i-1)$.  If $c_1=c_3=2$ and $c_2=1$, then this
tableau is \[\tikz[thick,scale=.6,baseline]{\draw (0,0) -- (3,0);\draw (0,1) -- (3,1);\draw (0,2)
  -- (3,2);\draw (0,0) -- (0,2);\draw (1,0) -- (1,2);\draw (2,0) --
  (2,2); \draw (3,0) --
  (3,2);\node at (.5,.5){$1$};\node at (1.5,.5){$3$};\node at
  (.5,1.5){$2$};\node at (1.5,1.5){$5$};\node at
  (2.5,.5){$4$};\node at (2.5,1.5){$6$};}\]

Now apply the second split (that is, tensor with $Y_2$), taking the
straight line diagram in the smaller Y.  Finally, we act on this
diagram by pulling the rightmost $c_1c_3$ strands in the large Y into
the smaller Y without introducing any additional crossings between
black strands.  This has the
correct top, and a minimal number of crossings in order to obtain this
top.  Thus, it defines an isomorphism.  In fact, one can easily check
that the map defined symmetrically through a horizontal axis is
inverse to this one, since in both cases, all crossings introduced are
between labels that commute.  
\end{proof}

For example, if $c_1=c_3=2$ and $c_2=1$, then
$\Bi_1=(\omega_2,2,1,\omega_1, 3,4,2,3,1,2,\omega_2)$.  The image of
the straight-line diagram in $e_{\Bi_1}L_2$ is
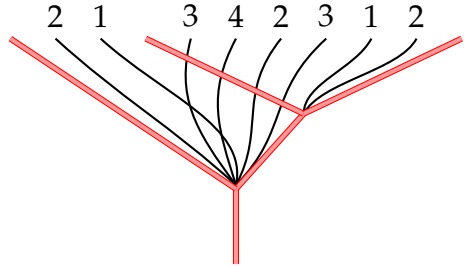
\[
\tikz[thick,baseline,xscale=3]{
\draw (0,-.45) to[out=110,in=-70] node
  [above, at end]{$2$} (-.8,1.5);
\draw (0,-.45) to[out=86,in=-70]node
  [above, at end]{$1$} (-.6,1.5);
\draw (0,-.45) to[out=104,in=-95] node
  [above, at end]{$3$} (-.2,1.5);
\draw (0,-.45) to[out=96,in=-100] node
  [above, at end]{$4$} (0,1.5);
\draw (0,-.45) to[out=80,in=-105] node
  [above, at end]{$2$} (.2,1.5);
\draw (0,-.45) to[out=72,in=-110] node
  [above, at end]{$3$} (.4,1.5);
\draw (.3,.5) to[out=90,in=-100] node
  [above, at end]{$1$} (.6,1.5);
\draw (.3,.5) to[out=80,in=-100] node
  [above, at end]{$2$} (.8,1.5);
\draw[wei] (0,-1.5) -- (0,-.5);\draw[wei]
  (.3,.5) -- (0,-.5); \draw[wei] (-1,1.5) -- (0,-.5);\draw[wei]
  (-.4,1.5) -- (.3,.5);\draw[wei] (1,1.5) -- (.3,.5);}\]
One can easily extend this argument to show that any two ladders that
start at a single red line and branch out to the same top have canonically
isomorphic bimodules. By Mac Lane's coherence theorem, it suffices to show that the 
following diagram commutes on the nose:
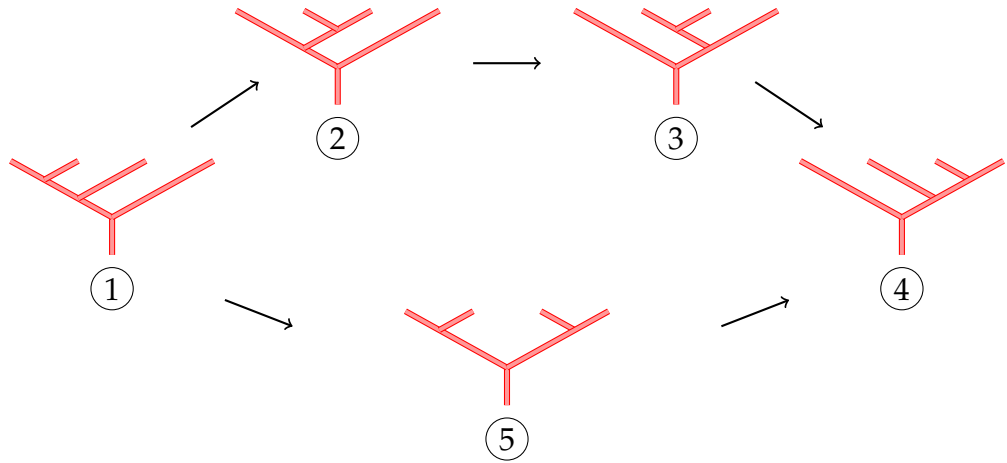
\[
\tikz[thick,baseline,xscale=3]{
\draw[wei] (0,-0.25) -- (0,0.25); \draw[wei] (0,0.25) -- (-0.45,1); \draw[wei] (-0.15,0.5) -- (0.15,1); \draw[wei] (-0.3,0.75) -- (-0.15,1);\draw[wei] (0,0.25) -- (0.45,1);  
\draw[wei] (1,1.75) -- (1,2.25); \draw[wei] (1,2.25) -- (0.55,3); \draw[wei] (0.85,2.5) -- (1.15,3); \draw[wei] (1,2.75) -- (0.85,3);\draw[wei] (1,2.25) -- (1.45,3);  
\draw[wei] (2.5,1.75) -- (2.5,2.25); \draw[wei] (2.5,2.25) -- (2.05,3); \draw[wei] (2.65,2.5) -- (2.35,3); \draw[wei] (2.5,2.75) -- (2.65,3);\draw[wei] (2.5,2.25) -- (2.95,3);  
\draw[wei] (3.5,-0.25) -- (3.5,0.25); \draw[wei] (3.5,0.25) -- (3.05,1); \draw[wei] (3.65,0.5) -- (3.35,1); \draw[wei] (3.8,0.75) -- (3.65,1);\draw[wei] (3.5,0.25) -- (3.95,1);  
\draw[wei] (1.75,-2.25) -- (1.75,-1.75); \draw[wei] (1.75,-1.75) -- (1.3,-1); \draw[wei] (2.05,-1.25) -- (1.9,-1); \draw[wei] (1.45,-1.25) -- (1.6,-1);\draw[wei] (1.75,-1.75) -- (2.2,-1);  
\draw[->] (0.35,1.45) -- (0.65,2.05); \draw[->] (1.6,2.3) -- (1.9,2.3); \draw[->] (2.85,2.05) -- (3.15,1.45); 
\draw[->] (0.5,-0.85) -- (0.8,-1.2);\draw[->] (2.7,-1.2) -- (3,-0.85);
\node at (0,-0.7) {$\circled{1}$};\node at (1,1.3) {$\circled{2}$};\node at (2.5,1.3) {$\circled{3}$};\node at (3.5,-0.7) {$\circled{4}$};\node at (1.75,-2.7) {$\circled{5}$};
}\]
We number the ladders in this coherence diagram 1 to 5 as indicated.  
Note that in the proof of Proposition~\ref{associative}, there is a canonical degree-zero isomorphism 
$e_{\Bi_1}L_2\cong e_{\Bi_2}L_2$ induced by the permutation diagram with bottom 
labeled $\Bi_1$ and top labeled $\Bi_2$. 
Similarly, all the ladders in the coherence diagram above have a natural idempotent associated 
to their top boundary, denoted $e_{\Bi_1},\ldots,e_{\Bi_5}$. And there is a canonical degree-zero isomorphism $e_{\Bi_1}L_j\cong e_{\Bi_j}L_j$, for every $j=1,\ldots, 5$, where $L_j$ is the simple module associated to the $j$th 
ladder as in the proof of Proposition~\ref{associative}. Furthermore, there are canonical degree-zero isomorphisms 
$e_{\Bi_j}L_{j+1}\cong e_{\Bi_1}L_{j+1}$, for $j=1,\ldots,3$, and $e_{\Bi_5}L_4\cong e_{\Bi_4}L_4$. Using the above canonical isomorphisms, one can easily compute the 
image of the generator of $e_{\Bi_1}L_1$ in $e_{\Bi_1}L_j$, for any $j=2,\ldots,5$: it is always given by the degree-zero 
diagram with a minimal number of crossings whose top is $\Bi_1$ and which at each vertex starts with 
the canonical configuration of black strands. This shows in particular that the two isomorphisms $L_1\to L_4$ in the coherence diagram, the one composed by the isomorphisms over the top and the other composed by the isomorphisms over the bottom, are equal. 

We will use this associativity isomorphism many times over
the course of the paper.  Note that this associativity can be used to
show the invariance of the corresponding bimodule under changing
slicing.

\excise{\begin{proposition}
  ladders behave like you expect under rotation.
\end{proposition}
\begin{proof}
  \bentodo{Still needs a proof.  This is basically still just the fact
    that critical points can cancel (by results from \cite{Webmerged})
    and associativity, though.}
\end{proof}}

For any ladder bimodule $W$, we can consider the bimodule $W^\star$
attached to the ladder given by its reflection through the $x$-axis.
This coincides with the bimodule given by switching the left and right
actions on $W$ via the reflection automorphism on $\tilde{T}$.  Let
$\eta(W)$ be the sum over trivalent vertices in $W$ of the quantity
$\eta$ defined as follows: for $Y_i$ or $Y_i^\star$ labeled $a,b,a+b$, 
let 
$\eta(Y_i)=
-ab
$.
\begin{proposition}\label{self-dual}
  For every ladder bimodule $W$, we have $\RHom_{\tilde{T}^{\Bp}
  }(W,\tilde{T}^{\Bp})\cong W^\star\langle \eta(W)\rangle$.  This defines adjunction
  maps \[\ep_W\colon W\Lotimes_{\tilde{T}^{\Bp}} W^\star\langle \eta(W) \rangle\to R\qquad
  \iota_W\colon R\to  W^\star\Lotimes_{\tilde{T}^{\Bp}} W\langle \eta(W) \rangle.\]
\end{proposition}

\begin{proof}
  Of course, it's sufficient to prove this for ladders with one trivalent
  vertex.  We can achieve this by induction on $\min(a,b)$. The case $\min(a,b)=1$ 
was proved in Lemma~\ref{a-b-1}.  For the induction step we can assume
  that $1<a\leq b$. Then we can attach the ladder which splits the first
  strand to arrive at $(a-1,1,b)$.  By Proposition \ref{associative},
  this is the same as starting with $(a-1,b+1)$ and splitting the
  second strand.  By induction, both these ladders have the desired
  adjoints, so the same is true of their composite. This shows that
  the ladder $W'$ given by 
  our original trivalent vertex with a bigon attached to the first
  branch is adjoint to its reflection.  By Proposition \ref{bigon},
  $W'\cong W^{\oplus [a]_q}$, so the same is true of $W$.  
\end{proof}

Note that by symmetry, this means that $W$ and $W^\star$ are biadjoint
up to shift.

\subsection{2-categories}
\label{sec:2-categories}

The material from the previous sections can be encapsulated into the
definition of two 2-categories:
\begin{definition}
  Let $\tWB^n_\ell$ (resp. $\WB^n_\ell$) be the 2-category with:
  \begin{itemize}
  \item objects given by sequences of integers $(p_1,\dots, p_\ell)$, with $p_i\in [0,n]$ for $i\in [1,\ell]$.
  \item morphisms given formal sums of $n$-ladders with $\ell$ uprights.
  \item 2-morphisms $A\to {A'}$ between two ladders is given by the space of 
    morphisms between the ladder bimodules in
    the derived category of bimodules over
    $\tilde{T}^{\Bp'}\operatorname{-}\tilde{T}^{\Bp}$
    (resp. ${T}^{\Bp'}\operatorname{-}{T}^{\Bp}$) that commute with
    the  $\tU_n^-\times \tU_n^-$-action (resp. $\tU_n$-action):
    \begin{align*}
    \Hom_{\tWB}(A,A')&:= \mathbb{R}\Hom_{\tilde{T}^{\Bp'}\otimes
      \tilde{T}^{\Bp}}(W_A,W_{A'})^{\tU_n^-\times \tU_n^-}\\
\Hom_{\WB}(A,A')&:= \mathbb{R}\Hom_{{T}^{\Bp'}\otimes
      {T}^{\Bp}}(W_A\otimes_{\tilde{T}^{\Bp}}{T}^{\Bp},W_{A'}\otimes_{\tilde{T}^{\Bp}}{T}^{\Bp})^{\tU_n}
    \end{align*}
  \end{itemize}
\end{definition}
There is a 2-functor
$\tWB\to \WB$ which is induced by the tensor product functor
$\otimes_{\tilde{T}^{\Bp}}{T}^{\Bp}$. 
Corollary \ref{cor:reduction-tensor} shows the compatibility with
composition of 1-morphisms, and  the discussion before Proposition
\ref{prop:trivalent-commute}, shows that this functor sends natural
transformations commuting with $\tU_n^-\times \tU_n^-$ to those
commuting with $\tU_n$. In particular, it is the
identity on objects and 1-morphisms.  
 Proposition \ref{self-dual}
shows that every 1-morphism in this category has a left and right
adjoint, isomorphic to the reflection of that ladder.   These
2-categories act in an obvious way on the derived categories of modules over
the rings $\tilde{T}^\Bp$ and $T^\Bp$ respectively.  The isomorphisms
of Propositions \ref{bigon} and \ref{associative} can be regarded as
isomorphisms on 1-morphisms in the category $\tWB$ or $\WB$.

\subsection{Grothendieck groups}
\label{sec:grothendieck-groups}

As we mentioned in Section \ref{sec:skew-howe-duality}, we wish to
argue that the 2-categories $\WB$ and $\tWB$ give categorifications of
the category $\mathcal{L}ad^n_\ell/\mathcal{I}^n_\ell$.  One of the
ways we will make this precise is to compare the action of a ladder on
$\iwedge{\Bp}_q \C_q^n$ with that of a ladder bimodule on the
corresponding Grothendieck group.  We'll prove a slightly stronger
statement using the algebra $\tilde{T}$.

\begin{theorem}\label{thm:ladder-GG}
  The induced action of a ladder bimodule on $\bigoplus_{\vert \Bp\vert =p}K^0_q({T}^{\Bp})$ (resp. $\bigoplus_{\vert \Bp\vert =p}K^0_q(\tilde{T}^{\Bp})$)
  agrees with the action of the corresponding ladder on $\iwedge{p}_q(\C_q^\ell\otimes \C_q^n)$
  (resp. $U^-_q\otimes \iwedge{p}_q(\C_q^\ell\otimes \C_q^n)$).
\end{theorem}

\begin{proof}
First, by the compatibility of the maps in Proposition
\ref{prop:GG-iso}, is suffices to check this for $\tilde{T}$, since
this will imply the statement for $T^\Bp$.

We also only need to check this in the case of $Y_i$ and
$Y_i^\star$. Furthermore, by adjunction, it suffices to only check
$Y_i$. 
Using induction functors, we can reduce to the case where there are no
strands other than those joining the trivalent vertex.  That is, we
need only check that these maps are correct on the $\omega_c$ weight
space.  

In the case of $Y_i$, the functor sends $
\tilde{T}^{\omega_c}_{\omega_c}\cong \K$ to the
module $L_{a,b}$ over $\tilde{T}^{a,b}$.  Under the identification
$K^0_q({T}^{a,b})\cong \iwedge{a}_q\C^n_q\otimes
\iwedge{b}_q\C^n_q$, the class of the simple $L_{a,b}$ considered as a ${T}^{a,b}$-module and denoted 
$v_{a,b}:=[L_{a,b}]$, spans
the space of highest weight vectors of weight $\omega_c$.  In fact, it
is the unique highest weight vector of the form
$[L_{a,b}]=v_{\omega_c-\omega_b}\otimes v_{\omega_b}+\cdots$, since
$L_{a,b}$ is the unique simple quotient of the standard
$\Delta_{\omega_c-\omega_b,\omega_b}$.  This is also how the ladder $Y_i$ acts on
the vector $v_c$ (see~\cite[Sect. 3.1]{CKM}).

Considered over $\tilde{T}^{a,b}$, the  class of $[L_{a,b}]$
is thus $1\otimes v_{a,b}$, which is, indeed, how the foam $Y_i$ acts on
the vector $1\otimes v_c$.  
\end{proof}
\subsection{Comparison to category \texorpdfstring{$\cO$}{O}}
\label{sec:comp-categ-co}

There is an equivalence of categories $T^{\Bp}_\mu\mmod$ to a block of
parabolic category $\cO^{\Bp}$ for the Lie algebra $\mathfrak{gl}_p$ and the
parabolic $\fp_{\Bp}$ of block upper triangular matrices with blocks
given by $\Bp$, which we denote $\xi\colon T^{\Bp}_\mu\mmod\to \cO^{\Bp}_\mu$.  If we let $\Bp'=(p_1,\dots,
p_j+p_{j+1},\dots,p_\ell)$, then we have an inclusion functor
$I^{\Bp'}_{\Bp}\colon \cO^{\Bp'}\to \cO^{\Bp}$ and its right adjoint, the {\bf Zuckerman
  functor} $Z^{\Bp}_{\Bp'}\colon D^b(\cO^{\Bp'}) \to D^b(\cO^{\Bp})$.  

\begin{proposition}
  We have isomorphisms $\xi_{\Bp}\circ W_{Y_i}\cong I^{\Bp'}_{\Bp}\circ \xi_{\Bp'}$ and
  $\xi_{\Bp'}\circ W_{Y_i^\star}\cong  Z^{\Bp}_{\Bp'}\circ \xi_{\Bp}\langle p_ip_{i+1} \rangle$.
\end{proposition}
\begin{proof} By 
\cite[\ref{m-trans-act}]{Webmerged}, the action of $\tU_n$ is
intertwined by $\xi_{\Bp}$ with the action of certain projective functors.
By the commutation of Zuckerman and projective functors, we have that
the functor $\xi_{\Bp}^{-1}\circ Z^{\Bp}_{\Bp'}\circ \xi_{\Bp}$
commutes with the action of $\tU_n$, just as $W_{Y_i}$ does.
Similarly, if we have a larger algebra $\fg$, a Levi $\fl$ for a
parabolic $\fp$, and two
parabolics $\fq'\subset \fq\subset \fl$, then we have that the
Zuckerman functor for the pair $\fq'\subset \fq$ intertwines under
parabolic induction with that of $\fq'+\operatorname{rad}(\fp)\subset
\fq+\operatorname{rad}(\fp)$. The description of standardization given in
\cite[\ref{m-ind-sta}]{Webmerged} shows that $\xi_{\Bp}^{-1}\circ
Z^{\Bp}_{\Bp'}\circ \xi_{\Bp}$ thus commutes with the functor
$\fI_{\omega_p}$ adding a new red strand at the right.  

Thus, if we establish the theorem when $\ell=i+1$, it will follow in
all other cases, since we can build all other projective modules using
$\fI_{\omega_p}$ and the action of $\tU_n$, which commute with both
functors.  Thus, let us restrict to that case.

As usual, we let $a=p_i,b=p_{i+1}, c=a+b$.  One consequence of the
standard stratification shown in \cite[\ref{m-SS}]{Webmerged} is that
every projective is a summand of a module obtained by applying the
action of $\tU_n$ to 
$\Ind_{\fp_{(p-c,c)}}(\xi(P)\boxtimes \C^c)=\xi(\fI_{\omega_c}(P))$,
for  $P$ a projective module in $T^{\Bp'_-}_\mu\mmod$ with $~\Bp'_-=(p_1,\ldots,p_{\ell-1})$. 
Thus, it suffices to construct the isomorphism on $\xi(\fI_{\omega_c}(P))$ and extend it using the
action of $\tU_n$.

As in \cite[\ref{m-trace-standard}]{Webmerged},  for an arbitrary
$T^{\Bp'_-}_\mu$-module, 
we have that $W_{Y_i}\Lotimes \fI_{\omega_c}(P)$ is the
standardization of the $T^{(p_1,\dots, p_{j-1})}\boxtimes
T^{a,b}$-module $P\boxtimes L_{a,b}$.  Furthermore, we have an isomorphism
$\xi_{(a,b)}(L_{a,b})\cong \C^c$ since in both cases, these are the
unique simples killed by translation to any singular central character
(which are intertwined with the functors $\fE_i$ on $T^{(a,b)}$-modules).  By
\cite[\ref{m-ind-sta}]{Webmerged}, we thus have an isomorphism
\[\xi_{\Bp}(W_{Y_i}\Lotimes \fI_{\omega_c}(P))\cong
\Ind_{\fp_{(p-c,c)}}(\xi(P)\boxtimes \C^c)\cong
I^{\Bp'}_{\Bp}\xi_{\Bp'}(\fI_{\omega_c}(P))\] that is natural in $P$. 
This completes the proof.
\end{proof}
Due to the difficulty of calculating in derived categories, we will
now smuggle certain results from category $\cO$ into our picture.  We
apologize to those readers who don't like such techniques, but they
should rest assured that these techniques are used in a minor way,
have a big payoff, and facilitate rather than replace computation.  

We call a complex of projective modules over $\tilde{T}^\bla$ {\bf
  linear} if every homogeneous summand $P$ in homological degree
$j$ is
isomorphic to $P^{\circledast}(-2j)$.  
If we replace $\tilde{T}^\bla$
with a Morita equivalent algebra which is positively graded, then this
condition will be equivalent to requiring that the $j$th grade is
generated in degree $-j$. Thus this definition agrees with the usual one (see,
for example,
\cite{MOS}) for a positively graded algebra.  This more general definition allows us to speak of a
linear complex of projectives in any humorous category in the sense of
\cite{WebCB}.  The algebra $\tilde{T}^\bla$ is called {\bf Koszul} if
every self-dual simple module has a linear resolution.
  Any algebra which is Koszul in this sense satisfies a numerical
  criterion of Koszulity: the matrices whose entries are the graded
  dimension of $\Hom(P_\mu,P_\nu)$ where $P_*$ denote the self-dual
  projectives, and
  the graded Euler characteristic of $\Ext(L_\nu,L_\xi)$  with $L_*$
  the simple quotient of $P_*$ are inverse
  (since their product computes the graded multiplicity of $L_\xi$ in
  the projective resolution of $L_\mu$).  By linearity,
  $\Ext(L_\nu,L_\xi)$ is non-positively graded, so $\Hom(P_\mu,P_\nu)$
  is non-negatively graded.  
Thus, $\Hom(\oplus_\mu P_\mu,\oplus_\mu P_\mu)$ is a Morita equivalent
algebra which is Koszul in the usual sense of \cite{BGS96}.

\begin{theorem}\label{T-Koszul}
The algebras  
$T^{\Bp}$ are Koszul.
  Any ladder bimodule preserves the category of linear complexes of
  projectives over 
$T^{\Bp}$.
\end{theorem}
\begin{proof}
Koszulity follows from \cite[\ref{m-sln-Koszul}]{Webmerged}.
By \cite[Cor. 36]{MOS}, the
  Zuckerman and inclusion functors are Koszul dual to projective
  functors, which are exact.  Thus the Zuckerman and inclusion
  functors are $t$-exact in the $t$-structure which is the image of
  the usual $t$-structure on the Koszul dual.  This is the
  $t$-structure whose heart is linear complexes.\excise{

  For $\tilde{T}^{\Bp}$, we must bootstrap from the case of
  ${T}^{\Bp}$.  
By Lemma \ref{simple-isomorphism}, for any homological degree $k$ and any
graded degree $r$, we have that $\Ext^k(L,M)_r\cong \Ext^k(L',M')_r$
for $\Bq$ chosen so that $q$ is sufficiently large.  Thus, since
$T^{\Bq\Bp}$ is Koszul, the latter group is only non-zero if $k=-r$,
so the same is true over $\tilde{T}^{\Bp}$, which is thus Koszul.  The
same argument shows that for any ladder bimodule $\Ext^k_{\tilde{T}^{\Bp} }(L,W_\beta
M)_r\cong \Ext^k_{T^{\Bq\Bp} }(L', W_{\Bq\beta} M')_r$.  Since the
internal and homological gradings coincide on the RHS, the same is
true of the LHS as well.  Similarly,
since the ladder
bimodules over $T^{\Bq\Bp}$ preserve linear complexes, this shows that ladder
bimodules over $\tilde T^{\Bp}$.
preserve linear complexes. }
\end{proof}
We also expect that in most cases, the category $\tilde{T}^\Bp$ is
Koszul.  In certain degenerate cases, it seems to be necessary to
consider its quotient by certain central elements.  For example, if
there are no red strands and one black one, the corresponding algebra
is a polynomial ring with a generator in degree 2, which is thus not
Koszul.  However, if we add a red strand with the same label, then the
resulting algebra is easily shown by hand to be Koszul; this is, for
example, a consequence of
Proposition \ref{prop:endomorphisms}.

\section{The dual categorical action}
\label{sec:dual-categ-acti}

We now wish to show that there is a 2-functor $\tU_\ell\to \tWB$, which
thus induces a categorical action on the module categories
$\tilde{T}^\Bp$.  We first turn to defining this 2-functor on 1-morphisms.

\begin{definition}
  Let $\dF_i^{(n)}$ be the functor of
  tensor product by the ladder bimodule associated to 
$\wF^{(n)}_i$. Let $\dE_i^{(n)}$ be the functor of
  tensor product by the ladder bimodule associated to 
$\wE^{(n)}_i$. 
\end{definition}

\noindent Note that our notation suppresses the labels of the vertical strands, as usual. Only when needed will they 
be specified.

\begin{proposition}
\label{dividedpowers}
We have 
$$\dE_i^n\cong \left(\dE_i^{(n)}\right)^{\oplus [n]_q!}\quad\text{and}\quad \dF_i^n\cong \left(\dF_i^{(n)}\right)^{\oplus [n]_q!}.$$
\end{proposition}
Here we use $[p]_q=\frac{q^p-q^{-p}}{q-q^{-1}}$ as a multiplicity to indicate the sum of $p$
copies of a module, with grading shifts that match the quantum integer.
\begin{proof}
We prove the result for $\dF_i^{(n)}$, for $\dE_i^{(n)}$ it can be proved similarly. By associativity and Proposition~\ref{bigon}, we have
 
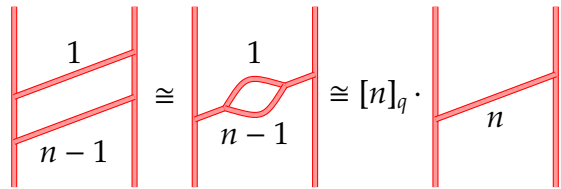
\begin{center}
\tikz[xscale=.8, yscale=.6]
{
\draw[wei] (1,0)--(1,4); \draw[wei] (3,0)--(3,4); \draw[wei] (1,1)--(3,2); \draw[wei] (1,2)--(3,3);
\draw (2,0.8) node {$n-1$};
\draw (2,3) node {$1$};
\draw (3.5,2) node {$\cong$}; 
\draw[wei] (4,0)--(4,4); \draw[wei] (6,0)--(6,4); \draw[wei] (4,1.5)--(4.5,1.75);\draw[wei] (5.5,2.25)--(6,2.5);
\draw[wei] (4.5,1.75) .. controls (4.85,2.5) and (4.85,2.5) .. (5.5,2.25);
\draw[wei] (4.5,1.75) .. controls (5.15,1.5) and (5.15,1.5) .. (5.5,2.25);
\draw (5,1.2) node {$n-1$};
\draw (5,3) node {$1$};
\draw (7,2) node {$\cong [n]_q\,\cdot $}; 
\draw[wei] (8,0)--(8,4); \draw[wei] (10,0)--(10,4); \draw[wei] (8,1.5)--(10,2.5);
\draw (9,1.5) node {$n$};
}
\end{center}
The result follows by induction on $n\geq 1$.
\end{proof}

The functors $\dE_i$ and $\dF_i$ are biadjoint up to shift. 
By Lemma \ref{a-b-1}, we have that left and right adjoints of $Y_i$
with bottom given by $\Bp$ and top by $(p_1,\dots,
p_{i}-1,1,p_{i+1},\dots, p_\ell)$ are given by $Y_i^L=Y_i^\star\langle
1-p_i\rangle$ and $Y_i^R=Y_i^\star\langle
p_i-1\rangle$.  Similarly, for $Y_{i+1}^\star$ with top given by
$\Bp-\al_i$ and bottom by $(p_1,\dots,
p_{i}-1,1,p_{i+1},\dots, p_\ell)$, the adjoints are $(Y_{i+1}^\star)^L=Y_{i+1}\langle
p_{i+1}\rangle$ and $(Y_{i+1}^\star)^R=Y_{i+1}\langle
-p_{i+1}\rangle$.   For convenience, let $\pi_i=p_i-p_{i+1}-1$.   Thus, the left and right adjoints of $\dF_i$ are
given by \[\dF_i^L=\dE_i\langle
-\pi_i\rangle\qquad\dF_i^R=\dE_i\langle
\pi_i\rangle.\]
We consider the 
adjunctions (co)units for the left and right adjunctions, respectively:
\begin{align*}
{\epsilon}_i&:=\epsilon_{Y_i^\star}\circ\epsilon_{Y_{i+1}}\colon
\dE_i\dF_i\langle -\pi_i\rangle \to
\tilde{T}^{p}& {\iota}_i&:=\iota_{Y_{i}^\star}\circ
\iota_{Y_{i+1}} \colon \tilde{T}^{p'}\to \dF_i \dE_i\langle
-\pi_i \rangle \\
\epsilon_i'&:= (-1)^{p_{i}-1}\epsilon_{Y_{i+1}^\star}\circ\epsilon_{Y_{i}}\colon \dF_i\dE_i\langle \pi_i \rangle \to \tilde{T}^{p'}&\iota_i'&:= (-1)^{p_{i}-1}\iota_{Y_{i+1}^\star}\circ \iota_{Y_i} \colon \tilde{T}^p\to\dE_i \dF_i\langle 
 \pi_i\rangle
\end{align*}


where $p,p'\in\Gamma^n_\ell$ are arbitrary such that $p'_i=p_i-1$,
$p'_{i+1}=p_{i+1}+1$ and $p_j'=p_j$ for all $j\ne i,i+1$.  Note that
our grading shifts look a little different from those in \cite{KLIII};
in part this is because we have fixed $\Bp$ to be the list of weights
at the bottom of $\dF_i$, not the bottom of the picture, amongst other
differences in convention.
For the reader to "get the picture", we draw the maps needed for
$\epsilon_i'$ and $\iota_i$.  The maps for $\epsilon_i$ and $\iota'_i$
come from the reflection of this diagram through the $y$-axis.
 
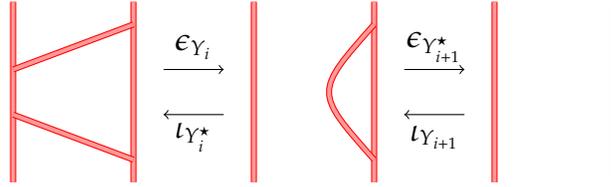
\begin{center}
\tikz[xscale=.8, yscale=.6]
{
\draw[wei] (1,0)--(1,4); \draw[wei] (3,0)--(3,4); \draw[wei] (1,1.5)--(3,0.5); \draw[wei] (3,3.5)--(1,2.5);
\draw[->](3.5,2.5)--(4.5,2.5);\draw (4,3) node {$\epsilon_{Y_{i}}$};
\draw[<-](3.5,1.5)--(4.5,1.5);\draw (4,1) node {$\iota_{Y_{i}^\star}$};
\draw[wei] (5,0)--(5,4); \draw[wei] (7,0)--(7,4); \draw[wei] (7,0.5) .. controls (6,2) and (6,2) .. (7,3.5);
\draw[->](7.5,2.5)--(8.5,2.5);\draw (8,3) node {$\epsilon_{Y_{i+1}^\star}$};
\draw[<-](7.5,1.5)--(8.5,1.5);\draw (8,1) node {$\iota_{Y_{i+1}}$};
\draw[wei] (9,0)--(9,4); \draw[wei] (11,0)--(11,4);
}
\end{center}


\begin{definition}
  Let $y$ be the natural transformation of $\dF_i$  given by $1\otimes
  \ssy_{i+1}\otimes 1$.  Let $\psi\colon \dF_i \dF_j\to \dF_j \dF_i$ be
  the map defined by
  \begin{itemize}
  \item isotopy if $|i-j|>1$
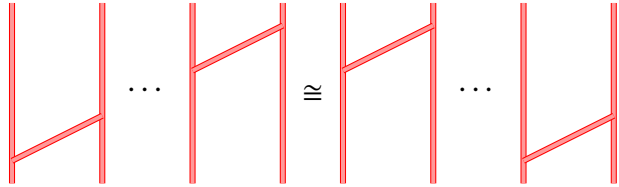
\begin{center}
\tikz[xscale=.8, yscale=.6]
{
\draw[wei] (0,0)--(0,4); \draw[wei] (1.5,0)--(1.5,4); \draw[wei] (3,0)--(3,4);\draw[wei] (4.5,0)--(4.5,4);
\draw[wei] (0,0.5)--(1.5,1.5); \draw[wei] (3,2.5)--(4.5,3.5);
\draw (2.25,2) node {$\cdots$};
\draw (5,2) node {$\cong$};
\draw[wei] (5.5,0)--(5.5,4); \draw[wei] (7,0)--(7,4); \draw[wei] (8.5,0)--(8.5,4);\draw[wei] (10,0)--(10,4);
\draw[wei] (5.5,2.5)--(7,3.5); \draw[wei] (8.5,0.5)--(10,1.5);
\draw (7.75,2) node {$\cdots$};
}
\end{center}
  \item the adjoint of the map of the associativity isomorphism $
    \dF_{i\pm 1} \dE_i\to \dE_i \dF_{i\pm 1} $ if $j=i\pm 1$
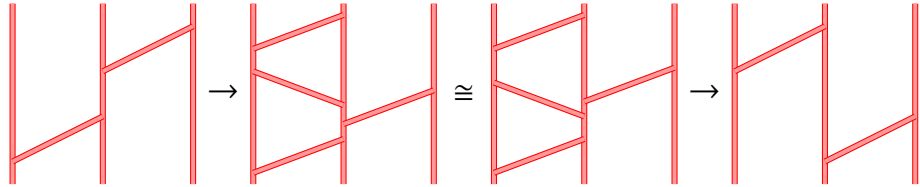
\begin{center}
\tikz[xscale=.8, yscale=.6]
{
\draw[wei] (0,0)--(0,4); \draw[wei] (1.5,0)--(1.5,4); \draw[wei] (3,0)--(3,4);
\draw[wei] (0,0.5)--(1.5,1.5); \draw[wei] (1.5,2.5)--(3,3.5);
\draw (3.5,2) node {$\to$};
\draw[wei] (4,0)--(4,4); \draw[wei] (5.5,0)--(5.5,4); \draw[wei] (7,0)--(7,4);
\draw[wei] (4,0.25) -- (5.5,1); \draw[wei] (5.5,1.33)--(7,2.08); \draw[wei] (4,2.5) -- (5.5,1.75);\draw[wei] (4,3) -- (5.5,3.75);
\draw (7.5,2) node {$\cong$};
\draw[wei] (8,0)--(8,4); \draw[wei] (9.5,0)--(9.5,4); \draw[wei] (11,0)--(11,4);
\draw[wei] (8,0.25) -- (9.5,1); \draw[wei] (9.5,1.5)--(8,2.25); \draw[wei] (9.5,1.83)--(11,2.58); \draw[wei] (8,3) -- (9.5,3.75);
\draw (11.5,2) node {$\to$};
\draw[wei] (12,0)--(12,4); \draw[wei] (13.5,0)--(13.5,4); \draw[wei] (15,0)--(15,4);
\draw[wei] (12,2.5)--(13.5,3.5); \draw[wei] (13.5,0.5)--(15,1.5);
}
\end{center}
  \item the natural transformation induced by $\iota_Y\circ
    \ep_{Y^\star}$ after using the associativity isomorphism if $i=j$
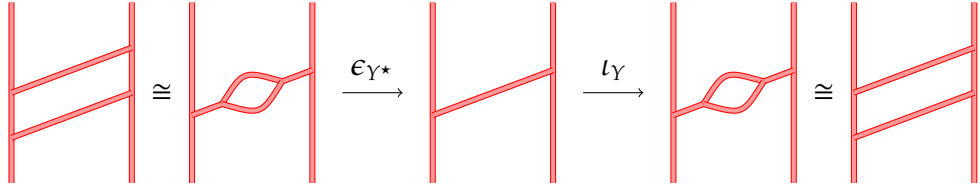
 \begin{center}
\tikz[xscale=.8, yscale=.6]
{
\draw[wei] (1,0)--(1,4); \draw[wei] (3,0)--(3,4); \draw[wei] (1,1)--(3,2); \draw[wei] (1,2)--(3,3);
\draw (3.5,2) node {$\cong$}; 
\draw[wei] (4,0)--(4,4); \draw[wei] (6,0)--(6,4); \draw[wei] (4,1.5)--(4.5,1.75);\draw[wei] (5.5,2.25)--(6,2.5);
\draw[wei] (4.5,1.75) .. controls (4.85,2.5) and (4.85,2.5) .. (5.5,2.25);
\draw[wei] (4.5,1.75) .. controls (5.15,1.5) and (5.15,1.5) .. (5.5,2.25);
\draw[->](6.5,2)--(7.5,2);\draw (7,2.5) node {$\epsilon_{Y^{\star}}$};
\draw[wei] (8,0)--(8,4); \draw[wei] (10,0)--(10,4); \draw[wei] (8,1.5)--(10,2.5);
\draw[->](10.5,2)--(11.5,2);\draw (11,2.5) node {$\iota_{Y}$};
\draw[wei] (12,0)--(12,4); \draw[wei] (14,0)--(14,4); \draw[wei] (12,1.5)--(12.5,1.75);\draw[wei] (13.5,2.25)--(14,2.5);
\draw[wei] (12.5,1.75) .. controls (12.85,2.5) and (12.85,2.5) .. (13.5,2.25);
\draw[wei] (12.5,1.75) .. controls (13.15,1.5) and (13.15,1.5) .. (13.5,2.25);
\draw (14.5,2) node {$\cong$}; 
\draw[wei] (15,0)--(15,4); \draw[wei] (17,0)--(17,4); \draw[wei] (15,1)--(17,2); \draw[wei] (15,2)--(17,3);
}
\end{center}
  \end{itemize}
\end{definition}

Our aim is to show that these natural transformations define a
categorical action of the 2-category $\tU_\ell$ sending $\eF_i\to
\dF_i$ and $\eE_i\to \dE_i$.

\begin{remark}\label{dual-grading}
  Note that this functor follows the ``Koszul dual'' grading
  convention. It sends grading shift $(i)$ to the Tate twist $\langle i\rangle=[i](-i)$ and vice versa; in particular, a map of degree $k$ in the 2-category
  $\tU$ is sent to an element of $\Ext^k$ of internal degree $-k$. 
\end{remark}

\subsection{NilHecke relations}
\label{sec:nilhecke-relations}

We wish to establish that we have an action of the nilHecke algebra on
$\dF_i^n$ defined by.  

\begin{lemma}
  On $\dF_i^2$, we have the relation
  \begin{equation}
(y\otimes 1)\psi-\psi(1\otimes
  y)=\psi (y\otimes 1)-(1\otimes y) \psi=1.\label{nilHecke-a}
\end{equation}
\end{lemma}

\begin{proof}
In terms of our usual notation for Hochschild cohomology classes, we
must show that   \begin{equation}\label{nilHecke-y}
\ssy_i\psi-\psi \ssy_{i+1}=\psi \ssy_i-\ssy_{i+1}\psi=1.
\end{equation}

Recall that, by associativity, we have 
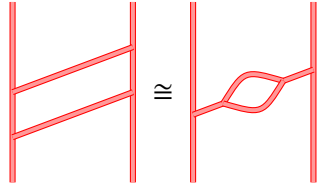
\begin{center}
\tikz[xscale=.8, yscale=.6]
{
\draw[wei] (1,0)--(1,4); \draw[wei] (3,0)--(3,4); \draw[wei] (1,1)--(3,2); \draw[wei] (1,2)--(3,3);
\draw (3.5,2) node {$\cong$}; 
\draw[wei] (4,0)--(4,4); \draw[wei] (6,0)--(6,4); \draw[wei] (4,1.5)--(4.5,1.75);\draw[wei] (5.5,2.25)--(6,2.5);
\draw[wei] (4.5,1.75) .. controls (4.85,2.5) and (4.85,2.5) .. (5.5,2.25);
\draw[wei] (4.5,1.75) .. controls (5.15,1.5) and (5.15,1.5) .. (5.5,2.25);
}
\end{center}
Therefore, we can prove~\eqref{nilHecke-y} on the bigon web $\beta=Y^\star Y$ (we supress the
subscript $i$ here) with $a=b=1$ and $c=2$.

The associated bimodule resolutions from Corollary \ref{W-res} are
both honest projective modules on the sides where the tensors occur.
Let $A$ and $D$ be the generator of $P_{\emptyset}$ in homological degree $0$ and $-2$, respectively, 
$B$ the generator of $P_{1}$ and $C$ the generator of $P_{2}$. 
Thus, we can realize the tensor product as an $A_\infty$ bimodule,
generated by the tensors $A\otimes A, $ etc.  The
counit $\epsilon_{Y^\star}$ is defined by the pairing
\begin{equation}
  \label{pairing}
    \langle A,D\rangle=\langle B,B\rangle=-\langle C,C\rangle=\langle D,A\rangle=1
\end{equation}
with all other pairings 0. One can easily check that $\langle 
 \partial x,y\rangle=\langle x, \partial y\rangle$.

 The map
$\psi$ sends $A\otimes A$ to the canonical element of the pairing
given in \eqref{pairing}.  That is:
\[ A\otimes A \mapsto \mathsf{k}:=A\otimes D + B\otimes B -C\otimes C+D\otimes
A \]

Technically, we should use the $A_\infty$ tensor product
$\tilde{T}^{\Bp}\overset{\infty}\otimes W_\beta\overset{\infty}\otimes
\tilde{T}^{\Bp}$; however, the map $W_\beta\to \tilde{T}^{\Bp}\overset{\infty}\otimes W_\beta\overset{\infty}\otimes
\tilde{T}^{\Bp} $ sending
$w\mapsto 1\otimes w\otimes 1$ is a quasi-isomorphism of complexes
by so we can always transfer our
maps using this quasi-isomorphism. 

Note that since $W_\beta\cong T^{\Bp}\langle 1\rangle\oplus
T^{\Bp}\langle -1\rangle$ with these copies generated by $A\otimes A$
and $\mathsf{k}$, in order to check an equality of chain maps up to homotopy, it
suffices to check that it holds on $A\otimes A$
and $\mathsf{k}$ modulo boundaries.

Let us note several important equalities.  Using the associated
deformation, we can calculate $\ssy_i(\mathsf{k})$  by using $\ssy_i$ to deform the
differential on the left hand side of the complex (see Lemma~\ref{Hochschild-action} and the text above it). Thus, taking the square of the deformed differential, we have that \[\ssy_i (D\otimes A)= \frac{1}{h}\,\,\tikz[very thick,baseline=-3pt]{\draw (0,-.5) -- (-1,0) -- (0,.5);  \draw[wei]
  (-.5,-.5) -- (-.5,.5);\draw[wei]
  (.5,-.5) -- (.5,.5);}\, \,A\otimes A  -\frac{1}{h}\,\,\tikz[very thick,baseline=-3pt]{\draw (0,-.5) -- (1,0) -- (0,.5);  \draw[wei]
  (-.5,-.5) -- (-.5,.5);\draw[wei]
  (.5,-.5) -- (.5,.5);}\,\, A\otimes A\,=\frac{1}{h} (hA\otimes
A)=A\otimes A.\] Since the other terms in $\mathsf{k}$ are killed, we have
that $\ssy_i (\mathsf{k})=A\otimes A$. Similarly, $\ssy_{i+1}
(\mathsf{k})=-A\otimes A$ since 
\[\ssy_{i+1} (D\otimes A)=  \frac{1}{h}\,\,\tikz[very thick,baseline=-3pt]{\draw (0,-.5) -- (-1,0) -- (0,.5);  \draw[wei]
  (-.5,-.5) -- (-.5,.5);\draw[wei]
  (.5,-.5) -- (.5,.5);}\, \,A\otimes A -\frac{1}{h}\,\,\tikz[very thick,baseline=-3pt]{\draw (0,-.5) -- (1,0) -- (0,.5);  \draw[wei]
  (-.5,-.5) -- (-.5,.5);\draw[wei]
  (.5,-.5) -- (.5,.5);}\,\, A\otimes A\,=-\frac{1}{h} (hA\otimes
A)=-A\otimes A.\]

On the other hand, $\psi(\mathsf{k})$ is a boundary since $\psi^2=0$, and $\ssy_i
(A\otimes A)=\ssy_{i+1} (A\otimes A)=0$ for degree reasons.

Thus, we have that (modulo boundaries)
\begin{align*}
  \mathsf{y}_i\psi(A\otimes A)&= A\otimes A&  \mathsf{y}_{i}(A\otimes A)&=0\\
  \mathsf{y}_{i+1}\psi(A\otimes A)&= -A\otimes A&\psi 
  \mathsf{y}_{i+1}(A\otimes A)&=0\\
 \mathsf{y}_i\psi(\mathsf{k})&=0&\psi \mathsf{y}_{i}(\mathsf{k})&=\mathsf{k}\\
\mathsf{y}_{i+1}\psi(\mathsf{k})&=0&\psi \mathsf{y}_{i+1}(\mathsf{k})&=-\mathsf{k}\\
\end{align*}
This shows that up to homotopy
\begin{equation*}
\mathsf{y}_i\psi-\psi 
  \mathsf{y}_{i+1}=
\psi\mathsf{y}_i- 
  \mathsf{y}_{i+1}\psi=1,
\end{equation*} thus giving the relation \eqref{nilHecke-y}.

\end{proof}

\begin{lemma}
  On $\dF^3_i$, we have the relation
  \begin{equation}
    \label{nilHecke-b}
    (\psi\otimes 1)(1\otimes \psi)(\psi\otimes 1)= (1\otimes \psi) (\psi\otimes 1)(1\otimes \psi)
  \end{equation}
\end{lemma}

\begin{proof}
Note that the bidegree of both maps is $(-6,6)$, where the first entry is the homological degree and the second the internal degree. Note also that, by associativity, we have 
 
 
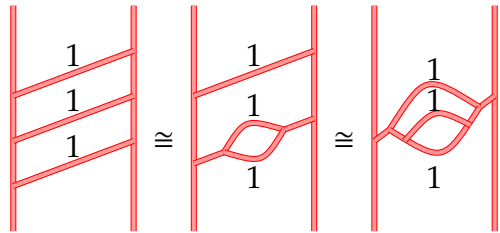
\begin{center}
\tikz[xscale=.8, yscale=.6]
{
\draw[wei] (1,0)--(1,5); \draw[wei] (3,0)--(3,5); \draw[wei] (1,1)--(3,2); \draw[wei] (1,2)--(3,3);\draw[wei] (1,3)--(3,4);
\draw (2,1.9) node {$1$}; \draw (2,2.9) node {$1$}; \draw (2,3.9) node {$1$}; 
\draw (3.5,2) node {$\cong$}; 
\draw[wei] (4,0)--(4,5); \draw[wei] (6,0)--(6,5); \draw[wei] (4,1.5)--(4.5,1.75);\draw[wei] (5.5,2.25)--(6,2.5);
\draw[wei] (4.5,1.75) .. controls (4.85,2.5) and (4.85,2.5) .. (5.5,2.25);
\draw[wei] (4.5,1.75) .. controls (5.15,1.5) and (5.15,1.5) .. (5.5,2.25);
\draw[wei] (4,3)--(6,4);
\draw (5,2.8) node {$1$}; \draw (5,1.2) node {$1$}; \draw (5,3.9) node {$1$}; 
\draw (6.5,2) node {$\cong$};
\draw[wei] (7,0)--(7,5); \draw[wei] (9,0)--(9,5); \draw[wei] (7,2)--(7.25,2.25);\draw[wei] (8.75,2.75)--(9,3);
\draw[wei] (7.25,2.25) .. controls (7.85,3.5) and (7.85,3.5) .. (8.75,2.75);
\draw[wei] (7.25,2.25) .. controls (8.15,1.5) and (8.15,1.5) .. (8.75,2.75);
\draw[wei] (7.5,2) .. controls (8,2.75) and (8,2.75) .. (8.6,2.35);
\draw (8,2.9) node {$1$}; \draw (8,1.2) node {$1$}; \draw (8,3.6) node {$1$}; 
}
\end{center}
and 
 
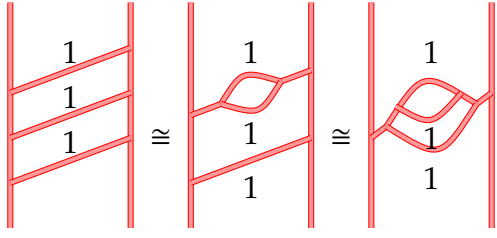
\begin{center}
\tikz[xscale=.8, yscale=.6]
{
\draw[wei] (1,0)--(1,5); \draw[wei] (3,0)--(3,5); \draw[wei] (1,1)--(3,2); \draw[wei] (1,2)--(3,3);\draw[wei] (1,3)--(3,4);
\draw (2,1.9) node {$1$}; \draw (2,2.9) node {$1$}; \draw (2,3.9) node {$1$}; 
\draw (3.5,2) node {$\cong$}; 
\draw[wei] (4,0)--(4,5); \draw[wei] (6,0)--(6,5); \draw[wei] (4,2.5)--(4.5,2.75);\draw[wei] (5.5,3.25)--(6,3.5);
\draw[wei] (4.5,2.75) .. controls (4.85,3.5) and (4.85,3.5) .. (5.5,3.25);
\draw[wei] (4.5,2.75) .. controls (5.15,2.5) and (5.15,2.5) .. (5.5,3.25);
\draw[wei] (4,1)--(6,2);
\draw (5,0.9) node {$1$}; \draw (5,2.1) node {$1$}; \draw (5,3.9) node {$1$}; 
\draw (6.5,2) node {$\cong$};
\draw[wei] (7,0)--(7,5); \draw[wei] (9,0)--(9,5); \draw[wei] (7,2)--(7.25,2.25);\draw[wei] (8.75,2.75)--(9,3);
\draw[wei] (7.25,2.25) .. controls (7.85,3.5) and (7.85,3.5) .. (8.75,2.75);
\draw[wei] (7.25,2.25) .. controls (8.15,1.5) and (8.15,1.5) .. (8.75,2.75);
\draw[wei] (7.4,2.7) .. controls (8,2.25) and (8,2.25) .. (8.5,3);
\draw (8,1.1) node {$1$}; \draw (8,2) node {$1$}; \draw (8,3.9) node {$1$}; 
}
\end{center}

 Applying Proposition~\ref{bigon} twice gives
 
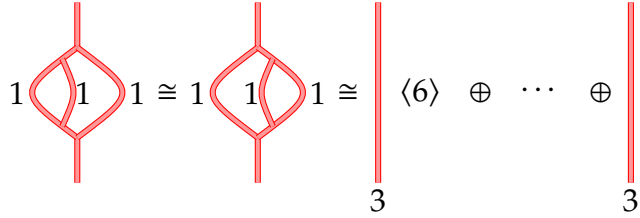
\begin{center}
\tikz[xscale=.8, yscale=.6]
{
\draw[wei] (1,0)--(1,1); \draw[wei] (1,3)--(1,4); 
\draw[wei] (1,1) .. controls (0,2) and (0,2) .. (1,3);
\draw[wei] (1,1) .. controls (2,2) and (2,2) .. (1,3);
\draw[wei] (0.75,1.25) .. controls (1,2) and (1,2) .. (0.75,2.75);
\draw (0,2) node {$1$}; \draw (1.1,2) node {$1$}; \draw (2,2) node {$1$}; 
\draw (2.5,2) node {$\cong$};
\draw[wei] (4,0)--(4,1); \draw[wei] (4,3)--(4,4); 
\draw[wei] (4,1) .. controls (3,2) and (3,2) .. (4,3);
\draw[wei] (4,1) .. controls (5,2) and (5,2) .. (4,3);
\draw[wei] (4.25,1.25) .. controls (4,2) and (4,2) .. (4.25,2.75);
\draw (3,2) node {$1$}; \draw (3.9,2) node {$1$}; \draw (5,2) node {$1$}; 
\draw (5.5,2) node {$\cong$};
\draw[wei] (6,0)--(6,4); 
\draw (7.7,2) node {$\oplus$};
\draw (8.7,2) node {$\cdots$};
\draw (9.7,2) node {$\oplus$};
\draw[wei] (10.2,0)--(10.2,4); 
\draw (6.7,2) node {$\langle 6\rangle $};
\draw (6,-0.4) node {$3$}; \draw (10.2,-0.4) node {$3$}; 
}
\end{center}

Let $W_1$ and $W_2$ be the bimodules associated to the first and the second web in this picture, respectively.
Then $$\Ext^{-6}(W_1,W_1)\cong Z(\tilde{T}^{\omega_3})(-6)\cong \Ext^{-6}(W_2,W_2)$$
and the composite of the two isomorphisms is equal to the associativity isomorphism in homological degree $-6$. The remarks below Proposition~\ref{bigon} show that both $(\psi\otimes 1)(1\otimes \psi)(\psi\otimes 1)$ 
and $(1\otimes \psi)(\psi\otimes 1)(1\otimes \psi)$ correspond to $1\in Z(\tilde{T}^{\omega_3})(-6)$.
\end{proof}

This action gives us an alternate way of defining the divided power
functor $\dF_i^{(n)}$ as the image of a primitive idempotent in
$NH_n$.

In particular, this shows that:
\begin{corollary}
  The maps $y$ and $\psi$ define an action of the nilHecke algebra
  $NH_n$ on $\dF_i^n$, and thus of the symmetric polynomials
  $\K[y_1,\dots, y_n]^{S_n}\cong Z(NH_n)$ on the summand $\dF_i^{(n)}$.  
\end{corollary}

\subsection{Reduction to trees}
\label{sec:reduction-trees}

\begin{definition}
  We call a ladder a {\bf tree} if at $y=0$ there is a single red strand
  (necessarily labeled $p$) and there are no cycles in its underlying
  graph (that is, the underlying graph is a tree in the usual sense).
  We call a ladder a {\bf rootstock} if its reflection is a tree.
\end{definition}

Note that up to isotopy, if we fix the sequence at $y=1$ to be $\Bp$,
there is a single tree with this top, which we denote
$\tau_{\Bp}$. Note that every tree is generated from
$\tau_{(p,0,\dots,0)}$ by applying categorification functors:
\[\tau_{\Bp}\cong
\wF_1^{(p_2)}\circ \wF_2^{(p_3)}\circ \wF_1^{(p_3)}\circ \cdots \circ \wF_{\ell-2}^{(p_{\ell -1})}\circ \cdots \circ  \wF_2^{(p_{\ell -1})}\circ \wF_1^{(p_{\ell -1})}\circ \wF_{\ell-1}^{(p_\ell)}\circ \cdots \circ \wF_{2}^{(p_\ell)}\circ \wF_1^{(p_\ell)}\tau_{(p,0,\dots,0)}.\]

The action of symmetric polynomials on $\dF_1^{p_i}$ as the center of the nilHecke algebra defines an action
of $R_{\Bp}\cong \K[\ssy_1,\dots, \ssy_p]^{S_{p_1}\times \cdots \times
  S_{p_\ell}}$ on $\tau_{\Bp}$.  In terms of ladders, we can think of
this as the branch of the tree with label $p_i$ having an action of
symmetric polynomials on an alphabet of $p_i$ variables.

We
let $P_{\Bp}\subset G:=GL(\C^p)$ be the block upper-triangular matrices
with block sizes $\Bp$.   Note that $H^*(G/P_{\Bp};\C)$ is naturally a
quotient of $R_{\Bp}$ identifying $\ssy_i$ with the Chern classes of
tautological line bundles on the full flag variety; the kernel of this
map is the ideal generated by positive degree elements in the full
ring of $S_p$ invariant polynomials $R_{(p,0,0,\dots)}$.  Throughout what
follows, we'll use $\dagger$ to denote either $\WB$ or $\tWB$ when a
result holds in both.

\begin{proposition}\label{tree-end}
  The map $R_{\Bp}\to \Hom_{\dagger}({\tau_{\Bp}},{\tau_{\Bp}})$ induces an
  isomorphism \[\Hom_{\dagger}({\tau_{\Bp}},{\tau_{\Bp}})\cong H^*(G/P_{\Bp};\C).\]
\end{proposition}
\begin{proof}
 By Proposition
\ref{self-dual}, we have that \[
\Hom_{\tWB}({\tau_{\Bp}},\tau_{\Bp})\cong \Hom_{\tWB}(\id_{(p)},
{\tau_{\Bp}^\star}\circ {\tau_{\Bp}}).\] Applying the bigon
relation (Proposition \ref{bigon}) shows that 
\begin{align*}
   W_{\tau_{(p_1,\dots,p_\ell)}}^\star\Lotimes_{\tilde{T}^{(p_1,\dots,p_\ell)}}
  W_{\tau_{(p_1,\dots,p_\ell)}}&\cong\frac{(p_1+p_2)!}{p_1!p_2!}\cdot (
  W_{\tau_{(p_1+p_2,\dots,p_\ell)}}^\star\Lotimes_{\tilde{T}^{(p_1+p_2,\dots,p_\ell)}}
  W_{\tau_{(p_1+p_2,\dots,p_\ell)}})\\ &\cong \frac{(p_1+p_2+p_3)!}{p_1!p_2!p_3!}\cdot (
  W_{\tau_{(p_1+p_2+p_3,\dots,p_\ell)}}^\star\Lotimes_{\tilde{T}^{(p_1+p_2+p_3,\dots,p_\ell)}}
  W_{\tau_{(p_1+p_2+p_3,\dots,p_\ell)}})\\ &\cong \cdots \cong
                                             \frac{n!}{p_1!\cdots
                                             p_\ell!}\cdot \tilde{T}^{\Bp}\\
&=\dim H^*(G/P_{\Bp};\C)\cdot \tilde{T}^{\Bp}
\end{align*}
By Corollary~\ref{cor:reduction-tensor}, the same holds over $T^{\Bp}$. Note that we have $\dim \Hom_{\dagger}(\id_{(p)},\id_{(p)})=1$, since
every projective module over $\tilde{T}^{(p)}$ is a summand of an
induction of the simple with no black strands.  Thus, any
endomorphism commuting with induction is a scalar multiplication.
Thus, we have that \[
\dim \Hom_{\dagger}({\tau_{\Bp}},\tau_{\Bp})= \dim
H^*(G/P_{\Bp};\C)\cdot \dim \Hom_{\dagger}(\id_{(p)},\id_{(p)})= \dim
H^*(G/P_{\Bp};\C).\] 
First, we specialize to the case where $\Bp=(1,\dots,1)$.  Note that any element of the kernel of $R_{(1,\dots,1)}$ acting on $W_{\tau_{\Bp}}$ must lie in the kernel of the
action of $NH_p$ on $W_{\tau_{(1,\dots,1)}^\star}\Lotimes_{\tilde{T}^{(1,\dots,1)}}
W_{\tau_{(1,\dots,1)}}$.  Since $NH_p$ is Morita equivalent to $\K[\ssy_1,\dots,
\ssy_p]^{S_p}$, the kernel of any action of $NH_p$ must be generated
by an ideal of $\K[\ssy_1,\dots,
\ssy_p]^{S_p}$.  For any homogeneous action, this kernel must lie in
the ideal $I$ generated by positive degree symmetric functions.  This shows that for general $\Bp$, the kernel is
still contained in $I$.  

On the other hand,  the codimension of the kernel can be no greater
than 
$\frac{n!}{p_1!\cdots p_\ell!}=\dim
R_{\Bp}/I$. This is only possible if the map $R_{\Bp}\to
\Hom_{\dagger}({\tau_{\Bp}},\tau_{\Bp})$ induces an isomorphism
\[\Hom_{\dagger}({\tau_{\Bp}},\tau_{\Bp})\cong R_{\Bp}/I\cong H^*(G/P_{\Bp};\C).\qedhere\]
\end{proof}

Recall that the space of 2-morphisms between two ladders in $\tWB$ or $\WB$
carries both an internal and a homological grading.
We call a 2-morphism $N$ {\bf pure} if its internal and homological gradings sum
to 0.  Note that all
of our 2-morphisms for the proposed dual action of
$\tU_{ \ell}$ are pure by construction.

For technical reasons, we will restrict ourselves for now to
constructing a 2-functor $\tU_\ell\to \WB$; we will turn later to the
stronger statement that we can lift this to a 2-functor $\tU_\ell\to \tWB$.
Thus, until stated otherwise, all ladder bimodules will be cyclotomically reduced.
\begin{lemma}
   Assume $N\colon W_{\beta_1}\to W_{\beta_2}$ is a pure 2-morphism in
   $\WB$ for $\beta_i$ a 1-morphism $\Bp\to {\Bp'}$.  Then $N$ is
  0 if and only if it induces the trivial map on modules whose weight is
  $(1,\dots, 1)$.  
\end{lemma}
\begin{proof}
  The algebra $T^\Bp_{\mathbf{m}}$ for any weight
  $\mathbf{m}=(m_1,\dots ,m_n)$ of $\mathfrak{gl}_n$ can be thought of
  as the same algebra for $\mathfrak{gl}_{n+1}$ and the weight
  $(m_1,\dots ,m_n,0)$ where we simply never use the label we have
  added.  By adding additional labels we never use,
\begin{itemize}
\item [(1)] we may assume that $p\leq n$.
\end{itemize}

Assume that $P$ is a projective module with weight $\mathbf{m}\neq (1,\dots,1,0,\dots, 0)$. 
There must be some $m_q$ with
$q\leq p$ such that $m_q=0$.  Note that in this case,
$\eF_{q-1}^{(m_{q-1})}$ and $\eE_{q}^{(m_{q+1})}$ are equivalences of
  categories from $T^\Bp_{\mathbf{m}}\mmod $ to $
  T^\Bp_{s_{q-1}\mathbf{m}}\mmod$ and $T^\Bp_{s_{q}\mathbf{m}}$, respectively. 
Thus, we can move any zero further right in the dimension vector
  without affecting whether $N$ is 0. Thus,
  \begin{itemize}
  \item[(2)] we can assume that the last $n-p$ entries of the weight are 0.
  \end{itemize}
Since we have
  not used the strands with labels $>p$, we can forget them and so
  \begin{itemize}
  \item[(3)] we may assume that $n=p$.
  \end{itemize}
Thus, either the weight is $(1,\dots, 1)$, or we have some largest entry
$m_r>1$.  We can also assume that $m_{r+1}=0$, since at least one
entry must be 0, and we can move it freely using the Weyl group.  In this case, the module $P$ must be highest weight
  for the categorical action of $\eE_r$ and $\eF_r$.  Thus, we have
  that  $\eE_r\eF_rP\cong P^{\oplus m_r}$, by the commutator relation
  in $\mathfrak{sl}_2$.  Thus, it suffices to check the result for
  $\eF_rP$. The resulting weight has fewer entries which are $0$.
  Thus applying this result inductively, we can see that we can
  assume that $P$ has weight $(1,\dots,1)$, and the proof is done.  
\end{proof}
\begin{lemma}\label{tree-enough}
  Assume $N\colon W_{\beta_1}\to W_{\beta_2}$ is a pure 2-morphism in
  $\WB$ as before.  Then $N$ is
  0 if and only if the same is true for the induced morphism
  $N\circ \id_{\tau_{\Bp}}\in \Hom_{\WB}({\beta_1}\circ {\tau_{\Bp}},
  {\beta_2}\circ {\tau_{\Bp}})$ for $\tau_{\Bp}$ a tree. 
\end{lemma}
\begin{proof}
 By \cite[\ref{m-equiv}]{Webmerged},
  the weight $(1,\dots, 1)$ subcategory is equivalent to a 
  block of parabolic category $\cO$ for the
  parabolic $\Bp$ containing the module corresponding to a tableau
  with entries given by $1,2,3,\dots, n$ under the rule given by
  Brundan and Kleshchev in \cite{BKSch}. Applying their formulas, this
  corresponds to a block containing a finite dimensional
  representation; that is, a regular block $\cO^{\Bp}_0$.  This is thus Koszul dual to singular category $\cO$ for
  the corresponding singularity by the main theorem of \cite{BGS96}.

The equivalence of \cite[\ref{m-equiv}]{Webmerged} sends the tree
bimodule $W_{\tau_\Bp}$ to the unique simple module killed by all translation
functors to walls.  This is, of course, the finite dimensional module
in $\cO^{\Bp}_0$.

  This Koszul duality of \cite{BGS96} by construction sends this
  finite dimensional simple (and thus the tree bimodule $W_{\tau_\Bp}$) to the unique
  self-dual projective of the dual singular block of category
  $\cO$. Furthermore, it sends any ladder
  functor to a translation functor, and a pure natural transformation
  to a natural transformation of projective functors (as functors on
  abelian categories).  A natural transformation between
  projective functors on category $\cO$ is 0 if and only if it is 0 when applied to the
  self-dual projective, since every other projective embeds in the sum
  of some number of these, which shows that $N=0$, as desired.
\end{proof}

\begin{lemma}
  For any tree $\tau_{\Bp}$, we have isomorphisms in the category 
$\dagger$ (that is, in $\tWB$ or $\WB$):
  \[\wE_i\circ {\tau_{\Bp}}\cong
  {\tau_{\Bp+\al_{i}}}^{\oplus [p_{i}]_q}\qquad \wF_i\circ {\tau_{\Bp}}\cong
  {\tau_{\Bp-\al_{i}}}^{\oplus [p_{i+1}]_q}.\]
\end{lemma}
\begin{proof}
  Clear from Propositions \ref{bigon} and \ref{associative}.
\end{proof}
Fix $i$; for any $\Bp$, we let \[\Bp^-=(p_1,\dots,
p_{i},1,p_{i+1}-1,\dots,p_\ell)\qquad \Bp^+=(p_1,\dots,
p_{i}-1,1,p_{i+1},\dots,p_\ell).\]  Note that $(\Bp\pm\alpha_i)^{\pm}=\Bp^{\mp}$. We have maps 
\[ G/P_{\Bp}\leftarrow G/P_{\Bp^-} \rightarrow
G/P_{\Bp+\al_{i}}\qquad  G/P_{\Bp}\leftarrow G/P_{\Bp^+} \rightarrow
G/P_{\Bp-\al_{i}}.\]
These maps induce a
$H^*(G/P_{\Bp})\operatorname{-}H^*(G/P_{\Bp\pm\al_{i}})$-bimodule
structure on $H^*(G/P_{\Bp^{\mp}})$.  
\begin{lemma}\label{lem:bim-iso}
  We have bimodule isomorphisms
  \[\Hom_{\dagger}({\tau_{\Bp+\al_{i}}},\wE_i\circ {\tau_{\Bp}})\cong
  H^*(G/P_{\Bp^-})\qquad
  \Hom_{\dagger} ({\tau_{\Bp-\al_{i}}},\wF_i\circ {\tau_{\Bp}})\cong
  H^*(G/P_{\Bp^+}).\]  
\end{lemma}
\begin{proof}
  We can write $\wE_i\cong Y_i^\star\circ  Y_{i+1}$.  Thus, by adjunction, we
  have that
  \[\Hom_{\dagger}({\tau_{\Bp+\al_{i-1}}},\wE_i\circ {\tau_{\Bp}})\cong
  \Hom_{\dagger}(Y_i\circ {\tau_{\Bp+\al_{i-1}}},Y_{i+1}\circ {\tau_{\Bp}}).\]  Both
  ladders that appear on the RHS
  are trees with top $\Bp^-$.  Thus, the result follows from
  Proposition \ref{tree-end}.  The isomorphism for $\wF_i$ follows by adjunction.
\end{proof}
Since these isomorphisms involve a choice of isomorphism between
$Y_i^\star$ and the adjoint of $Y_i$, they are not canonical.  We are
using the adjunction $\iota_{Y_i}$ chosen in Section
\ref{sec:case-a=1}, so
 that after identifying $\wE_i\circ {\tau_{\Bp}}$ with the tree
$\tau_{\Bp+\al_{i}}$ with a bigon blown up in the $i$th strand, the element $1\in H^*(G/P_{\Bp^{\pm}})$ corresponds to
$\iota_{Y_i}\circ \id_{\tau_{\Bp+\al_i}}$. 


\begin{lemma}
For any two ladders in $\beta_1,\beta_2$,
we have that \[\Hom_{\dagger} ({\tau_{\Bp''}},\beta_1\circ \beta_2 \circ{\tau_{\Bp}})\cong
\Hom_{\dagger}({\tau_{\Bp''}},\beta_1\circ {\tau_{\Bp'}})\otimes_{H^*(G/P_{\Bp'})} \Hom_{\dagger}
({\tau_{\Bp'}},\beta_2 \circ{\tau_{\Bp}}),\] where $\Bp$ is fixed, and
$\Bp',\Bp''$ are the top $\mathfrak{gl}_\ell$ weights of $\beta_2$ and $\beta_1$ respectively.
\end{lemma}
\begin{proof}
Let $\EuScript{T}$ be the dg-subcategory generated by the trees.  The
weight $\Bp$ part of $\EuScript{T}$ is generated by $\tau_{\Bp}$, and
thus is equivalent to the category of dg-modules over
$H^*(G/P_{\Bp})$.  
Composition with the ladder $Y_i$ or $Y_i^\star$ preserves $\EuScript{T}$ and thus
corresponds to a dg-bimodule over $H^*(G/P_{\Bp})$ and $H^*(G/P_{\Bp^\pm_i})$.  In fact, we know that the corresponding bimodules are
isomorphic to $H^*(G/P_{\Bp^\pm_i})$ with the trivial differential, since a tree is
sent to a sum of shifts of trees (rather than a complicated cone of
these).  These bimodules are free both as left or right modules (i.e.,
they are sweet).  Thus, the result follows from the fact that the dg-tensor product
of any monomial in these bimodules is the same as the naive tensor product.
\end{proof}

What we have now is that sending 
\[\Bp\mapsto \Hom_{\dagger} ({\tau_{\Bp}},{\tau_{\Bp}})\mmod \qquad
\eE_i\mapsto \Hom_{\dagger} ({\tau_{\Bp+\al_{i-1}}},\wE_i\circ {\tau_{\Bp}})\qquad
\eF_i\mapsto \Hom_{\dagger} ({\tau_{\Bp-\al_{i-1}}},\wF_i\circ {\tau_{\Bp}})\]
agrees on the level of 1-morphisms with the ``quiver flag''
construction of Khovanov and Lauda in \cite{KLIII}. More formally:
\begin{corollary}
  If $A$ is any monomial in $\wE_i$'s and $\wF_i$'s, then $\Hom_{\dagger}
({\tau_{\Bp'}},A\circ {\tau_{\Bp}})$ is canonically isomorphic to the
bimodule assigned to this monomial by the action of \cite[\S 6]{KLIII}.  
\end{corollary}
Now, we need only check that our construction agrees with Khovanov and
Lauda's on the level of 2-morphisms as well.  We'll note an important
relation for us, which is immediate from the choice of adjunctions we
fixed in Section \ref{sec:case-a=1}.
Let $\psi^\circ\colon \dE_i\dF_j\to
  \dF_j\dE_i$ be the adjoint of the map $\psi\colon \dF_j\dF_i\to
  \dF_i\dF_j$.  
\begin{lemma}
  The action of $y$ on $\dF_i$, of $\psi$ on $\dF_i^2$, of
  $\ep,\iota,\ep',\iota',$ and $\psi^\circ$ agrees with the formulas in \cite{KLIII}.
\end{lemma}
\begin{proof}
  The agreement of the action of $y$ is essentially the definition
  of the isomorphism of Proposition \ref{tree-end}; in both cases they
  are just multiplication operators.  
Similarly, the agreement of $\psi$ follows from the commutation
relations of \eqref{nilHecke-1}.  

For $j\neq i$, we have that $\dE_i\dF_jW_{\tau_{\Bp}}\cong
  \dF_j\dE_i W_{\tau_{\Bp}}$.  By definition, we have that $\psi^\circ$ is the unique
  isomorphism that induces the identity after we identify
  $\Ext(W_{\tau_{\Bp'}}, \dE_i\dF_jW_{\tau_{\Bp}})$ with
  $H^*(G/P_{\Bp''})$ where $\Bp''= (p_1,\dots,
  p_{i-1},p_i,1,p_{i+1}-1,\dots, p_j-1,1,p_{j+1},\dots,p_\ell)$, 
which
  agrees with Khovanov and Lauda's formulation.

Finally, we need to check that $\iota,\epsilon,\iota',\epsilon'$ all
induce the correct maps.  
We'll need to consider maps
$H^*(G/P_{\Bp^\pm})\otimes_{H^*(G/P_{\Bp\mp
    \al_i})}H^*(G/P_{\Bp^\pm})\longleftrightarrow H^*(G/P_{\Bp})$.
Note that the bimodule on the LHS is generated by $\ssy_i^k\otimes 1
$ for $k=1,\dots,\max(p_i,p_{i+1})$. 

First consider the special case where $p_i=0$.  In this case,
$\dF_i\dE_i$ is simply a bigon, and the sum of $p_{i+1}$ many copies
of the identity functor.  In Khovanov and Lauda's construction, we can
identify the bimodule for this functor with  $H^*(G/P_{\Bp^-})$,
thought of as a bimodule by the inclusion of $H^*(G/P_{\Bp})$ via
pullback.  In Khovanov and Lauda's construction, the map $\iota_i$ is
simply the pullback inclusion, sending $1\mapsto 1$, and the map
$\epsilon'_i$ is integration along the fibers (with appropriately chosen
relative orientation).  Thus, our $\iota_i$, which corresponds to $\iota_{Y_{i+1}}$ in this case, agrees with Khovanov and
Lauda's by the convention we have chosen for the isomorphism of  Lemma~\ref{lem:bim-iso}.

Furthermore, our map $\epsilon'_i\colon \dF_i\dE_i\to
\id$, which corresponds to $\ep_{Y_{i+1}^\star}$ in this case, is determined by the values of its compositions with
$(\ssy_i^{q}\otimes 1)\iota$ for $q\leq p_{i+1}-1$.
Lemma \ref{lem:bubble-dot} allows us to compute these compositions;
the result is 0 unless $q=p_{i+1}-1$
in which case it is 1.  The same relations hold in Khovanov and
Lauda's case by \cite[(3.4)]{KLIII} and the unnumbered equation below. 

The same argument shows that if $p_{i+1}=0$, we have agreement of
$\iota'_i$ and $\epsilon_i$ with Khovanov and Lauda's, using the other
case of Lemma \ref{lem:bubble-dot}.  By definition,
$\iota'_i=(-1)^{p_i-1}\iota_{Y_{i+1}}$, so Lemma  \ref{lem:bubble-dot}
shows that \[\ep_i(1\otimes \ssy_i^{q})\iota_i'=(-1)^{p_i-1}\ep_{Y_i^\star}(1\otimes \ssy_i^{q})\iota_{Y_i}=
\begin{cases}
  1 & q=p_{i}-1\\
  0 & q< p_i-1.
\end{cases}
\]

\excise{Similarly, consider $\iota\colon \id \to \dF_i\dE_i$; we can check
that two such maps agree if their composition with
$\epsilon',\epsilon'(\ssy_{i+1}\otimes 1), \dots,
\epsilon'(\ssy_{i+1}^{p_{i+1}-1}\otimes 1)$ agree.  As calculated
above, the results for both $\iota$ and Khovanov and Lauda's unit are
$0,\dots, 0,1$, by Lemma \ref{lem:bubble-dot} and \cite[(3.4)]{KLIII}
respectively.}

Now let $p_i=1$. Then $\iota'$ corresponds to $\iota_{Y_{i+1}^\star}$ and $\epsilon$ to $\ep_{Y_{i+1}}$ and the zig-zag relations of units and counits show that both are uniquely determined by $\iota$ and $\ep'$ for $p_i=0$. So also 
in this case, we have agreement between our maps and Khovanov and Lauda's 
$\iota',\epsilon$. Similarly, we have agreement with their maps for $\iota,\epsilon'$ when
$p_{i+1}=1$. 

In terms of ladders, every map of the form
$\iota,\epsilon,\iota',\epsilon'$ can be written as a composition of
these, since $\dE_i$ and $\dF_i$ are both compositions of two ladders
which involve splitting off a strand with label 1, and then joining a
new one on.  Thus, $\iota,\epsilon,\iota',\epsilon'$ can be written as
composition of the special cases discussed above.  Since the same is
true for Khovanov and Lauda's representation on the cohomology of flag
varieties, we are done.  
\end{proof}
By Lemma \ref{tree-enough}, we thus have that:
\begin{corollary}\label{cor:catskewfunc}
There is a 2-functor $\mathbf{a}\colon \tU_\ell\to \WB$ sending $\eE_i\mapsto
\wE_i,\eF_i\mapsto \wF_i$ and 2-morphisms to 2-morphism as indicated above.

Applying the natural representation of $\WB$, we find that  the functors $\dE_i$ and $\dF_i$ define a categorical action of
  $\mathfrak{gl}_\ell$ sending the weight $(p_1,\dots, p_\ell)$ to
  $D^b(T^{(p_1,\dots, p_\ell)}\mmod)$ which commutes with the action
  of $\tU_n$.
\end{corollary}
Combining this result with Theorem \ref{thm:ladder-GG} concludes the
proof of Theorem \ref{dualaction}.  

\subsection{Extension to \texorpdfstring{$\tWB$}{LB}}
\label{sec:extension-tildet}

As noted before, we have a natural 2-functor $\tWB\to \WB$.  We wish
to show that the 2-functor $\mathbf{a}\colon \tU_\ell\to \WB$ defined in the previous
section factors through this reduction. In order to show this, first
we'll prove the following lemma:
\begin{lemma}\label{lem:boot-strap}
    Consider $N\in \Hom_{\tWB}(\beta_1,\beta_2)$. Then $N$ is
  an isomorphism 
  if and only if  the induced natural transformation $\overline{N}\in
  \Hom_{\WB}(\beta_1,\beta_2)$ is an isomorphism.
\end{lemma}\begin{proof}
  It suffices to check that $N(P)$ is an isomorphism for any
  projective $P$.   By \cite[\ref{m-tilde-SS}]{Webmerged}, any
  projective module over $\tilde{T}^{\Bp} $ is filtered by standard
  modules, which are left inductions of standard modules pulled back
  from the cyclotomic quotients.  
 The natural transformation $N$ on
  these subquotients is an isomorphism by assumption. 
Thus, $P$ has a filtration
  such that $N(\operatorname{gr} P)$ is an isomorphism, which shows
  the same is true for $N(P)$.  
\end{proof}

For each index $i$, we can define $h_{k,i}$ to be the sum of all
diagrams with no crossings and exactly $k$ dots in total on the strands with label $i$ and no
dots on other strands (this is a generalization of the complete
symmetric polynomial).
\begin{lemma}\label{lem:center-polynomial}
  The center of $\tilde{T}^\Bp_\nu$ for $\nu=\sum\om_{p_i}-\sum
  n_j\al_j$ is the polynomial ring $\K[\{h_{k,i}\}_{k\leq n_j}]$.
\end{lemma}
\begin{proof}
  The proof directly follows \cite[2.9]{KLI}.  Obviously, these
  elements are central and  the basis theorem
  \cite[\ref{m-basis}]{Webmerged} for
  $\tilde{T}^\Bp_\nu$ shows that they freely generate a polynomial
  ring.  Let $e$ be
  the sum of straight-line diagrams where the red stands are all at
  the far left and the black strands all at the far right. We have that $e
  \tilde{T}^\Bp_\nu e\cong R(\sum n_i\al_i)$ again by
  \cite[\ref{m-basis}]{Webmerged}, and \cite[2.9]{KLI} shows that the
  center of the latter algebra is also a free polynomial ring $\K[\{eh_{k,i}e\}_{k\leq n_j}]$. Thus, if there is an element of the center of $Z(\tilde{T}^\Bp_\nu
  )$ not in $\K[\{h_{k,i}\}_{k\leq n_j}]$, then there would have to be
  a non-zero central element $z$ such that $eze=0$.  

 On the other hand, we have composition with the diagram sweeping
  all strands to the far right is injective (by \cite[\ref{m-basis}]{Webmerged}), so every projective
  embeds as a submodule of $\tilde{T}^\Bp_\nu e$.  But on any element
  $a\in \tilde{T}^\Bp_\nu e$, we have that $za=zaee=aeze=0$; the
  embedding above shows this is only possible if $z=0$.  Thus, we have
  arrived at a contradiction, and the $h_{k,i}$ with $k\leq n_j$ must
  freely generate the center.
\end{proof}

\begin{theorem}
There is a 2-functor $\tilde{\mathbf{a}}\colon \tU_\ell\to \tWB$ which
lifts $\mathbf{a}$.  That is,
the functors $\dE_i$ and $\dF_i$ define a categorical action of
  $\mathfrak{gl}_\ell$ sending the weight $(p_1,\cdots, p_\ell)$ to
  $D^b(\tilde{T}^{(p_1,\dots, p_\ell)}\mmod)$.
\end{theorem}
\begin{proof}
  We'll use Cautis's rigidity theorem to show that the relations hold
  in this case.  Consider $D^b(\tilde{T}^{(p_1,\dots, p_\ell)}\mmod)$
  as a graded category with the grading shift given by the Tate twist
  $\langle 1\rangle$.  We wish to show that this category carries a
  $(\mathfrak{g},\theta)$ action in the sense of \cite[\S
  2.2]{Caurigid}.  The basic data of such an action is given by the
  functors $\dE_i$ and $\dF_i$, and the map which Cautis denotes by
  $\theta$ is our $\By$; note that since our grading shift is the Tate
  twist, $\End^2(\mathbf{1}_\la)$ (in Cautis's notation) is the degree
  $-2$ part of the Hochschild cohomology
  $H\!H^2(\tilde{T}^{(p_1,\dots, p_\ell)})$, which is the target of
  $\By$.

  Cautis's conditions are all of the form that two functors are
  isomorphic, or that a given map is an isomorphism.  Thus, they are
  left unchanged by taking the associated graded, and will hold for
  $D^b(\tilde{T}^{(p_1,\dots, p_\ell)}\mmod)$ if and only if they hold
  for $D^b({T}^{(p_1,\dots, p_\ell)}\mmod)$.  Let us discuss them in
  more detail.

  Condition (i) is simply the fact that Hochschild cohomology of any algebra is trivial
  in negative degree, and the degree 0 part of the center of
  $\tilde{T}^{(p_1,\dots, p_\ell)}$ is the scalars by Lemma
  \ref{lem:center-polynomial}.  The conditions (ii-v) follow from
  noting that we can construct the desired morphisms using
  combinations of the 2-morphisms in our prospective categorical
  action.  While we cannot check the relations directly, Lemma
  \ref{lem:boot-strap} shows that the fact that they induce
  isomorphisms over $T^\Bp$ (by Theorem \ref{dualaction}) suffices to show the same result for
  $\tilde{T}^\Bp$.  Condition (vii) is vacuous for
  $\mathfrak{gl}_\ell$, and conditions (vi,viii) are direct
  consequences of the fact that $ \tilde{T}^\Bp\neq 0$ if and only if
  we have $0\leq p_i \leq n$.  Thus we satisfy all of Cautis's
  conditions and \cite[2.2]{Caurigid} implies that we have the desired
  categorical action.
\end{proof}

By Proposition \ref{prop:GG-iso} and Theorem \ref{thm:ladder-GG}, this
action categorifies the $\mathfrak{gl}_\ell$-action on $U^-_q\otimes
\iwedge{p}_q(\C_q^\ell\otimes \C_q^n)$ which acts by the identity on the
first factor.

\subsection{Foams}
\label{sec:foams}

Quefellec and Rose \cite{QR} define a foam 2-category $n
  \mathbf{Foam}(p)$, which is arguably a more natural object to
  consider in this case, and contains essentially the same information
  as the dual $\mathfrak{gl}_\ell$ action, without needing to fix a
  value of $\ell$.  The objects of this category are tuples (of any
  length) of integers from $[1,n]$ whose sum is $p$.  The 1-morphisms
  between $(p_1,\dots, p_k)$ and $(p_1',\dots, p_m')$
  are sums of ladders whose top and bottom are labeled by $\Bp'$ and $\Bp$
  respectively. Note that these are the same as the objects and
  1-morphisms of the 2-categories $\tWB,\WB$.  However, 2-morphisms
  are defined quite differently; they are explicitly given by singular cobordisms
  between these ladders modulo relations given in \cite[\S 3.1]{QR}. Recall that the facets of 
these foams can be decorated with symmetric polynomials.

{
\tikzset{external/export=true}

  \begin{proposition}
    There is a 2-functor of the 2-category $n
  \mathbf{Foam}(p)$ to $\tWB$ (and thus to $\WB$) 
which is the identity on objects and 1-morphisms.  This induces a
representation of $n
  \mathbf{Foam}(p)$ 
 which sends the sequence $(p_1,\dots, p_\ell)$ to the
  derived category $D^b(\tilde{T}^{\Bp}\mmod)$, a 1-morphism to the derived
  tensor product with the
  corresponding ladder bimodule, and each foam to a natural
  transformation between these functors.
  \end{proposition}
These natural transformations are uniquely determined by the fact that
they send the foamation of a 2-morphism in $\tU_{\ell}$ to its dual action as
we have defined it.  In particular:
\begin{itemize}
\item 
if $e_1$ denotes the first elementary symmetric polynomial in $p_i$ variables, then  
$$
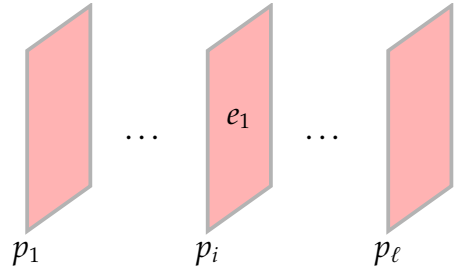
\begin{tikzpicture}[thick,scale=1.2, line width=.5mm]
\draw[fill=red,opacity=0.3] (0,0) -- (0.7,0.5) -- (0.7,2.5) -- (0,2) -- cycle;
\draw[fill=red,opacity=0.3] (2,0) -- (2.7,0.5) -- (2.7,2.5) -- (2,2) -- cycle;
\draw[fill=red,opacity=0.3] (4,0) -- (4.7,0.5) -- (4.7,2.5) -- (4,2) -- cycle;
\node at (1.3,1) {$\cdots$};
\node at (3.3,1) {$\cdots$};
\node at (2.35,1.25) {$e_1$};
\node at (0,-0.25) {$p_1$};
\node at (2,-0.25) {$p_i$};
\node at (4,-0.25) {$p_\ell$};
\end{tikzpicture}
$$
is sent to the action of the Hochschild
  homology class $\ssy_i$,
\item The singular seams $$
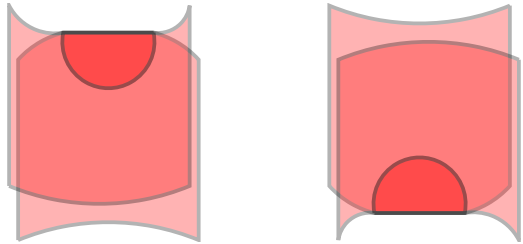
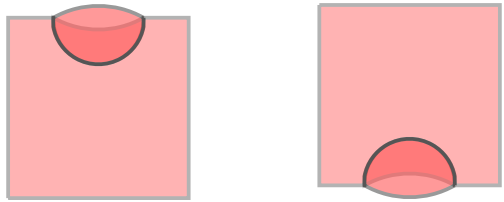
\begin{tikzpicture}[thick,scale=1.2, line width=.5mm]
\draw[fill=red, opacity=0.3,use Hobby shortcut] (0,0) .. (0.5,0.15) .. (1.5,0.15) .. (2,0) -- (2,2) .. (1.75,2.2) .. (1.5,2.3) -- (0.5,2.3)..(0.25,2.2)..(0,2)--(0,0);
\draw[fill=red, opacity=0.3,use Hobby shortcut] (-0.1,0.6) .. (0.4,0.45) .. (1.4,0.45) .. (1.9,0.6) -- (1.9,2.6) ..(1.8,2.4).. (1.5,2.3) -- (0.5,2.3)..(0.2,2.33)..(-0.1,2.6)--(-0.1,0.6);
\draw[fill=red, opacity=0.4,use Hobby shortcut] (1.5,2.3) ..(1.25,1.75)..(0.75,1.75)..(0.5,2.3)--(1.5,2.3);
\end{tikzpicture}
\qquad\qquad
\begin{tikzpicture}[thick,scale=1.2,xscale=-1,yscale=-1, line width=.5mm]
\draw[fill=red, opacity=0.3,use Hobby shortcut] (0,0) .. (0.5,0.15) .. (1.5,0.15) .. (2,0) -- (2,2) .. (1.75,2.2) .. (1.5,2.3) -- (0.5,2.3)..(0.25,2.2)..(0,2)--(0,0);
\draw[fill=red, opacity=0.3,use Hobby shortcut] (-0.1,0.6) .. (0.4,0.45) .. (1.4,0.45) .. (1.9,0.6) -- (1.9,2.6) ..(1.8,2.4).. (1.5,2.3) -- (0.5,2.3)..(0.2,2.33)..(-0.1,2.6)--(-0.1,0.6);
\draw[fill=red, opacity=0.4,use Hobby shortcut] (1.5,2.3) ..(1.25,1.75)..(0.75,1.75)..(0.5,2.3)--(1.5,2.3);
\end{tikzpicture}
$$
and
$$\begin{tikzpicture}[thick,scale=1.2, line width=.5mm]
\draw[fill=red, opacity=0.3,use Hobby shortcut] (0,0) -- (2,0) -- (2,2) -- (1.5,2) .. (1.25,1.55) .. (0.75,1.55) .. (0.5,2) -- (0,2)--(0,0);
\draw[fill=red, opacity=0.2,use Hobby shortcut] (1.5,2) -- (1.5,1.9).. (1.25,1.55) .. (0.75,1.55) .. (0.5,1.9)--(0.5,2) .. (0.75,1.9)..(1.25,1.9)..(1.5,2);
\draw[fill=red, opacity=0.4,use Hobby shortcut] (1.5,2) -- (1.5,1.9).. (1.25,1.55) .. (0.75,1.55) .. (0.5,1.9)--(0.5,2) .. (0.75,2.1)..(1.25,2.1)..(1.5,2);
\end{tikzpicture}
\qquad\qquad
\begin{tikzpicture}[thick,scale=1.2, yscale=-1, line width=.5mm]
\draw[fill=red, opacity=0.3,use Hobby shortcut] (0,0) -- (2,0) -- (2,2) -- (1.5,2) .. (1.25,1.55) .. (0.75,1.55) .. (0.5,2) -- (0,2)--(0,0);
\draw[fill=red, opacity=0.2,use Hobby shortcut] (1.5,2) -- (1.5,1.9).. (1.25,1.55) .. (0.75,1.55) .. (0.5,1.9)--(0.5,2) .. (0.75,1.9)..(1.25,1.9)..(1.5,2);
\draw[fill=red, opacity=0.4,use Hobby shortcut] (1.5,2) -- (1.5,1.9).. (1.25,1.55) .. (0.75,1.55) .. (0.5,1.9)--(0.5,2) .. (0.75,2.1)..(1.25,2.1)..(1.5,2);
\end{tikzpicture}
$$
are sent to the appropriate units
  and counits of adjunctions,
\item 
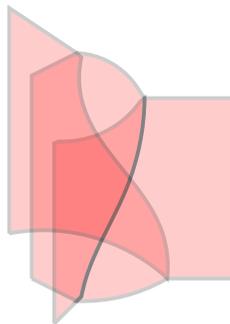
$$\begin{tikzpicture}[thick,scale=1.2, line width=.5mm]
\draw[fill=red, opacity=0.2,use Hobby shortcut] (0,0)..(0.25,0.25) --(0.3,0.3)..(0.35,0.5)..(0.8,1.5) ..(1,2.5)--(0.95,2.46) .. (0.5,2.1)..(0.05,2.02)--(0,2) --(0,0);
\draw[fill=red, opacity=0.2,use Hobby shortcut] (-0.25,0.5).. (0.25,0.25) --(0.3,0.3)..(0.35,0.5)..(0.8,1.5) ..(1,2.5)--(0.95,2.57) .. (0.5,2.95)..(0.25,3)--(-0.25,2.75) --(-0.25,0.5);
\draw[fill=red, opacity=0.2,use Hobby shortcut] (-0.5,1)..(1,0.7).. (1.2,0.55)--(1.25,0.5)..(1.2,1)..(0.3,2.5)..(0.3,2.95)--(0.25,3)..(-0.5,3.5)--(-0.5,1);
\draw[fill=red, opacity=0.2,use Hobby shortcut] (2,0.5)--(1.25,0.5)..(0.75,0.3)..(0.25,0.25)--(0.3,0.3)..(0.35,0.5)..(0.8,1.5) ..(1,2.5)--(2,2.5)--(2,0.5);
\end{tikzpicture}
$$
is sent to the associativity map of Proposition \ref{associative}.
\end{itemize}
}
\begin{proof}
First, we note that the functor $\tU_{\ell}\to \tWB$ kills any weight space corresponding to a sequence
$(p_1,\dots, p_\ell)$ with $p_i>n$ or $p_i<0$.  Thus, this action
factors through the quotient (following Queffelec and Rose's notation)
$\tU^{(n)}$.  Note that as we increase $\ell$,
these actions are compatible with the inclusion that adds $0$'s at the
start or end of the sequence.
 We can now use \cite[3.22]{QR} which writes the foam
  category as a direct limit of 
  $\tU_{\mathfrak{sl}_\ell}^{(n)}$ as $\ell\to \infty$, and thus shows
  that the desired 2-functor exists.
\end{proof}

\section{Knot homology}
\label{sec:knot-homology}

\nc{\tT}{\tilde{T}}
\nc{\tbra}{\tilde{\bra}}

In \cite{Webmerged}, the second author defined
equivalences of derived categories which correspond to the braid
action on tensor products of representations of quantum groups.  These
equivalences are induced by derived tensor product with a bimodule
$\tbra_i$ over $\tT^{\Bp}$ and $\tT^{s_i\Bp}$.

The reader can refer to \cite[\ref{m-bra-def}]{Webmerged}; the
bimodule is spanned by diagrams like Stendhal diagrams, but with a
single crossing between the $i$th and $i+1$st red strands, which satisfies
the relations: \begin{equation*}\label{side-dumb}\subeqn
    \begin{tikzpicture}
      [very thick,scale=1,baseline] \usetikzlibrary{decorations.pathreplacing}
      \draw[wei] (1,-1) -- (-1,1) node[at start,below]{$\la_{i+1}$}
      node[at end,above]{$\la_{i+1}$} node[midway,fill=white,circle]{};
      \draw[wei] (-1,-1) -- (1,1) node[at start,below]{$\la_{i}$}
      node[at end,above]{$\la_{i}$}; \draw (0,-1)
      to[out=135,in=-135] (0,1); \node at (3,0){=}; \draw[wei] (7,-1)
      -- (5,1) node[at start,below]{$\la_{i+1}$} node[at
      end,above]{$\la_{i+1}$} node[midway,fill=white,circle]{}; \draw[wei] (5,-1)
      -- (7,1) node[at start,below]{$\la_{i}$} node[at
      end,above]{$\la_{i}$}; \draw (6,-1) to[out=45,in=-45] (6,1);
    \end{tikzpicture}.
  \end{equation*}
  \begin{equation*}\label{top-dumb}\subeqn
    \begin{tikzpicture}
      [very thick,scale=1,baseline] \usetikzlibrary{decorations.pathreplacing}
      \draw[wei] (1,-1) -- (-1,1) node[at start,below]{$\la_{i+1}$}
      node[at end,above]{$\la_{i+1}$} node[midway,fill=white,circle]{};
      \draw[wei] (-1,-1) -- (1,1) node[at start,below]{$\la_{i}$}
      node[at end,above]{$\la_{i}$}; \draw (1.8,-1)
      to[out=145,in=-20] (-1,-.2) to[out=160,in=-80] (-1.8,1); \node
      at (3,0){=}; \draw[wei] (7,-1) -- (5,1) node[at
      start,below]{$\la_{i+1}$} node[at end,above]{$\la_{i+1}$}
      node[midway,fill=white,circle]{}; \draw[wei] (5,-1) -- (7,1) node[at
      start,below]{$\la_{i}$} node[at end,above]{$\la_{i}$}; \draw
      (7.8,-1) to[out=100,in=-20] (7,.2) to[out=160,in=-35] (4.2,1);
    \end{tikzpicture}.
  \end{equation*}
We grade this bimodule by giving this crossing the degree $-\langle
\la_{i+1},\la_{i}\rangle$.  Note that this is not an integer valued
grading, but rather a $\nicefrac{1}{n}\Z$-valued one.  However, for
fixed weights, the degrees of
all elements lie in the same class in $\nicefrac{1}{n}\Z/\Z$.  Let
$\xi(a,b)=(n-\max(a,b))\min(a,b)/n$.  Note
that $\langle
\omega_a,\omega_b\rangle=\xi(a,b)$. 

These functors induce an action of the braid group on the derived
categories $D^b(\tT^{\Bp})$, i.e. they satisfy
\begin{equation*}\label{Reidemeister2}\subeqn
    \begin{tikzpicture}
      [very thick,scale=1,baseline] \usetikzlibrary{decorations.pathreplacing}
      \draw[wei] (0,-1) .. controls (1,0) .. (0,1) node[at start,below]{$\la_{i}$}
      node[at end,above]{$\la_{i}$};
\draw[white,line width=7pt] (1,-1) .. controls (0,0) .. (1,1);
 \draw[wei] (1,-1) .. controls (0,0) .. (1,1) node[at start,below]{$\la_{i+1}$}
      node[at end,above]{$\la_{i+1}$};
\node at (2,0) {$\cong$};
 \draw[wei] (4,-1) .. controls (3,0) .. (4,1) node[at start,below]{$\la_{i+1}$}
      node[at end,above]{$\la_{i+1}$};
\draw[white,line width=7pt] (3,-1) .. controls (4,0) .. (3,1);
\draw[wei] (3,-1) .. controls (4,0) .. (3,1) node[at start,below]{$\la_{i}$}
      node[at end,above]{$\la_{i}$};
\node at (5,0) {$\cong$};
\draw[wei] (6,-1) -- (6,1)node[at start,below]{$\la_{i}$}
      node[at end,above]{$\la_{k-1}$}; 
\draw[wei] (7,-1) -- (7,1) node[at start,below]{$\la_{i+1}$}
      node[at end,above]{$\la_{i+1}$};
    \end{tikzpicture}
  \end{equation*} 
and
\begin{equation*}\label{triple-point-red}\subeqn
    \begin{tikzpicture}
      [very thick,scale=1,baseline] \usetikzlibrary{decorations.pathreplacing}
      \draw[wei] (1,-1) -- (-1,1) node[at start,below]{$\la_{i+2}$}
      node[at end,above]{$\la_{i+2}$};
 \draw[white,line width=7pt] (0,-1) .. controls (1,0) .. (0,1);
 \draw[wei] (0,-1) .. controls (1,0) .. (0,1) node[at start,below]{$\la_{i+1}$}
      node[at end,above]{$\la_{i+1}$};
  \draw[white,line width=7pt] (-1,-1) -- (1,1);
      \draw[wei] (-1,-1) -- (1,1) node[at start,below]{$\la_{i}$}
      node[at end,above]{$\la_{i}$}; 
\node
      at (3,0){$\cong$};       \draw[wei] (7,-1) -- (5,1) node[at start,below]{$\la_{i+2}$}
      node[at end,above]{$\la_{i+2}$};
 \draw[white,line width=7pt] (6,-1) .. controls (5,0) .. (6,1);
 \draw[wei] (6,-1) .. controls (5,0) .. (6,1) node[at start,below]{$\la_{i+1}$}
      node[at end,above]{$\la_{i+1}$};
  \draw[white,line width=7pt] (5,-1) -- (7,1);
      \draw[wei] (5,-1) -- (7,1) node[at start,below]{$\la_{i}$}
      node[at end,above]{$\la_{i}$}; 
    \end{tikzpicture}.
  \end{equation*}

Therefore, they lead to the construction of a knot
invariant. Now let us turn to comparing this action with the 
dual action of $\tU_{\ell}$.

\begin{proposition}\label{split-commute}
  We have an isomorphism $Y_{i+1}\Lotimes \tbra_i (-\xi(p_1,p_2))\cong \tbra_{i}\Lotimes
  \tbra_{i+1}\Lotimes Y_{i}(-\xi(a,p_2)-\xi(b,p_2))$.  In terms of ladders, we have that  
\[\tikz[baseline]{  
 \draw[wei] (1,-1) -- (-1,1) node[above, at end]{$p_2$} node[below, at start]{$p_2$}; 
\draw[white,line width=7pt] (-1,-1) -- (.2,.2);
      \draw[wei] (-1,-1) -- (.2,.2) node[below, at start]{$p_1$};  \draw[wei] (.5,1) -- (.2,.2)
      node[above right, at start]{$a$};  \draw[wei] (1,.5) -- (.2,.2)
      node[above right, at start]{$b$};  }(-\xi(p_1,p_2))\hspace{1mm}=\hspace{1mm}\tikz[baseline]{  
 \draw[wei] (1,-1) -- (-1,1) node[above, at end]{$p_2$} node[below, at start]{$p_2$}; 
\draw[white,line width=7pt] (.5,1) to[out=-135,in=90] (-.4,-.4);
      \draw[wei] (-1,-1) -- (-.4,-.4) node[below, at start]{$p_1$}; 
\draw[white,line width=7pt] (1,.5) to[out=-135,in=0] (-.4,-.4);
\draw[wei] (.5,1) to[out=-135,in=90]  node[above right, at start]{$a$} (-.4,-.4);  \draw[wei] (1,.5) to[out=-135,in=0]node[above right, at start]{$b$} (-.4,-.4);  }(-\xi(a,p_2)-\xi(b,p_2))\]
\end{proposition}
Note that while this is one isomorphism of bimodules, applying the
(anti-)automorphisms flipping diagrams through the $x$ or $y$ axes,
together with taking adjoints of functors shows that this relation
holds no matter how the diagram is rotated, and also if the positive
crossing is replaced with a negative one.  Thus, we can apply this
relation locally, no matter where in the diagram we find it.

\begin{proof}
If we consider these as underived tensor products then the
isomorphism between these spaces is clear: both spaces are spanned by
Stendhal diagrams with the red/red crossing and triple point added,
and the isomorphism simply shifts the triple point from one side of
the red strand to the other.  Thus, the result will follow if we can
show that any higher Tor vanishes in both cases.

  The tensor product $Y_{i+1}\Lotimes \tbra_i$ has no higher Tors since $Y_{i+1}$ is projective on
  the right. 

  We wish to also show that $\tbra_{i}\Lotimes
  \tbra_{i+1}\Lotimes Y_{i}$ has no higher Tors.  We can
  reduce to the case where there are no extraneous red strands
  (i.e. $\ell=2$ on the right and $\ell=3$ on the left).  Furthermore,
  using the bigon relation Proposition \ref{bigon} and the
  associativity Proposition \ref{associative}, we can reduce to the
  case where the sequence $(p_2,a,b)$ at the left has either $a=1$ or
  $b=1$.  Every projective module is filtered by modules which are
  left and right inductions of standard modules
  $\nabla_{(\mu,\omega_{p_2})}$ for various weights $\mu$. 
Thus, we need only check that the
  higher homology groups of $ \tbra_{i}\Lotimes \tbra_{i+1}\Lotimes
  Y_{i}\Lotimes\nabla_{(\mu,\omega_{p_2})}$ vanish.

The module $Y_{i}\Lotimes\nabla_{(\mu,\omega_{p_2})}\cong Y_{i}\otimes\nabla_{(\mu,\omega_{p_2})}$ is given by
$\fI_{\omega_{p_2}}\fF_{\omega_1-\mu}L_{a,b}$. Thus, if we replace
$L_{a,b}$ with its projective resolution $Q_\bullet$, then $\fI_{\omega_{p_2}}\fF_{\omega_1-\mu}Q_\bullet$ 
is a projective resolution of
$\fI_{\omega_{p_2}}\fF_{\omega_1-\mu}L_{a,b}$, because
$\fI_{\omega_{p_2}}$ and $\fF_{\omega_1-\mu}$ are both exact and send
projectives to projectives by Prop \ref{exactfunctors}. 

Recall that $\tbra_i\otimes\tbra_{i+1}\cong \tbra_{\sigma_i\sigma_{i+1}}$ as derived functors \cite[Lem. 6.11]{Webmerged}, so $\tbra_{i}\Lotimes \tbra_{i+1}\Lotimes  Y_{i}\otimes\nabla_{(\mu,\omega_{p_2})}$ is given by 
the cohomology of the complex 
$\tbra_{\sigma_i\sigma_{i+1}}\otimes \fI_{\omega_{p_2}}\fF_{\omega_1-\mu}Q_\bullet$.  
We claim that this complex only has non-zero cohomology in degree 0.   
If we use the basis for $\tbra_{\sigma_i\sigma_{i+1}}$ given in \cite[\ref{m-B-basis}]{Webmerged}, then we can
  write an element of the basis of any term in the resolution of 
 $\tbra_{\sigma_i\sigma_{i+1}}\otimes \fI_{\omega_{p_2}}\fF_{\omega_1-\mu}Q_\bullet$
uniquely as a sum of the products of:
  \begin{itemize}
  \item basis diagrams where all crossings and dots are on the strands
    coming from the terms in $Q_\bullet$, with all other strands
    remaining in place, times 
\item basis diagrams where the strands from $Q_\bullet$ have no crossings between them or dots (but they can
    cross the other strands).  This piece includes the red/red crossings.
  \end{itemize}
The differential in $Q_\bullet$ only changes the first type of diagram, so as a
complex of vector spaces, our resolution is isomorphic to $Q_\bullet$ tensored with a fixed vector space. Since $Q_\bullet$ only has cohomology in degree 0, the same is
true of this tensor product. This completes the proof. 
\end{proof}

Chuang and Rouquier \cite[\S 6.1]{CR04} define a complex of 1-functors called the {\bf
  Rickard complex} in the 2-quantum
group $\tU_{\mathfrak{sl}_2}$ which we'll denote $\Theta$.  This is
given by the sum of complexes of 1-morphisms:
\begin{align}
\label{eq:cpx'} \Theta \mathbf{1}_n &:= \left[ 
\mathsf{E}^{(-n)} \rightarrow \dots   \rightarrow  \mathsf{E}^{(-n+s)}
\mathsf{F}^{(s)} \langle -n +s \rangle \rightarrow \dots \right ]
\mathbf{1}_n \quad (n\leq 0)\\
\label{eq:cpx} \Theta \mathbf{1}_n &:= \left[ \mathsf{F}^{(n)} \rightarrow \dots \rightarrow \mathsf{F}^{(n+s)} \mathsf{E}^{(s)} \langle -n+s \rangle  \rightarrow\dots \right] \mathbf{1}_n \quad (n\geq 0)
\end{align}
\begin{remark}
  As often happens in mathematics, different papers use different
  conventions for these due to incompatible choices in the past.  Our
  conventions match those of Queffelec and Rose \cite[(2.42-43)]{QR},
  and are the opposite of those in Cautis \cite[(3-4)]{Cauclasp},
  which uses the adjoint of these complexes:
\begin{align}
\label{eq:cpxi'} \Theta^{-1} \mathbf{1}_n &:= \left[ \dots \rightarrow
  \mathsf{E}^{(-n+s)} \mathsf{F}^{(s)} \langle n -s \rangle
  \rightarrow \dots \rightarrow \mathsf{E}^{(-n)} \right]
\mathbf{1}_n \quad (n\leq 0)\\
\label{eq:cpxi} \Theta^{-1} \mathbf{1}_n &:= \left[ \dots \rightarrow \mathsf{F}^{(n+s)} \mathsf{E}^{(s)} \langle n-s \rangle \rightarrow \dots \rightarrow \mathsf{F}^{(n)} \right] \mathbf{1}_n \quad (n\geq 0)
\end{align}
\end{remark}
This
categorifies the quantum Weyl group inside $U_q(\mathfrak{sl}_2)$.  For any
higher rank group, we thus have a Rickard complex $\Theta_i$ for each
simple root given by the image of $\Theta$ under the associated root
embedding of $\tU_{\mathfrak{sl}_2}$.  Note that according to this definition, the
ladder bimodule $W_{|p_i-p_{i+1}|}$ is the degree 0 part of 
of $\Theta$ acting on $\tT^{\Bp}$-modules.

\begin{theorem}\label{same-crossing}
  There exists a quasi-isomorphism between the image of the
 Rickard complex $\Theta_i$ under the action on the derived
  categories $D^b(\tT^\Bp)$ and the braiding functor
  $\mathbb{B}_i(-\xi(p_i,p_{i+1}))$ switching the $i$th
  and $i+1$st strands.
\end{theorem}
\begin{proof}
First, we establish this in the case where $p_i=p_{i+1}=1$.  In this case, the complex $\Theta$ has two terms:
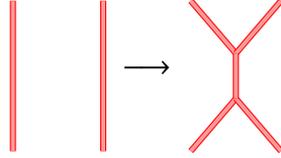
\[\tikz[baseline,xscale=.6]{\draw[wei](-1,-1) --(-1,1);
  \draw[wei](1,-1) --(1,1);}\hspace{1mm}\longrightarrow
\hspace{1mm}\tikz[baseline,xscale=.6]{\draw[wei](-1,-1) --(0,-.3);
  \draw[wei](1,-1) --(0,-.3);\draw[wei](-1,1) --(0,.3);
  \draw[wei](1,1) --(0,.3);\draw[wei](0,-.3) --(0,.3);}\]
The map is the unit of the adjunction between $Y_i$ and $Y_i^*$.  
Since this map commutes with the action of any element of the algebra,
we only need to consider its effect on an idempotent; in fact, every
diagram factors through idempotents where between two consecutive
red strands labeled 1, there is either a single black strand labeled 1 or no
black strands at all. The action of the map on these diagrams is given by
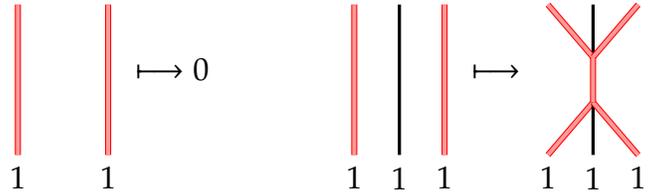
\[ \tikz[baseline,xscale=.6,very thick]{\draw[wei](-1,-1) -- node[below, at start]{$1$} (-1,1); \draw[wei](1,-1)
  -- node[below, at start]{$1$} (1,1);} \longmapsto 0\qquad\qquad \tikz[baseline,xscale=.6,very thick]{\draw[wei](-1,-1) -- node[below, at start]{$1$} (-1,1); 
  \draw (0,-1) -- node[below, at start]{$1$} (0,1); \draw[wei](1,-1)
  -- node[below, at start]{$1$} (1,1);} \longmapsto
\tikz[baseline,xscale=.6,very thick]{\draw[wei](-1,-1) -- node[below, at start]{$1$} (0,-.3); \draw (0,-1) -- node[below, at start]{$1$} (0,-.3);  \draw (0,1) -- (0,.3);
  \draw[wei](1,-1) --node[below, at start]{$1$} (0,-.3);\draw[wei](-1,1) --(0,.3);
  \draw[wei](1,1) --(0,.3);\draw[wei](0,-.3) --(0,.3);}.\]
This map is clearly surjective and we wish to show that its kernel is
$\tbra_i$. Here we include $\tbra_i\to \tT^{\Bp}$ sending 
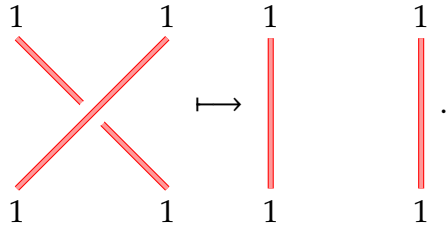
\[   \begin{tikzpicture}
      [very thick,scale=1,baseline] \usetikzlibrary{decorations.pathreplacing}
      \draw[wei] (1,-1) -- (-1,1) node[at start,below]{$1$}
      node[at end,above]{$1$} node[midway,fill=white,circle]{};
      \draw[wei] (-1,-1) -- (1,1) node[at start,below]{$1$}
      node[at end,above]{$1$};
    \end{tikzpicture}\longmapsto  \begin{tikzpicture}
      [very thick,scale=1,baseline] \usetikzlibrary{decorations.pathreplacing}
      \draw[wei] (1,-1) -- (1,1) node[at start,below]{$1$}
      node[at end,above]{$1$}; 
      \draw[wei] (-1,-1) -- (-1,1) node[at start,below]{$1$}
      node[at end,above]{$1$};
    \end{tikzpicture}.\] Note that this map increases degree by
    $\xi(1,1)=(n-1)/n$, so if it is an isomorphism, the theorem will
    be established with the correct shift.  

We only need to check this after it is
    applied to an arbitrary idempotent, i.e. sequence of black strands.  As argued above, using commutation with induction, we can reduce 
to proving this when the two $i$-labeled red strands are the only red strands and there are no black
strands, or one with label 1.  The exactness of the sequence $0\to
\tbra_i\to \tT^{\Bp}\to W_{\beta_i}$ is clear in the former case, and
proven in the latter case in \cite[4.11]{WebTGK}. 

\nc{\Cone}{\mathsf{Cone}}

The above short exact sequence gives rise to a quasi-isomorphism of complexes 
$$
\tbra_i \to \Cone(\tT^{\Bp}\to W_{\beta_i}).  
$$
This proves that in the case $p_i=p_{i+1}=1$, the derived tensor product with $\tbra_i$ and $\Cone(\tT^{\Bp}\to W_{\beta_i})$ 
gives rise to functors on $D^b(\tT^\Bp)$ which are isomorphic up to the shift indicated in the theorem.

Now we turn to the general case. To make the pictures easier to draw, we first prove the case with $p_i$ arbitrary and $p_{i+1}=1$, by induction on $p_i$. So let $p_i=a>1$. Since we aim to just prove the existence of an isomorphism, it suffices to prove this result after replacing both sides by the same number of summands of each functor. Thus, we multiply both functors by $[a]$. By the bigon relation Proposition \ref{bigon} this is the same as creating an $(a-1,1)$-bigon on the red $a$-strand at the bottom of the braid diagram. Then we use Proposition~\ref{split-commute} and its analogue for $\Theta_i$, given by
\cite[5.2]{Cauclasp}, 
\[ [a]  \begin{tikzpicture}
      [very thick,scale=1,baseline] \usetikzlibrary{decorations.pathreplacing}
      \draw[wei] (1,-1) -- (-1,1) node[at start,below]{$1$}
      node[at end,above]{$1$} node[midway,fill=white,circle]{};
      \draw[wei] (-1,-1) -- (1,1) node[at start,below]{$a$}
      node[at end,above]{$a$};
    \end{tikzpicture}\cong  
\begin{tikzpicture}
      [very thick,scale=1,baseline] \usetikzlibrary{decorations.pathreplacing}
      \draw[wei] (1,-1) -- (-1,1) node[at start,below]{$1$}
      node[at end,above]{$1$} node[midway,fill=white,circle]{};
      \draw[wei] (-1,-1) -- (-0.75,-0.75) node[at start,below]{$a$};
\draw[wei] (-0.25,-0.25) -- (1,1) node[at end,above]{$a$};
\draw[wei] (-0.75,-0.75) .. controls (-0.75,-0.1) .. (-0.25,-0.25);
\draw[wei] (-0.75,-0.75) .. controls (-0.25,-0.9) .. (-0.25,-0.25);
\node at (-0.75,0.1) {$a-1$};
\node at (-0.25,-1.1) {$1$};
    \end{tikzpicture}
\cong  
\begin{tikzpicture}
      [very thick,scale=1,baseline] \usetikzlibrary{decorations.pathreplacing}
      \draw[wei] (1,-1) -- (-1,1) node[at start,below]{$1$}
      node[at end,above]{$1$}; 
\draw[white,line width=5pt] (-0.4,0.4) -- (-0.5,0.5);
\draw[white,line width=5pt] (-0.25,0.25) -- (-0.35,0.35);
\draw[white,line width=5pt] (0.4,-0.4) -- (0.5,-0.5);
\draw[white,line width=5pt] (0.25,-0.25) -- (0.35,-0.35);
      \draw[wei] (-1,-1) -- (-0.5,-0.5) node[at start,below]{$a$};
\draw[wei] (0.5,0.5) -- (1,1) node[at end,above]{$a$};
\draw[wei] (-0.5,-0.5) .. controls (-0.5,0.5) .. (0.5,0.5);
\draw[wei] (-0.5,-0.5) .. controls (0.5,-0.5) .. (0.5,0.5);
\node at (0.1,0.9) {$a-1$};
\node at (0.9,0) {$1$};
    \end{tikzpicture}
\] 
and use induction. The general case now follows from this one by induction on $p_{i+1}$, using similar arguments with a 
bigon on the other red strand.

\excise{
Thus,
we multiply each of these functors by $[p_i]_q![p_{i+1}]_q!$. By the bigon
relation Proposition \ref{bigon} this is the same as ``puffing'' each strand with label $r$ to $r$ strands with label 1
at the bottom of the diagram.  Using Proposition
\ref{split-commute} and its analogue for $\Theta_i$, given by
\cite[5.2]{Cauclasp}, we can isotope so this diagram so that the
crossings all lie between the ``puffed'' part of each strand. The shift in Proposition
\ref{split-commute} shows that we now desire an isomorphism where we
shift the LHS by $p_ip_{i+1}\xi(1,1)$. Furthermore, we now have only
crossings where both labels are 1.  Thus, we can use the
$p_i=p_{i+1}=1$ case to complete the proof.
\bentodo{Worth making a picture?  I think the idea is clear to us, but
maybe it wouldn't be to a reader?}}

\end{proof}
\begin{remark}\label{hmmm}
  We should note a slightly confusing point here.  The functor $\mathbb{B}_i$ is
  right exact, and thus preserves the bounded {\it above} derived category.
  The functor $\Theta_i$ is conventionally presented as a bounded
  {\it below} complex of 1-morphisms.  But remember that the dual
  $\mathfrak{gl}_\ell$-action is very much not by exact functors: it
  instead is compatible with the $t$-structure whose heart is linear
  complexes of projectives.
\end{remark}

It's worth noting, though, that the braiding functor $\mathcal{T}_i =\Theta_i (\min(a,b))$
considered in \cite{Cauclasp,QR} is a grading shift of $\Theta_i$.  Let
\[\xi'(a,b)=\frac{\min(a,b)\cdot \max(a,b)}{n}.\]  Note that $\xi'(a,b)=\min(a,b)-\xi(a,b)$.
Accounting for the grading shift in \cite[(4.1-2)]{QR}, we have that: 
\begin{corollary}
  The complex of functors $\mathcal{T}_i=\Theta_i(\min(a,b))$ acts on $\tT^{\Bp}\mmod$
  by $\mathbb{B}_i(\xi'(a,b))$.  
\end{corollary}

There is another comparison result we'll require. In \cite[\S
5,\S 11]{Cauclasp}, Cautis constructs idempotent complexes of functors
$\mathsf{P}^{\pm}$.   These complexes are constructed as a limit of powers of
Rickard complexes.  The element $(\Theta_1\cdots
\Theta_{\ell-1})^{\ell-1}$ corresponds to the action of the full twist
braid in the quantum Weyl group action;  Cautis shows, building on
ideas of Rozansky \cite{Roztorus}, that the functors $(\Theta_1\cdots
\Theta_{\ell-1})^{\pm p(\ell-1)}$ stabilize  to an
idempotent functor $\mathsf{P}^{\pm}$ in appropriate sense as $p\to \infty$.

We wish to understand how this complex acts in the
dual action we have constructed.  As in \cite{WebTGK}, we consider the
subcategory $J_{\Bp}^\pm$ of $D^\pm(\tilde{T}^{\Bp}\mmod)$ generated by projectives $\tilde{T}^{\Bp} e_{\Bi,\kappa}$ with
$\kappa$ constant.  That is, in terms of the sequence of red and black
terminals, we have a group of black terminals, a group consisting of
all red terminals, and then another black group.  Let $\tilde P^0$ be
the sum of $\tilde{T}^{\Bp} e_{\Bi,\kappa}$ over all such $\kappa$.  By \cite[\ref{m-split-strands}]{Webmerged}, we
have an isomorphism $\tilde{T}^{\la_{\Bp}}\cong \End(\tilde P^0)$
where $\la_{\Bp}=\sum\omega_{p_i}$. Thus, we have an
equivalence $J_{\Bp}^\pm\cong D^\pm (\tilde{T}^{\la_{\Bp}}\mmod)$.  The
inclusion $\iota^\pm\colon J_{\Bp}^\pm\to D^\pm (\tilde{T}^{\Bp}\mmod)$ has an exact
right adjoint $\iota_*$,
which we can interpret as the functor $\Hom(\tilde P^0,-)$, which satisfies
$\iota_*^\pm\iota ^\pm\cong \operatorname{id}_{J_{\Bp}}$.

\begin{proposition}\label{projection}
  The complex $\mathsf{P}^{\pm}$ is well-defined on the category
  $D^{\pm}(\tilde {T}^{\Bp}\mmod)$ and acts in the dual action as
  the projection $\iota ^\pm\iota_*^\pm$.
\end{proposition}
\begin{proof}
First, we note that $\mathsf{P}^{\pm}$ preserves the category
$D^{\pm}(\tilde {T}^{\Bp}\mmod)$, since the functors
$\Theta_i=\mathbb{B}_i(-\xi(a,b))$ (resp. $\Theta_i^{-1}=\mathbb{B}_i^{-1}(\xi(a,b))$) are
right (resp. left) exact.

  Let us specialize to the case of $\mathsf{P}^-$; the case of
  $\mathsf{P}^+$ is proved similarly.   Note that we have an
  orthogonal decomposition of the category into $J_{\Bp}^-$ and the
  kernel of $\iota_*^-$, which is generated by the simple modules whose projective
  cover is not a summand of $P_0$. 
Thus, to establish the equality $\iota ^-\iota_*^-\cong \mathsf{P}^{-}$, we must show that $\mathsf{P}^-$ acts by the
  identity on $J_{\Bp}^-$ and kills any simple module whose projective
  cover is not a summand of $P_0$. 

Proving the first fact is easy. From the
  commutation of braiding and induction functors, it suffices to check
  that $\mathsf{P}^-$ acts by the identity on $\tT^{\Bp}e_{\emptyset}$,
  the subalgebra with no black strands. This is clear from the fact
  that a full twist does so.

For the second fact, let us first consider the case where $L$ is a simple factoring through $
  {T}^{\Bp}$; it then suffices to check this result using braiding
  functors over ${T}^{\Bp}$.  If we let $\tau$ denote a half twist,
  then the bimodule associated to the bimodule $\tbra_\tau$ is tilting as a left or right
  module by \cite[\ref{m-tilting}]{Webmerged}.  In particular, for a
  simple $L$, we have that $\Hom(\tbra_\tau,L)=0$ unless $L$ is a
  quotient of a tilting module.  Put differently, the module
  $\mathbb{B}_\tau^{-1} (L)$ has cohomology concentrated in degrees
  $\geq 1$ unless $L$ is a quotient of a tilting
  module. Since any
  tilting has a costandard filtration, any tilting receives a
  surjective map from the projective covers of some set of costandard
  modules. By 
  \cite[\ref{m-self-dual}]{Webmerged}, the projective cover of any
  costandard (which is dual to the injective hull of a standard) is
  injective and thus a
  sum of summands of $P^0$.  

That is, the module $\mathbb{B}_\tau^{-1} (L)$ has cohomology concentrated in degrees
  $\geq 1$ unless $L$ is a quotient of $P^0$.  

Note that \[\R\Hom(P^0,\mathbb{B}_\tau^{-1}(L))\cong
\R\Hom(\mathbb{B}_\tau (P^0),L)\cong \R\Hom(P^0,L)=0;\] here we have
abused notation and used $P^0$ to denote the corresponding module both
over $\tT^{\Bp}$ and $\tT^{\tau\cdot \Bp}$. This shows that
$\mathbb{B}_\tau^{-1} $ preserves the
kernel of $\iota_*^-$.  
Thus, none of its cohomology groups have composition factors whose projective
  cover is a summand of $P_0$. It follows that
  $\mathbb{B}_\tau^{-2}(L)$ must be concentrated in degrees $\geq 2$,
  and have the same restriction on its composition factors.
 
Applying this argument inductively, the $n$-fold full twist $\mathbb{B}_\tau^{-2n}(L)$
  is concentrated in degrees $\geq 2n$, so this complex vanishes in
  the limit.  

Every other simple of $\tilde {T}^{\Bp}$ which is not a quotient of $\tilde P^0$ is
a quotient of $\tilde\fF_i^*L'$ where $L'$ is another simple which is not
a quotient of $\tilde P^0$.  By induction on the number of black
strands, we can assume that $\mathsf{P}^-$ kills $L'$, and thus
$\tilde\fF_i^*L'$.  The exactness of $\iota_*^-$ shows that it also
kills $L$.  
\end{proof}
Using induction, we can extend this result further.  Let
$\bla=(\la_1,\dots, \la_m)$ be a sequence of dominant weights for
$\mathfrak{gl}_{n}$ of total level $2w$.  Then, we can choose a
sequence of fundamental weights $(p_1,\dots,p_{2w})$ such that the
first $w_1$ add to $\la_1$, the next $w_2$ add to $\la_2$, etc.  We
can define $\mathsf{P}^-_\bla$ to be the idempotent sequence of
functors given by horizontally composing $\mathsf{P}^-$ acting on the
first $w_1$ red strands, then on the next $w_2$, etc.  Just as before,
we can define $J_{\bla,\Bp}$ to be the subcategory generated by
projectives with $\kappa$ constant on $[1,w_1]$, on $[w_1+1,w_1+w_2]$,
etc. and $\iota^{\bla},\iota^{\bla}_*$ the corresponding inclusion and
its right adjoint.  As before, we have an equivalence
$J_{\bla,\Bp}\cong \tilde{T}^\bla\mmod$ by   \cite[\ref{m-split-strands}]{Webmerged}.

For any braid $\sigma$ on $m$ strands, we can arrive at an associated
braid $\sigma_w$ on
$2w$ strands by taking the cabling with the blackboard framing where
the $i$th strand from the left at the top of the braid into $w_i$
strands.   The arguments of Theorem \ref{projection} show that: 
\begin{corollary}
  The idempotent $\mathsf{P}^\pm_\bla$ is isomorphic to
  $\iota^{\bla,\pm}\iota^{\bla,\pm}_*$. \hfill\qed
\end{corollary}
Furthermore, applying the argument of
\cite[\ref{m-split-strands}]{Webmerged} to the braiding bimodules
shows that:
\begin{corollary}\label{cable-braiding}
  The action of the cabled braiding
  functor
  $\mathsf{P}^{\pm}_{\sigma\cdot\bla}\mathbb{B}_{\sigma_w}\mathsf{P}^{\pm}_\bla$
  matches that of $\mathbb{B}_{\sigma}$ under the equivalence
  $J_{\bla,\Bp}^{\mp}\cong D^{\mp}(\tilde{T}^\bla\mmod)$. \hfill\qed
\end{corollary}

While we have proved the theorems above for the category
$\tilde{T}^{\Bp}$-mod, the same results all hold for the category
of $T^\Bp$-mod embedded via pullback.  In particular, Theorem
\ref{same-crossing} and Corollary \ref{cable-braiding} still hold for
the braiding bimodule $\bra_{i}\cong T^\Bp\otimes_{\tT^\Bp}\tbra_i$
and the image of the Rouquier complex under the action $\mathbf{a}$.  

\begin{remark}
  Just so the reader does not think we are leaving a stone unturned:
  one can carry out the knot homology construction of
  \cite[Sec.~\ref{m-sec:invariants}]{Webmerged} 
  over $\tilde{T}$.  However, it does not result in anything new from
  the perspective of knot homology: the functors attached to a tangle
  using $T$ are those constructed using $\tilde{T}$ when restricted to
  the pullback of $T$-modules.  In particular, associated to a closed
  link, we have the functor of tensor product with a vector space,
  which remains tensor product with the same vector space when
  restricted to $T$-modules. 
\end{remark}

This allows us to complete the proof of Theorem \ref{thmA}.   We wish
to consider link homologies attached to a link labeled with a
representation of $\mathfrak{sl}_n$. 
 Each such representation has a level, given by $\sum m_i$ if
 $\la=\sum m_i\omega_i$.  For a given labeled link projection, let
 $\ell=\ell(L)$ be the sum of the levels of the representations
 labeling the minima of the projection.

In \cite{Cauclasp}, Cautis defines a knot homology constructed in any
additive categorification of the representation of
$\mathfrak{sl}_{2\ell}$ with highest weight $n\omega_{\ell}$.  These
categorifications can be classified in terms by the data of the
highest weight category which is additively generated by a unique
indecomposable object $C$, and an action of a polynomial ring
$\K[e_1,\dots, e_{n}]$ on this object, given by the fake bubbles in
the endomorphisms of the identity 1-morphism on $n\omega_{\ell}$.  Up to equivalence, this is
controlled by the ring $\End(C)$ and the associated homomorphism
$S=\K[e_1,\dots, e_{n}] \to \End(C)$.  This polynomial ring is graded
with $\deg (e_i)=2i$.  The most important example of such a
categorification is given by the cyclotomic quotient for
$n\omega_{\ell}$, which corresponds to the case where $\End(C)\cong
\K$ with the map sending $e_i\mapsto 0$. 

In the construction of \cite{Cauclasp}, each
link $L$ labeled with representations of $\mathfrak{sl}_n$ is translated
into a complex $\mathcal{Z}(L)$ of
1-morphisms in $\tU_{\mathfrak{sl}_{2\ell}}$.
\begin{itemize}
\item For cups and caps, this follows the ladder and ladder formalism
  introduced earlier in this paper.  
\item A crossing is associated to the
  corresponding shifted Rickard complex $\mathcal{T}_i^{-1}=\Theta_i^{-1}(-\min(a,b))$
\item Finally, non-fundamental weights are
  dealt with using cabling and inserted a copy of the projector $
  \mathsf{P}^+$.
\end{itemize}
Queffelec and Rose follow the same recipe in the language of foams;
their homology and ours differ from Cautis's by taking the mirror.
The result is a complex in the category of graded vector spaces.  We
think of this as a bigraded vector space $C=\oplus_{i,j}C^{i,j}$ where $i$ is the homological
grading, and $j$ the internal grading.

If one has a categorical module over $\mathfrak{sl}_{2\ell}$,
and a highest weight object $C$ of weight $n\omega_{\ell}$, the result
$\Hom(C,\mathcal{Z}(L)C)$ is an invariant of the labeled link $L$, and
the $S$-algebra $\End(C)$.  This allows one to identify knot homology
theories categorifying Reshetikhin-Turaev invariants with seemingly
disparate origins.

Let $\EuScript{W}^{i,j}$ be the  knot invariants attached to labeled
knots by the second author in
type A in \cite{Webmerged}, $\EuScript{QR}^{i,j}$ those produced by
Queffelec and Rose \cite{QR}, and $\EuScript{C}^{i,j}_{\pm}$ those produced
by Cautis.  Consider an oriented framed link, and let the {\bf weighted writhe}
$\operatorname{wwr}(L)$ be the sum over the postive/negative crossings
in any diagram of $L$ of $\pm 
\xi'(a,b)$ where $a$ and $b$ are the labels on the crossing.   
\begin{theorem}\label{compare1}
 \[\EuScript{W}^{i,j+\operatorname{wwr}(L)}(L)\cong \EuScript{QR}^{i+j,-j}(L)\cong \EuScript{C}^{i+j,-j}_{+}(L^!) \cong \EuScript{C}^{-i-j,+j}_{-}(L)\]
\end{theorem}
\begin{proof}
  It's clear from Proposition \ref{projection} and Corollary
  \ref{cable-braiding} that our functors coincide with Cautis's for
  the mirror up to grading shift, and the reindex of the grading.
  That is, for some invariant $r(L)$ of a framed link,
  $\EuScript{W}^{i,j+r(L)}(L)\cong \EuScript{QR}^{i+j,-j}(L)$.
  
  Note that by Theorem \ref{same-crossing}, when we switch a crossing
  from positive to negative, $r(L)$ must decrease by $2\xi(a,b)$.
  Since $\operatorname{wwr}(L)$ has the same property, 
  this allows us to reduce to the case of an unknot. 

If this unknot is 0-framed, then the theorem simply asserts that the
homologies are the same.  Indeed,
\[\sum_{i,j}t^iq^j\dim \EuScript{QR}^{i,j}(L)=\binom{n}{p}_q\qquad \sum_{i,j}t^iq^j\dim \EuScript{W}^{i+j,-j}(L)=\binom{n}{p}_q
\]
so the result follows in this case.

 Now, note that
a strand with a curl in it must correspond to a shift functor, so, we
need only check that adding some nonzero power of it gives the same
shift on both sides.  However, if we do an even number of curls, we
can get back to the 0-framed unknot by reversing crossings, so this
follows from our reduction to the unknot.  
\end{proof}
For completeness, let us note that these theories are already known to
coincide with the other known examples of categorifications of the
$\mathfrak{sl}_n$ Reshetikhin-Turaev invariants.
\begin{theorem}
  The following link homology theories for a labelling by
  representations of
  $\mathfrak{sl}_n$ are given by
  $\Hom(C,\mathcal{Z}(L)C)$ for a categorical action with $C$ a highest
  weight object of weight $n\omega_{\ell}$ with endomorphism $S$-algebra $\End(C)\cong
\K$:
\begin{enumerate}
\item Khovanov-Rozansky homology \cite{KR04} (and its colored
  generalization \cite{Wu,Yon}),
\item Cautis-Kamnitzer homology \cite{CKII},
\item the foam-based invariants of Queffelec and Rose \cite{QR},
\item the Mazorchuk-Stroppel-Sussan homology constructed from category
  $\cO$ \cite{MS09,Sussan2007} and
\item the invariants $\mathcal{K}$ constructed in \cite{Webmerged}.
\end{enumerate}
Thus, all these knot homologies coincide.
\end{theorem}
The items (1-3) are discussed in \cite{LQR,QR,Caurigid}; the new
contribution of this paper is to add items (4-5) to this list.
\begin{proof}\mbox{}
  \begin{enumerate}
  \item For Khovanov-Rozansky homology, this is based on the
    construction of a categorical action constructed in \cite{MY} on
    matrix factorizations.  This comparison is discussed in greater
    detail in \cite[\S 4]{QR}.
\item For  Cautis-Kamnitzer homology, this is essentially by the
  original definition given in \cite{CKII}; however, since it proved
  difficult to confirm all the relations between 2-morphisms,
this required a strictification result \cite[14.9]{Caurigid}.
\item By construction, Queffelec and Rose's invariant is the image
  under the foamation 2-functor \cite[\S 3.2]{QR} of $\mathcal{Z}(L)$.
\item The papers \cite{MS09} and \cite{Sussan2007} use different sides
  of Koszul duality; in the former
  case, the categorical action by translation functors is used, and in
  the latter, that by Zuckerman functors.  While the comparison
  between the  complex $\mathcal{Z}(L)$ and the appropriate functors
  on category $\cO$'s can be done directly, we will take the short-cut
  of noting that we have an equivalence between the category $\cO$
  picture and that of (5) by \cite[\ref{m-MSS-compare}]{Webmerged}.
\item This follows from Theorem \ref{compare1}.
  \end{enumerate}
This completes the proof of Theorem \ref{thmA}.
\end{proof}

 \bibliography{../gen}
\bibliographystyle{amsalpha}

\end{document}